\documentclass[arma,numbook,envcountsame,final]{svjour}
\usepackage{amssymb}
\usepackage{times}

\title{Regularity for Double Phase Problems at \\ Nearly Linear Growth}

\author{Cristiana De Filippis \& Giuseppe Mingione}

\usepackage{verbatim}
\usepackage[T1]{fontenc}
\usepackage[notref,notcite]{showkeys}
\usepackage{graphicx}
\usepackage{enumerate}
\usepackage{amsmath,amsfonts,amssymb}
\usepackage{color}
\usepackage{cite}
\definecolor{citation}{rgb}{0.2,0.58,0.2} 
\definecolor{formula}{rgb}{0.1,0.2,0.6}
\definecolor{url}{rgb}{0.3,0,0.5}
\usepackage{pgf,tikz}
\usepackage{mathrsfs}
\usepackage{fancyhdr}
\usepackage{dutchcal}
\usepackage{dsfont}

\newcommand{\reqnomode}{\tagsleft@false}

\usepackage[colorlinks,pdfpagelabels,pdfstartview = FitH,bookmarksopen = true,bookmarksnumbered = true,urlcolor=url,linkcolor = formula,plainpages = false,hypertexnames = false,citecolor = citation] {hyperref}

\def\dx{\, dx}
\def\dy{\, dy}
\def \d{\, d }
\def \diver{\,{\rm div}}
\def\eh{\mathbb H}
\def\dist{\,{\rm dist}}

\newcommand{\ao}{\alpha_0}
\newcommand{\tia}{\tilde \alpha}

\def\supp{\,{\rm supp}}

\newcommand\loc{{\rm loc}}

\newcommand\data{\textnormal{\texttt{data}}}

\newcommand\ssss{\mathfrak{s}}

\newcommand{\ti}[1]{\tilde{#1}}

\newcommand\ccc{\mathfrak{c}}

\newcommand{\bb}{\textnormal{\texttt{s}}}

\newcommand\mum{\mu_{\textnormal{max}}}

\def\ds{\,{\rm d}s}

\allowdisplaybreaks
\makeatletter
\DeclareRobustCommand*{\bfseries}{%
  \not@math@alphabet\bfseries\mathbf
  \fontseries\bfdefault\selectfont
  \boldmath
}

\makeatother

\newlength{\defbaselineskip}
\newcommand{\mint}{\mathop{\int\hskip -1,05em -\, \!\!\!}\nolimits}

\numberwithin{equation}{section}
\allowdisplaybreaks[4]

\newcommand{\kk}{\kappa}

\def\en{\mathbb N}
\def\er{\mathbb R}

\newcommand\eps\varepsilon

\def\eqn#1$$#2$${\begin{equation}\label#1#2\end{equation}}

\newcommand{\be}{\begin{equation}}
\newcommand{\ee}{\end{equation}}

\newcommand{\rr}{\varrho}

\newcommand{\snr}[1]{\lvert #1\rvert}

\newcommand{\nr}[1]{\lVert #1 \rVert}

\def\er{\mathbb R}

\newcommand{\N}{\mathbb{N}}

\def\name[#1, #2]{#1 #2}

\newcommand{\mm}{\mathfrak{m}}

\def \Hi{H_{\textnormal{i}}}
\def \tg{\tilde g}
\newcommand{\BB}{\mathcal{B}_1}
\def \M{\mathfrak{M}}

\newcommand{\rif}[1]{(\ref{#1})}
\newcommand{\trif}[1] {\textnormal{\rif{#1}}}

\newcommand{\stackleq}[1]{\stackrel{\rif{#1}}{ \leq}}

\allowdisplaybreaks

\begin{document}


\maketitle

\vspace{3mm}
\begin{abstract}
Minima of functionals of the type
$$ w\mapsto \int_{\Omega}\left[\snr{Dw}\log(1+\snr{Dw})+a(x)\snr{Dw}^{q}\right] \dx\,, \quad 0\leq a(\cdot) \in C^{0, \alpha}\,,$$
with $\Omega \subset \er^n$, have locally H\"older continuous gradient provided $1 < q < 1+\alpha/n$. 
\end{abstract}


\tableofcontents

\vspace{3mm}
\setcounter{tocdepth}{1}
\section{Introduction}\label{intro}
In this paper we prove the first Schauder type results for minima of nonuniformly elliptic integrals with nearly linear growth, and under optimal bounds on the rate of nonuniform ellipticity. This is a non-trivial task, as in nonuniformly elliptic problems Schauder type estimates change their nature and cannot be achieved via perturbation methods. This was first shown by the counterexamples in \cite{sharp, FMM}.  
A relevant model functional we have in mind is
\eqn{modello}
$$
 w\mapsto \mathcal{L}(w,\Omega):=\int_{\Omega}\left[\ccc(x)\snr{Dw}\log(1+\snr{Dw})+a(x)\snr{Dw}^{q}\right] \dx\,.
$$
Here, as also in the following, $\Omega \subset \er^n$ denotes an open subset, $n\geq 2$ and $q>1$; as our results will be local in nature, without loss of generality we assume that $\Omega$ is also bounded. The functions $a(\cdot), \ccc(\cdot)$ are bounded and non-negative with $\ccc(\cdot)$ which is also bounded away from zero; a crucial point here is that $a(\cdot)$ is instead {\em allowed to vanish}. For the definition of local minimizer see \rif{defilocal}. We are interested in the case $a(\cdot), \ccc(\cdot)$ are only H\"older continuous, when proving gradient H\"older regularity of minima in fact corresponds to establish a nonlinear version of Schauder estimates (originally also due to Hopf and Caccioppoli). The functional in \rif{modello} appears to be a combination of the classical nearly linear growth one
\eqn{modellol}
$$
 w\mapsto \int_{\Omega}\ccc(x)\snr{Dw}\log(1+\snr{Dw})\dx
$$
treated when $\ccc(\cdot)\equiv 1$ for instance in \cite{FM} and by Marcellini \& Papi in \cite{MP}, 
with a $q$-power growth term weighted by a non-negative coefficient $a(\cdot)$. The terminology ``nearly linear'' accounts for an integrand whose growth in the gradient variable is superlinear, but it is still slower than that of any power type integrand as $z \mapsto |z|^p$ with $p>1$. When $\ccc(\cdot)\equiv 1$ the functional $\mathcal{L} (\cdot)$ can be in fact considered as a limiting case ($p\to 1$) of the by now largely studied double phase integral
\eqn{doppio}
$$
 w\mapsto \int_{\Omega}\left(\snr{Dw}^p+a(x)\snr{Dw}^{q}\right) \dx \,,\qquad  1 <p < q\,\,.
$$
This last one is a basic prototype of a nonautonomous functional featuring nonstandard polynomial growth conditions and a soft kind of nonuniform ellipticity (see discussion below). 
Originally introduced by Zhikov \cite{Z0, Z1, Zhom} in the setting of Homogenization of strongly anisotropic materials, the functional in \rif{doppio} was first studied in \cite{CM, sharp, FMM} and can be thought as a model for composite media with different hardening exponents $p$ and $q$. The geometry of the mixture of the two materials is in fact described by the zero set $\{a(x)\equiv 0\}$ of the coefficient $a(\cdot)$, where the transition from $q$-growth to $p$-growth takes place. In turn, the functional in \rif{modellol} appears for instance in the theory of plasticity with logarithmic hardening, another borderline case between plasticity with power hardening ($p>1$) and perfect plasticity ($p=1$) \cite{FrS, FS1}. Here, we are mostly interested in the theoretical issues raised by functionals as \rif{modello} in the context of regularity theory. In fact, $\mathcal{L} (\cdot)$ provides a basic example of nonuniformly and nonautonomous elliptic functional for which the currently available techniques do not allow to prove sharp Schauder type estimates. 
  In particular, although being a limiting case of \rif{doppio}, there is no way to adapt the methods available for \rif{doppio} to treat $\mathcal{L} (\cdot)$, already when $\ccc(\cdot)\equiv 1$. We recall that general nonuniformly elliptic problems have been the object of intensive investigation over the last two decades \cite{BS0, BS, BS2, bbgi, bildhauer, BF, BF2, FM}, starting from the more classical case of minimal surface type functionals \cite{LU, Simon} and those from the 60s \cite{ivanov0, ivanov1, IO, serrin, trudinger}. The monograph \cite{ivanov2} gives a valuable account of the early age of the theory and we recall that nonuniformly elliptic energies have been considered several times in the setting of nonlinear elasticity \cite{BFM, FMal, FMa, M0, ma5}. As for linear and nearly linear growth conditions in the vectorial case, we mention the recent extensive work of Kristensen \& Gmeineder \cite{gme1, gme2, gmek1, gmek2}. But let us first recall the situation for \rif{doppio}. The bound
\eqn{doppiob}
$$
\frac{q}{p} \leq 1+ \frac{\alpha}{n}
$$
guarantees the local H\"older continuity of the gradient of minima of \rif{doppio}  \cite{BCM, CM, DM}. Condition \rif{doppiob} is optimal in the sense that its failure generates the existence of minimizers that do not even belong to $W^{1,q}_{\loc}$ \cite{sharp} and develop singularities on fractals with almost maximal dimension \cite{FMM}. In \rif{doppiob} notice the delicate interaction between ambient dimension, the $(p,q)$-growth conditions in the gradient and the regularity of coefficient $a(\cdot)$. Despite the simple form of \rif{doppio}, all this already happens in the scalar case. In the autonomous case bounds of the type in \rif{doppiob}, with no appearance of $\alpha$, play a key role in the regularity theory of functionals with nonstandard growth as built by Marcellini in his by now classical papers \cite{M0, M1, M2, M3}. The key property of the functional \rif{doppio} is that it is \emph{uniformly elliptic} in the classical (pointwise) sense. This means that, with 
$G(x,z):= |z|^{p}+a(x)|z|^q$, we have
\eqn{finita}
$$
\sup_{x\in \Omega,\snr{z}\geq 1}\, \mathcal R_{G(x, \cdot)} (z)< c(p,q)\,,
$$ where $\mathcal R_{G(x, \cdot)} $ denotes the (pointwise) {\em ellipticity ratio} of $G(x, \cdot)$, i.e., 
$$\mathcal R_{G(x, \cdot)} (z):=  \frac{\mbox{highest eigenvalue of}\ \partial_{zz}  G(x,z)}{\mbox{lowest eigenvalue of}\  \partial_{zz} G(x,z)}\,.
$$
Nevertheless, the functional $\mathcal{L} (\cdot)$ exhibits a weaker form of nonuniform ellipticity, detectable via a larger quantity called nonlocal ellipticity ratio \cite{ciccio}. This is defined by
$$
\mathcal R_{G} (z, B):=  \frac{\sup_{x\in B}\, \mbox{highest eigenvalue of}\ \partial_{zz}  G(x,z)}{\inf_{x\in B}\, \mbox{ lowest eigenvalue of}\  \partial_{zz} G(x,z)}\,,
$$
where $B \subset \Omega$ is any ball. Indeed, note that 
\eqn{limitazione}
$$\mathcal R_{G} (z, B)\approx 1+ \|a\|_{L^\infty(B)}|z|^{q-p}$$
holds whenever $\{a(x)\equiv 0\}\cap B$ is nonempty, and it is therefore unbounded with respect to $|z|$. We refer to \cite{ciccio} for a discussion on these quantities; we just note that they do coincide in the autonomous case. It is precisely the occurrence of both \rif{finita} and \rif{limitazione} to imply the regularity of minima of the double phase functional in \rif{doppio} when assuming \rif{doppiob}. Indeed \rif{doppiob} and \rif{limitazione}, allow to implement a delicate perturbation method based on the fact that minimizers of the frozen integral $w \mapsto \int G(x_0, Dw)\dx$ enjoy good decay estimates, which is a consequence of \rif{finita}. A similar scheme, based on the occurrence of pointwise uniform ellipticity \rif{finita}, is the starting point for treating larger classes of problems featuring non-polynomial ellipticity conditions, see for instance the recent interesting papers by H\"asto \& Ok  \cite{HO1, HO2}. These also include integrals of the type 
$$
 w\mapsto \int_{\Omega}\left[\Phi_1(|Dw|)+a(x)\Phi_2(|Dw|)\right] \dx \,,\qquad  1 <p < q\,,
$$
where $t \mapsto \Phi_1(t)$ has superlinear growth (these have been treated in \cite{byun0, byun1}). Such schemes completely break down in the case of functionals as $\mathcal{L} (\cdot)$ since the boundedness of $z \mapsto \mathcal R_{G(x, \cdot)}(\cdot)$ fails for $G(x, z)\equiv  \ccc(x)|z|\log(1+|z|)+a(x)|z|^q$. Indeed, already in the case of the functional in \rif{modellol}, 
which is nonuniformly elliptic in the classical sense (i.e., \rif{finita} fails), no Lipschitz continuity result for minima is available but when considering differentiablity assumptions on $\ccc(\cdot)$ (see \cite{dm}). In other words, the validity of Hopf-Caccioppoli-Schauder estimates is an open issue. We fix the current situation in the following:
\begin{theorem}[Nearly linear Hopf-Caccioppoli-Schauder]\label{t1}
Let $u\in W^{1,1}_{\loc}(\Omega)$ be a local minimizer of the functional $\mathcal{L}(\cdot)$ in \trif{modello}, with
\eqn{bound}
\begin{cases}
$$
\displaystyle 0\le a(\cdot)\in C^{0,\alpha}(\Omega), & 1<q<1+\alpha/n\\
 \ccc(\cdot)\in C^{0, \ao}_{\loc}(\Omega),&\displaystyle 1/\Lambda \leq  \ccc(\cdot) \leq \Lambda\,,
\end{cases}
$$ where $\alpha, \ao\in (0,1)$ and $\Lambda\geq 1$. 
Then $Du$ is locally H\"older continuous in $\Omega$ and moreover, for every ball $B\equiv B_{r}\Subset \Omega$, $r\leq1$, the inequality
$$
\nr{Du}_{L^{\infty}(B/2)}\le c\left(\mint_{B}\left[\snr{Du}\log(1+\snr{Du})+a(x)\snr{Du}^{q}\right] \dx\right)^{\vartheta} +c
$$
holds with $c\equiv c(n,q, \Lambda,\alpha, \alpha_0,\nr{a}_{C^{0,\alpha}}, \nr{\ccc}_{C^{0,\ao}})$ and $\vartheta\equiv \vartheta(n,q,\alpha, \alpha_0)>0$.
\end{theorem}
Of course the condition $q<1+\alpha/n$ in \rif{bound}$_1$ is the sharp borderline version of \rif{doppiob} (let $p\to 1$), that guarantees gradient H\"older continuity of minima for the classical double phase functional \rif{doppio}. Let us only observe that the equality case in \rif{doppiob}, in comparison to \rif{bound}$_1$ assumed here, is typically linked to superlinear growth conditions $p>1$ in \rif{doppio}. The equality in \rif{doppiob} is indeed obtained via Gehring type results and it is ultimately again an effect of the uniform ellipticity in the sense of \rif{finita} (see \cite{BCM, dm0}). Such effects are missing in the nearly linear growth case.  
Once local gradient H\"older continuity is at hand, in the non-singular/degenerate case we gain\begin{corollary}[Improved exponents]\label{t2}
Under assumptions \trif{bound}, local minimizers of the functional 
\eqn{nondeggi}
$$
 w\mapsto \int_{\Omega}\big[\ccc(x)\snr{Dw}\log(1+\snr{Dw})+a(x)(\snr{Dw}^{2}+1)^{q/2}\big] \dx
$$
are locally $C^{1,\tia/2}$-regular in $\Omega$, where $\tia:=\min\{\alpha, \alpha_0\}$. In particular, local minimizers of the functional in \trif{modellol} are locally $ C^{1,\alpha_0/2}$-regular in $\Omega$ provided \trif{bound}$_2$ holds. 
\end{corollary}
We do not expect that, at least with the currently available techniques, the exponent $\tia/2$ in Corollary \ref{t2} can be improved \cite{manth1, manth2}. The loss of $1/2$ in the exponent is typical when dealing with functionals with subquadratic growth. 
\subsection{General statements} Theorem \ref{t1} and Corollary \ref{t2} are special cases of a class of results covering general functionals of the type
\eqn{generale}
$$
 w\mapsto \mathcal{N}(w,\Omega):=\int_{\Omega}\big[F(x,Dw)+a(x)(\snr{Dw}^{2}+\bb^2)^{q/2}\big] \dx\,,
$$
with $F(\cdot)$ being arbitrarily close to have linear growth and $\bb \in [0,1]$. This time we are thinking of integrands of the type
\eqn{verylinear}
$$
\begin{cases}
\ F(x,z) \equiv \ccc(x)|z|L_{k+1}(|z|)  \ \  \mbox{for $k\geq 0$} \\
 \ L_{k+1}(|z|)=\log(1+L_{k}(|z|))  \ \  \mbox{for $k\geq 0$}\,, \quad 
L_{0}(|z|)= |z|\,,
\end{cases}
$$
where $\ccc(\cdot)$ is as in \rif{bound}$_2$. For situations like this, we prepare a more technical theorem. We consider continuous integrands $F\colon \Omega \times \er^n\to [0, \infty)$ such that $z \mapsto F(x, z) \in C^{2}(\er^n)$ for every $x\in \Omega$ and satisfying
\eqn{assif}
$$
\begin{cases}
\ \nu|z|g(z) \leq F(x,z) \leq L(|z|g(z) +1)\\
\ \displaystyle \frac{\nu |\xi|^2}{(|z|^2+1)^{\mu/2}}  \leq \langle \partial_{zz}F(x,z)\xi,\xi\rangle\, , \quad  |\partial_{zz}F(x,z)| \leq  \frac{Lg(|z|)}{(|z|^2+1)^{1/2}} \\
\ |\partial_zF(x,z)-\partial_z F(y,z)| \leq L \snr{x-y}^{\ao}g(|z|), 
\end{cases}
$$
for every choice of $x,y\in \Omega$, $z, \xi \in \er^n$, with $\ao \in (0,1)$, $1\leq \mu <3/2$ and $0 < \nu \leq 1 \le L$ being fixed constants. Here $g\colon [0, \infty) \to [1, \infty)$ is a non-decreasing, concave and unbounded function such that $ t\mapsto tg(t) $ is convex. Moreover, we assume that for every $\eps >0$ there exists a constant $c_{g}(\eps)$ such that
 \eqn{0.1}
 $$
 g(t) \leq c_{g}(\eps) t^{\eps}\quad \mbox{holds for every $t\ge 1$}\,.
$$
The above assumptions are sufficient to get Lipschitz continuity of minimizers. In order to get gradient local H\"older continuity we still need a technical one (see also Remark \ref{ogniv} below):
\eqn{assi3}
$$
|\partial_z F(x, z)| \leq L|z| \quad \mbox{holds for every $|z|\leq 1$\,.}
$$
For functionals as in \rif{generale} we have
\begin{theorem}\label{t3}
Let $u\in W^{1,1}_{\loc}(\Omega)$ be a local minimizer of the functional in \trif{generale} under assumptions \trif{bound}$_1$ and \trif{assif}-\trif{0.1}. There exists $\mum \in (1,3/2)$, depending only $n,q, \alpha, \alpha_0$, such that, if $1\leq \mu < \mum$, then 
\eqn{30bis}
$$
\nr{Du}_{L^{\infty}(B/2)}\le c\left(\mint_{B}\left[F(x, Du)+a(x)(\snr{Du}^2+\bb^2)^{q/2}\right] \dx\right)^{\vartheta} +c
$$
holds whenever $B\equiv B_{r}\Subset \Omega$, $r\leq 1$, where $c\equiv c(\data)$ and $\vartheta \equiv \vartheta (n,q, \alpha, \alpha_0)\linebreak >0$. Moreover, assuming also \trif{assi3} implies that $Du$ is locally H\"older continuous is $\Omega$. Finally, if in addition to all of the above we also assume that $\bb >0$ and that $\partial_{zz}F(\cdot)$ is continuous, then $u \in C^{1,\tilde \alpha/2}_{\loc}(\Omega)$, with $\tilde \alpha = \min\{\alpha, \alpha_0\}$.   \end{theorem}
The notion of local minimizer we are using in this paper is classical and prescribes that a function $u \in W^{1,1}_{\loc}(\Omega)$ is a local minimizer of $\mathcal{N}(\cdot)$ if, for every ball $B\Subset \Omega$, it happens that 
\eqn{defilocal}
$$
\begin{cases}\mbox{$\mathcal{N}(u,B)$ is finite}\\
\mbox{$\mathcal{N}(u,B)\leq \mathcal{N}(w,B)$ holds whenever $w-u\in W^{1,1}_0(B)$}\,.
 \end{cases}
 $$ This implies, in particular, that $\mathcal N(u, \Omega_0)$ is finite whenever $\Omega_0\Subset \Omega$ is an open subset. In Theorem \ref{t3}, we are using the shorthand notation 
\eqn{idati}
$$
\data :=(n,q,\nu, L,\alpha, \ao,\nr{a}_{C^{0,\alpha}}, \nr{\ccc}_{C^{0,\ao}}, c_{g}(\cdot))$$
that we adopt for the rest of the paper. For further notation we refer to Section \ref{preliminari} below. We just remark that, in order to simplify the reading, in the following we shall still denote by $c \equiv c (\data)$ a constant actually depending only on a subset of the parameters listed 
in \rif{idati}. 

Theorem \ref{t3} has a technical nature, especially as far as the size of $\mum$ is concerned, which is not allowed to be too far from one (a natural limitation in view of the condition $q < 1+\alpha/n$). This is anyway sufficient to cover the model examples considered here, including \rif{verylinear}; see Section \ref{dimos}. We note that in the autonomous, nearly linear case, the literature reports several interesting results for functionals $w \mapsto \int F(Dw)\dx$, satisfying \rif{assif}$_{1,2}$ with restrictions as $1\leq \mu < 1+\texttt{o}(n)$, $\texttt{o}(n) \to 0$ when $n\to \infty$. These bounds are in fact of the type we also need to impose when selecting $\mum$ (this can be checked by tracking the constant dependence in the proofs). Starting by \cite{FM} conditions \rif{assif}$_2$ have been used in several papers devoted to functionals with linear or nearly linear growth; see for instance \cite{bildhauer, BF, gmek2} and related references.  
\vspace{1mm}
\begin{remark}\label{ogniv} The technical assumption \rif{assi3} is automatically implied, for a different constant $L$, by $\partial_{z}F(x, 0_{\er^n})=0_{\er^n}$; this follows using the second inequality in \rif{assif}$_2$. In turn, when $F(\cdot)$ is independent of $x$ we can always assume that $\partial_{z}F(0_{\er^n})=0_{\er^n}$. Indeed, we can switch from $F(z)$ to $F(z)-\langle \partial_zF(0_{\er^n}), z\rangle$ observing that such a replacement in the integrand in \rif{generale} does not change the set of local minimizers. Assumption \rif{assi3} is anyway verified by the model cases we are considering and in fact it deals with the behaviour of $F(\cdot)$ at the origin, while in nonuniformly elliptic equations the main problems usually come from the behaviour at infinity.
\end{remark}

\subsection{Techniques} The techniques needed for Theorem \ref{t3} are entirely different from those used in the case of classical superlinear double phase functionals in \rif{doppio} \cite{BCM, CM} and for pointwise uniformly elliptic problems with nonstandard growth conditions \cite{HO1, HO2}. This is essentially due to the failure of uniform ellipticity \rif{finita} for functionals as $\mathcal L(\cdot)$ and $\mathcal N(\cdot)$. On the other hand, the techniques based on De Giorgi and Moser's methods   used in the literature to deal with nearly linear growth functionals \cite{bildhauer, BF2, FrS, FM} are not viable here. In fact, as in our case coefficients are H\"older continuous, it is impossible to differentiate the Euler-Lagrange equation to get gradient regularity. This is in fact the classical obstruction to direct regularity proofs one encounters when proving Schauder estimates, eventually leading to the use of perturbation methods. In turn, perturbation methods are again not working in the nonuniformly elliptic case as the available a priori estimates are not homogeneous and iteration schemes fail. Our approach is new to the context and relies on a number of ingredients, partially already introduced and exploited in \cite{piovra}. Most notably, we make a delicate use of nonlinear potential theoretic methods allowing to work under the sharp bound 
\eqn{doppiob2}
$$
q < 1+ \frac{\alpha}{n}\,.
$$
We employ a rather complex scheme of renormalized fractional Caccioppoli inequalities on level sets \rif{16}. These are aimed at replacing the standard Bernstein technique, and encode the essential data of the problem via a wide range of parameters; see \rif{ilbeta2}-\rif{sonot}. Bernstein technique is traditionally based on the fact that functions as $|Du|^\gamma$ are subsolutions to linear elliptic equations for certain positive values of $\gamma$. This follows differentiating the original equation, something that is impossible in our setting. We therefore use a fractional and anisotropic version of the method according to which certain nonlinear functions of the gradient, namely 
\eqn{intrinseca}
$$
E(x,Du) \approx |Du|^{2-\mu} + a(x)|Du|^q\,,
$$
belong to suitable fractional De Giorgi's classes, i.e., they satisfy inequalities \rif{16}. The term renormalized accounts for the fact that inequalities \rif{16} are homogeneous with respect to $E(\cdot,Du)$, despite the functional $\mathcal N(\cdot)$ is not. This costs the price of getting multiplicative constants depending on $\|Du\|_{L^{\infty}}$. Such constants must be carefully kept under control all over the proof and reabsorbed at the very end. Notice that inequalities \rif{16} are obtained via a dyadic/atomic decomposition technique, finding its roots in \cite{KM}, that resembles the one used for Besov spaces \cite{AH, triebel} in the setting of Littlewood-Paley theory; see Section \ref{hy2} and Remark \ref{atomicore}. Iterating inequalities \rif{16} via a potential theoretic version of De Giorgi's iteration (Lemma \ref{revlem}) then leads to establish local Lipschitz bounds in terms of the nonlinear potentials defined in \rif{defi-P}; see Proposition \ref{priora}. A crucial point is that the fractional nature of the inequalities \rif{16} allows to sharply quantify how the rate of H\"older continuity of coefficients interacts with the growth of the terms they stick to. Using this last very fact and combining the various parameters in \rif{ilbeta2}-\rif{sonot} leads to determine the optimal nonlinear potentials allowing to prove Lipschitz continuity of local minimizers under the sharp bound in \rif{doppiob2}, but no matter how small the exponent $\alpha_0$ is. Throughout the whole process, the structure of the functional $\mathcal N(\cdot)$ needs to be carefully exploited at every point. This reflects in the peculiar choice of \rif{intrinseca} as the leading quantity to iterate. Lipschitz regularity is as usual the focal point in nonuniformly elliptic problems, since once $Du$ is (locally) bounded condition \rif{finita} gets automatically satisfied when $z\equiv Du$. Indeed at this stage gradient H\"older continuity can be achieved by exploiting and extending some hidden facts from more classical regularity theory; see Section \ref{latecnica}. In particular, we shall use some of the iteration schemes employed in the proof of nonlinear potential estimates for singular parabolic equations presented in \cite{KuM}. The arguments developed for this point are general and can be used in other settings; see Remark \ref{mare}. We finally observe that in this paper we preferred to concentrate on model cases in order to highlight the main ideas. Nevertheless we believe that the approach included here can be exported to a variety of different settings where it might yield results that are either sharp or unachievable otherwise.

\section{Preliminaries}\label{preliminari}
\subsection{Notation}
In the following we denote by $c$ a general constant larger than~$1$. Different occurrences from line to line will be still denoted by $c$. Special occurrences will be denoted by $c_*,  \tilde c$ or likewise. Relevant dependencies on parameters will be as usual emphasized by putting them in parentheses. We denote by $ B_{r}(x_0):= \{x \in \er^n  :   |x-x_0|< r\}$ the open ball with center $x_0$ and radius $r>0$; we omit denoting the center when it is not necessary, i.e., $B \equiv B_{r} \equiv B_{r}(x_0)$; this especially happens when various balls in the same context share the same center. We also denote
$$
\mathcal B_{r} \equiv B_{r}(0):= \{x \in \er^n  :   |x|< r\}
$$ when the ball in question is centred at the origin. Finally, with $B$ being a given ball with radius $r$ and $\gamma$ being a positive number, we denote by $\gamma B$ the concentric ball with radius $\gamma r$ and by $B/\gamma \equiv (1/\gamma)B$. Moreover, $Q_{\textnormal{inn}}\equiv Q_{\textnormal{inn}}(B)$ denotes the inner hypercube of $B$, i.e., the largest hypercube, with sides parallel to the coordinate axes and concentric to $B$, that is contained in $B$, 
$Q_{\textnormal{inn}}(B)\subset B$. The sidelength of $Q_{\textnormal{inn}}(B)$ equals $2r/\sqrt{n}$. The vector valued version of function spaces like $L^p(\Omega), W^{1,p}(\Omega)$, i.e., when the maps considered take values in $\er^k$, $k\in \en$, will be denoted by $L^p(\Omega;\er^k), W^{1,p}(\Omega;\er^k)$;  when no ambiguity will arise we shall still denote $L^p(\Omega)\equiv L^p(\Omega;\er^k)$ and so forth. For the rest of the paper we shall keep the notation 
\eqn{laelle}
$$\ell_{\omega}(z):=\sqrt{\snr{z}^{2}+\omega^{2}}\,,$$
 for $z\in \mathbb{R}^{k}$, $k\ge 1$ and $\omega\in [0,1]$. In particular, when $k=1$, we have a function defined on $\er$. 
With $\mathcal U \subset \er^{n}$ being a measurable subset with bounded positive measure $0<|\mathcal U|<\infty$, and with $f \colon \mathcal U \to \er^{k}$, $k\geq 1$, being an integrable map, we denote  
$$
   (f)_{\mathcal U} \equiv \mint_{\mathcal U}  f(x) \dx  :=  \frac{1}{|\mathcal U|}\int_{\mathcal U}  f(x) \dx\,.
$$ 
Moreover, given a (scalar) function $f$ and a number $\kappa\in \er$, we denote
\eqn{trunca}
$$ (f-\kk)_{+}:=\max\{f-\kk,0\}\quad \mbox{and} \quad (f-\kk)_{-}:=\max\{\kk-f,0\}\,.$$
Finally, whenever $f\colon \mathcal U \to \er^k$ and $\mathcal U$ is any set, we define 
$$
\textnormal{\texttt{osc}}_{\mathcal U}\, f := \sup_{x,y\in \mathcal U}\,  \snr{f(x)-f(y)}\,.
$$
As usual, with $\beta\in (0,1]$ we denote
$$
\|f\|_{C^{0, \beta}( \mathcal U)}:= \|f\|_{L^{\infty}( \mathcal U)} + [f]_{0, \beta;\mathcal U}\,,\quad [f]_{0, \beta;\mathcal U}:=\sup_{x,y\in \mathcal U, x\not=y}\frac{\snr{f(x)-f(y)}}{\snr{x-y}^\beta}\,.
$$
\subsection{Auxiliary results}
We shall use the vector fields $V_{\omega,p}\colon \er^{n} \to  \er^{n}$, defined by
\eqn{vpvqm}
$$
V_{\omega,p}:= (|z|^{2}+\omega^{2})^{(p-2)/4}z, \qquad 0 < p < \infty\ \ \mbox{and}\ \ \omega\in [0,1]
$$
whenever $z \in \er^{n}$ (compare \cite[(2.1)]{ha} for such extended range of $p$). 
They satisfy
\eqn{Vm}
$$
\snr{V_{\omega,p}(z_{1})-V_{\omega,p}(z_{2})}\approx_{n,p} (\snr{z_{1}}^{2}+\snr{z_{2}}^{2}+\omega^{2})^{(p-2)/4}\snr{z_{1}-z_{2}}
$$
that holds for every $z_1, z_2 \in \er^n$, $\omega \in[0,1]$ and $p>0$; for this we refer to \cite[Lemma 2.1]{ha}. We shall also use the following, trivial
\eqn{l6} 
$$
 (\snr{z_{1}}^{2}+\snr{z_{2}}^{2}+\omega^{2})^{-\gamma/2} \lesssim   \int_{0}^{1}(\snr{z_{2}+\tau(z_{1}-z_{2})}^{2}+\omega^{2})^{-\gamma/2}\, d\tau $$
that holds for every $\gamma \geq 0$ and whenever $|z_1|+|z_2|>0$. 
 When restricting from above the range of $\gamma$, the inequality becomes double-sided, i.e.
\eqn{l60}
$$
 (\snr{z_{1}}^{2}+\snr{z_{2}}^{2}+\omega^{2})^{-\gamma/2} \approx_{n,\gamma} \int_{0}^{1}(\snr{z_{2}+\tau (z_{1}-z_{2})}^{2}+\omega^{2})^{-\gamma/2}\, d\tau\,, 
$$
that instead holds for every $\gamma <1$; see \cite{ha}, \cite[Section 2]{giamod}. With $B \subset \Omega$ being a ball and $\sigma \in (0,1)$, we let 
\eqn{aii}
$$
\begin{cases}
a_{\sigma}(x):= a(x)+\sigma\\
a_{\textnormal{i}}(B):=\inf_{x\in B}a(x)\ \  \mbox{and}\ \  a_{\sigma, \textnormal{i}}(B):=a_{\textnormal{i}}(B)+\sigma 
\end{cases}
$$
and define
\eqn{vvv}
$$
\begin{cases}
\mathcal{V}_{\omega, \sigma}^{2}(x,z_{1},z_{2}) := \snr{V_{1,2-\mu}(z_{1})-V_{1,2-\mu}(z_{2})}^{2}\\  \hspace{29.7mm} +a_{\sigma}(x)\snr{V_{\omega,q}(z_{1})-V_{\omega,q}(z_{2})}^{2}\\
\mathcal{V}_{\omega,\sigma, \textnormal{i}}^{2}(z_{1},z_{2};B) :=\snr{V_{1,2-\mu}(z_{1})-V_{1,2-\mu}(z_{2})}^{2}\\  \hspace{30mm} 
+a_{\sigma,\textnormal{i}}(B)\snr{V_{\omega,q}(z_{1})-V_{\omega,q}(z_{2})}^{2}
\end{cases}
$$
where $z_{1},z_{2}\in \mathbb{R}^{n}$, $x\in \Omega$. A consequence of assumptions \rif{assif} is the following:
\begin{lemma}\label{marclemma} Let $F\colon \Omega\times \er^n\to [0, \infty)$ be as in \trif{assif}. Then 
$|\partial_z F(x,z)| \leq cg(|z|)$ holds for every $(x,z)\in \Omega$, where $c$ depends on $L$ and the constant $c_{g}(1)$ defined in \trif{0.1}. 
\end{lemma}
\begin{proof}
This is a version of \cite[Lemma 2.1]{M1} and the proof mimics that of Marcellini. Specifically, there we take $|h|=1 +|z|$ and, after using \rif{assif}$_1$ as in \cite{M1}, we estimate $g(|z\pm he_i|)\leq g(|z|+h)\leq g(|z|)+g(h)\leq 2 g(|z|)+g(1) \leq cg(|z|)$ (here $\{e_k\}$ denotes the standard basis of $\er^n$). Here we have used that $g(\cdot)\geq 1$ is non-decreasing and sublinear (since it is concave and non-negative) and finally \rif{0.1} to get $g(1)\leq c_g(1)\leq g (|z|)$. 
\end{proof}
Finally, a classical iteration lemma \cite[Lemma 6.1]{giu}.
\begin{lemma} \label{l5} Let $h\colon [t,s]\to \mathbb{R}$ be a non-negative and bounded function, and let $a,b, \gamma$ be non-negative numbers. Assume that the inequality 
$ 
h(\tau_1)\le  (1/2) h(\tau_2)+(\tau_2-\tau_1)^{-\gamma}a+b,
$
holds whenever $t\le \tau_1<\tau_2\le s$. Then $
h(t)\le c( \gamma)[a(s-t)^{-\gamma}+b]
$, holds too. 
\end{lemma}
\subsection{Fractional spaces} 
The finite difference operator $\tau_{h}\colon L^1(\Omega;\er^{k}) \to L^{1}(\Omega_{|h|};\er^k)$ is defined as
$\tau_{h}w(x):=w(x+h)-w(x)$ for $x\in \Omega_{\snr{h}}$, for every $\tau \in L^1(\Omega)$ and 
where $\Omega_{|h|}:=\{x \in \Omega \, : \, 
\dist(x, \partial \Omega) > |h|\}$. 
With $\beta \in (0,1)$, $s \in [1, \infty)$, $k \in \en$, $n \geq 2$, and with $\Omega \subset \er^n$ being an open subset, the space $W^{\beta ,s}(\Omega;\er^k )$ is defined as the set of maps $w \colon \Omega \to \er^k$ such that
\begin{flalign*}
\notag
\| w \|_{W^{\beta ,s}(\Omega;\er^k )} & := \|w\|_{L^s(\Omega;\er^k)}+ \left(\int_{\Omega} \int_{\Omega}  
\frac{\snr{w(x)
- w(y)}^{s}}{\snr{x-y}^{n+\beta s}} \dx \dy \right)^{1/s}\\
&=: \|w\|_{L^s(\Omega;\er^k)} + [w]_{\beta, s;\Omega} < \infty\,.\label{gaglia}
\end{flalign*}
When considering regular domains, as for instance the ball $\mathcal B_{1/2}$ (this is the only case we are going to consider here in this respect), the  embedding inequality
\eqn{immersione}
$$
\nr{w}_{L^{\frac{ns}{n-s\beta}}(\mathcal B_{1/2};\er^k)}\leq c(n,s,\beta)\nr{w}_{W^{\beta,s}(\mathcal B_{1/2};\er^k)}
$$
holds provided $s\geq 1, \beta \in (0,1)$, $k\in \N$ and $s\beta<n$. For this and basic results concerning Fractional Sobolev spaces we refer to \cite{guide} and related references. A characterization of fractional Sobolev spaces via finite difference is contained in the following lemma, that can be retrieved for instance from \cite{dm}.
\begin{lemma}\label{l4}
Let $B_{\varrho} \Subset B_{r}\subset \er^n$ be concentric balls with $r\leq 1$, $w\in L^{s}(B_{r};\er^k)$, $s\geq 1$ and assume that, for $\alpha_* \in (0,1]$, $S\ge 1$, there holds
$$
\nr{\tau_{h}w}_{L^{s}(B_{\rr};\er^k)}\le S\snr{h}^{\alpha_* } \  \mbox{
for every $h\in \mathbb{R}^{n}$ with $0<\snr{h}\le \frac{r-\rr}{K}$, where $K \geq 1$}\,.
$$
Then
\begin{flalign}
\notag \nr{w}_{W^{\beta,s}(B_{\rr};\er^k)} &\le\frac{c}{(\alpha_* -\beta)^{1/s}}
\left(\frac{r-\rr}{K}\right)^{\alpha_* -\beta}S\\
& \quad +c\left(\frac{K}{r-\rr}\right)^{n/s+\beta} \nr{w}_{L^{s}(B_{r};\er^k)} \label{cru2}
\end{flalign}
holds for every $\beta < \alpha_*$, with $c\equiv c(n,s)$. 
\end{lemma}
\subsection{Nonlinear potentials}\label{losec} We shall use a family of nonlinear potentials that, in their essence, goes  back to the fundamental work of Havin \& Mazya \cite{HM}. We have recently used such potentials in the setting of nonuniformly elliptic problems in \cite{BM, camel, camel2, ciccio, piovra} and in fact we shall in particular use  \cite[Section 4]{piovra} as a main source of tools in this respect. With $t,\sigma>0$, $m, \theta\geq 0$ being fixed parameters, and $f\in L^{1}(B_{r}(x_{0}))$ being such that $\snr{f}^{m} \in L^{1}(B_{r}(x_0))$, where $B_{r}(x_0)\subset \er^n$ is a ball, the nonlinear Havin-Mazya-Wolff type potential ${\bf P}_{t,\sigma}^{m,\theta}(f;\cdot)$ is defined by 
\eqn{defi-P} 
$$
{\bf P}_{t,\sigma}^{m,\theta}(f;x_0,r) := \int_0^r \varrho^{\sigma} \left(  \mint_{B_{\varrho}(x_0)} \snr{f}^{m} \dx \right)^{\theta/t} \frac{\d\varrho}{\varrho} \,.
$$
Suitable integrability properties of $f$ guarantee that ${\bf P}_{t,\sigma}^{m,\theta}(f; \cdot)$ is bounded, as stated in 
\begin{lemma}\label{crit}
Let $t,\sigma,\theta>0$ be numbers such that
\eqn{required}
$$
 n\theta>t\sigma\,.
$$ Let 
$B_{\rr}\subset B_{\rr+r_{0}}\subset \mathbb{R}^{n}$ be concentric balls with $\rr, r_0\leq1$ and $f\in L^{1}(B_{\rr+r_{0}})$ be a function such that $\snr{f}^{m}\in L^{1}(B_{\rr+r_{0}})$, for some $m>0$. Then
\eqn{stimazza}
$$
\nr{\mathbf{P}^{m,\theta}_{t,\sigma}(\cdot,r_{0})}_{L^{\infty}(B_{\rr})}\le c\|f\|_{L^{\gamma}(B_{\rr}+r_0)}^{m\theta/t} \quad \mbox{holds for every $\gamma >\frac{nm\theta}{t\sigma}>0$}
$$
and with $c\equiv c(n,t,\sigma,m,\theta, \gamma)$. 
\end{lemma}
Note that in \rif{stimazza} we are not requiring that $\gamma \geq 1$ and any value $\gamma>0$ is allowed.  
Next lemma is a version of \cite[Lemma 2.3]{piovra} and takes information from various results from Nonlinear Potential Theory starting by the seminal paper \cite{kilp}. It is essentially a nonlinear potential theoretic version of De Giorgi's iteration. Notice that the very crucial point in the next statement is the tracking, in the various inequalities, of the explicit dependence on the constants $M_0, M_i$. 
\begin{lemma}\label{revlem}
Let $B_{r_{0}}(x_{0})\subset \mathbb{R}^{n}$ be a ball, $n\ge 2$, and, for $j \in \{1, 2,3\}$, consider functions $f_j$, 
$
|f_j|^{m_j} \in L^1(B_{2r_0}(x_{0}))$, and constants $\chi >1$, $\sigma_j, m_j, \theta_j>0$ and $c_*,M_0 >0$,  $\kappa_0, M_j\geq 0$. Assume that $w \in L^2(B_{r_0}(x_0))$ is such that for all $\kk\ge \kk_{0}$, and for every ball $B_{\rr}(x_{0})\subseteq B_{r_{0}}(x_{0})$, the inequality
\begin{flalign}
\notag  \left(\mint_{B_{\rr/2}(x_{0})}(w-\kk)_{+}^{2\chi}  \dx\right)^{1/\chi}  &\le c_{*}M_{0}^{2}\mint_{B_{\rr}(x_{0})}(w-\kk)_{+}^{2}  \dx\\
 & \qquad +c_{*}\sum_{j=1}^{3} M_j^{2}\rr^{2\sigma_j}\left(\mint_{B_{\rr}(x_{0})}\snr{f_j}^{m_j}  \dx\right)^{\theta_j}
 \label{revva}
\end{flalign}
holds (recall the notation in \trif{trunca}). If $x_{0}$ is a Lebesgue point of $w$, then
\begin{flalign}\label{siapplica}
 \notag w(x_{0})  & \le\kk_{0}+cM_{0}^{\frac{\chi}{\chi-1}}\left(\mint_{B_{r_{0}}(x_{0})}(w-\kk_{0})_{+}^{2}  \dx\right)^{1/2}\\
 &\qquad 
+cM_{0}^{\frac{1}{\chi-1}}\sum_{j=1}^{3} M_j\mathbf{P}^{m_j,\theta_j}_{2,\sigma_j}(f_j;x_{0},2r_{0})
\end{flalign}
holds with $c\equiv c(n,\chi,\sigma_j,\theta_j,c_{*})$.  
\end{lemma}

\section{Auxiliary integrands and their eigenvalues.}\label{const}
With reference to a general functional of the type $\mathcal{N}(\cdot)$ in \rif{generale}, considered under the assumptions of Theorem \ref{t3} (in particular implying that $q<3/2$), we define the integrand 
\eqn{hhh0}
$$H(x, z):=F(x, z)+a(x)\snr{z}^{q}$$ for $z \in \er^n$, $x\in \Omega$. 
We fix an arbitrary ball $B\subset \Omega$ centred at $x_{\rm c}$ and, with $\omega \in [0,1]$ and $\sigma \in (0,1)$, recalling \rif{laelle} and \rif{aii}, we define 
\eqn{defiH}
$$
\begin{cases}
H_{\omega,\sigma}(x,z):=F(x,z)+a_{\sigma}(x)[\ell_{\omega}(z)]^{q}\\
H_{\omega,\sigma,\textnormal{i}}(z)\equiv H_{\omega,\sigma,\textnormal{i}}(z;B):=F(x_{\rm c}, z)+a_{\sigma, \textnormal{i}}(B)[\ell_{\omega}(z)]^{q}\\
\eh_{\omega,\sigma,\textnormal{i}}(t)\equiv \eh_{\omega,\sigma,\textnormal{i}}(t;B):=tg(t)+a_{\sigma, \textnormal{i}}(B)[t^2+\omega^2]^{q/2}+1, \  \ t\geq 0\,.
\end{cases}
$$
It follows that
\eqn{defiHHH} 
$$
\mbox{$t \mapsto \eh_{\omega,\sigma,\textnormal{i}}(t)$ is non-decreasing}\,.
$$
Note that abbreviations as $H_{\omega,\sigma,\textnormal{i}}(z)\equiv H_{\omega,\sigma,\textnormal{i}}(z;B)$, used in \rif{defiH}, will take place whenever there will be no ambiguity concerning the identity of the ball $B$ in question. A similar shortened notation will be adopted for the case of further quantities depending on a single ball $B$. We next introduce two functions, $\lambda_{\omega,\sigma},\Lambda_{\omega,\sigma}\colon \Omega \times [0, \infty)\to [0, \infty)$, bound to describe the behaviour of the eigenvalues of $\partial_{zz}H_{\omega,\sigma}(\cdot)$. These are (recall that $q, \mu<3/2$)
\eqn{autovalori}
  $$
  \begin{cases}
    \lambda_{\omega,\sigma}(x,\snr{z}):=(|z|^2+1)^{-\mu/2}+ (q-1)a_{\sigma}(x)[\ell_{\omega}(z)]^{q-2}\\
    \Lambda_{\omega,\sigma}(x,\snr{z}):= (|z|^2+1)^{-1/2}g(|z|)+a_{\sigma}(x)[\ell_{\omega}(z)]^{q-2}\,.
    \end{cases}
$$
It then follows that
\begin{flalign}\label{rege.2}
\begin{cases}
\ \sigma [\ell_{\omega}(z)]^{q}\le H_{\omega,\sigma}(x,z)\le c\left([\ell_{\omega}(z)]^{q}+1\right)\\
\ \lambda_{\omega,\sigma}(x,\snr{z})\snr{\xi}^{2}\leq c\langle\partial_{zz}H_{\omega,\sigma}(x,z)\xi,\xi\rangle\,, \quad \snr{\partial_{zz}H_{\omega,\sigma}(x,z)}\le c\Lambda_{\omega,\sigma}(x,\snr{z})\\
\ \mathcal{V}_{\omega, \sigma}^{2}(x,z_{1},z_{2})\leq c\langle \partial_z  H_{\omega,\sigma}(x,z_{1})-\partial_z  H_{\omega,\sigma}(x,z_{2}),z_{1}-z_{2} \rangle\\
\ \snr{\partial_zH_{\omega,\sigma}(x,z)-\partial_zH_{\omega,\sigma}(y,z)}\le c \snr{x-y}^{\alpha_0}g(|z|)+c\snr{x-y}^{\alpha}[\ell_{\omega}(z)]^{q-1}
\end{cases}
\end{flalign}
hold for all $x,y\in \Omega$, $z,z_{1},z_{2},\xi\in\mathbb{R}^{n}$, where $c \equiv c (\data)$ (recall the definition of $\mathcal{V}_{\omega, \sigma}$ in \rif{vvv} and \rif{vpvqm}). 
Notice that \rif{rege.2}$_{1,2}$ directly follow from \rif{assif}$_{1,2}$, respectively, and \rif{0.1} used with $\eps=q-1$. The inequality in $\eqref{rege.2}_{3}$ is a straightforward consequence of \rif{Vm}-\rif{l6} and \rif{rege.2}$_{2}$ (see Remark \ref{VVremark} below). Finally, \rif{rege.2}$_4$ is a consequence of \rif{assif}$_{3}$. When $\omega >0$, by \rif{0.1} and \rif{rege.2}$_2$ the integrand $H_{\omega,\sigma}(\cdot)$ also satisfies, for all $x,y \in \Omega$, $\snr{x-y}\leq 1$ and $z,\xi\in \mathbb{R}^{n}$
\begin{eqnarray}\label{regecor}
\begin{cases}
\ \sigma(q-1) [\ell_{1}(z)]^{q-2}\snr{\xi}^{2}\leq c\langle\partial_{zz}H_{\omega,\sigma }(x,z)\xi,\xi\rangle\\
\ \snr{\partial_z  H_{\omega,\sigma}(x,z)}\ell_{1}(z)+\snr{\partial_{zz}H_{\omega,\sigma}(x,z)}[\ell_{1}(z)]^{2}\le c[\ell_{1}(z)]^{q}\\
\ \snr{\partial_zH_{\omega,\sigma}(x,z)-\partial_zH_{\omega,\sigma}(y,z)}\le c \snr{x-y}^{\tilde \alpha}[\ell_{1}(z)]^{q-1}\,,
\end{cases}
\end{eqnarray}
with $\tia =\min\{\alpha, \alpha_0\}$ and $c\equiv c(\data, \omega).$ Note that, on the contrary of the previous displays, the constant $c$ in \rif{regecor} depends on $\omega$ and that for \rif{regecor}$_2$ we also need Lemma \ref{marclemma} and \rif{rege.2}$_{2}$. As for $H_{\omega,\sigma,\textnormal{i}}(\cdot)$, by defining 
    \eqn{ill}
      $$
  \begin{cases}
  \lambda_{\omega,\sigma,\textnormal{i}}(\snr{z}):=(|z|^2+1)^{-\mu/2}+ (q-1)a_{\sigma, \textnormal{i}}(B)[\ell_{\omega}(z)]^{q-2}\\
   \Lambda_{\omega,\sigma,\textnormal{i}}(\snr{z}):=\ (|z|^2+1)^{-1/2}g(|z|)+a_{\sigma, \textnormal{i}}(B)[\ell_{\omega}(z)]^{q-2}
    \end{cases}
    $$
we have that
    \eqn{rege.2i}
$$
\begin{cases}
\ 
\sigma [\ell_{\omega}(z)]^{q}\le H_{\omega,\sigma,\textnormal{i}}(z)\le c([\ell_{\omega}(z)]^{q}+1)\\
\  \snr{\partial_z  H_{\omega,\sigma,\textnormal{i}}(z)}\le cg(|z|) +  ca_{\sigma, \textnormal{i}}(B)[\ell_{\omega}(z)]^{q-2}|z|\\
\ \lambda_{\omega,\sigma,\textnormal{i}}(\snr{z})\snr{\xi}^{2}\leq c\langle\partial_{zz}H_{\omega,\sigma,\textnormal{i}}(z)\xi,\xi\rangle, \quad \snr{\partial_{zz}H_{\omega,\sigma,\textnormal{i}}(z)}\le c\Lambda_{\omega,\sigma,\textnormal{i}}(\snr{z})\\
\ \mathcal{V}^{2}_{\omega, \sigma, \textnormal{i}}(z_{1},z_{2};B)\leq c \langle \partial_z  H_{\omega,\sigma,\textnormal{i}}(z_{1})-\partial_z  H_{\omega,\sigma,\textnormal{i}}(z_{2}),z_{1}-z_{2} \rangle
\end{cases}
$$
holds whenever $z,z_{1},z_{2},\xi\in\mathbb{R}^{n}$, with $c\equiv c(\data)$; for \rif{rege.2i}$_2$ we have used Lemma \ref{marclemma}. 
 In particular, it holds that
\eqn{ellrati}
$$
\frac{\Lambda_{\omega,\sigma,\textnormal{i}}(\snr{z})}{\lambda_{\omega,\sigma,\textnormal{i}}(\snr{z})}\le (|z|^2+1)^{\frac{\mu-1}{2}}g(\snr{z})+ \frac{1}{q-1}\leq c(q)(|z|^2+1)^{\frac{\mu-1}{2}}g(\snr{z}) 
$$
(recall that $g(\cdot)\geq 1$). 
As in \rif{regecor}, when $\omega>0$ the function $H_{\omega,\sigma,\textnormal{i}}(\cdot)$ satisfies
\eqn{regecor.1}
$$
\begin{cases}
\ \sigma(q-1) [\ell_{1}(z)]^{q-2}\snr{\xi}^{2}\leq c\langle\partial_{zz}H_{\omega,\sigma,\textnormal{i}}(z)\xi,\xi\rangle\\
\ \snr{\partial_z  H_{\omega,\sigma,\textnormal{i}}(z)}\ell_{1}(z)+\snr{\partial_{zz}H_{\omega,\sigma,\textnormal{i}}(z)}[\ell_{1}(z)]^{2}\le c[\ell_{1}(z)]^{q}
\end{cases}
$$
for all $z,\xi\in \mathbb{R}^{n}$, with $c\equiv c(\data, \omega).$
Two quantities that will have an important role in the forthcoming computations are, for $x\in \Omega$ and $t\geq 0$
\eqn{eii}
$$
\begin{cases}
\displaystyle E_{\omega,\sigma}(x,t):=\int_{0}^{t}\lambda_{\omega,\sigma}(x,s)s\ds\\
\displaystyle E_{\omega,\sigma, \textnormal{i}}(t)\equiv E_{\omega,\sigma, \textnormal{i}}(t;B):=\int_{0}^{t}\lambda_{\omega,\sigma,\textnormal{i}}(s;B)s\ds\,.
\end{cases}
$$
It follows that (recall that $\mu<2$)
\eqn{eii2}
$$
\begin{cases}E_{\omega,\sigma}(x,t):=\frac{1}{2-\mu}\left([\ell_{1}(t)]^{2-\mu}-1\right)+(1-1/q)a_{\sigma}(x)\left([\ell_{\omega}(t)]^{q}-\omega^q\right)\\
E_{\omega,\sigma,\textnormal{i}}(t):= \frac{1}{2-\mu}\left([\ell_{1}(t)]^{2-\mu}-1\right)+(1-1/q)a_{\sigma, \textnormal{i}}(B)\left([\ell_{\omega}(t)]^{q}-\omega^q\right)
\end{cases}
$$
and this leads to define
\eqn{eiit}
$$
\begin{cases}\ti E_{\omega,\sigma}(x,t):=\frac{1}{2-\mu}[\ell_{1}(t)]^{2-\mu}+(1-1/q)a_{\sigma}(x)[\ell_{\omega}(t)]^{q}\\
\ti E_{\omega,\sigma,\textnormal{i}}(t):= \frac{1}{2-\mu}[\ell_{1}(t)]^{2-\mu}+(1-1/q)a_{\sigma, \textnormal{i}}(B)[\ell_{\omega}(t)]^{q}\,.
\end{cases}
$$
We record a few basic properties of these functions in the following:
\begin{lemma}\label{leee} The following holds about the functions in \eqref{eii}, whenever $B\subseteq \Omega$ is a ball:
\begin{itemize}
\item For every $s,t\in [0,\infty)$
\begin{flalign}\label{0.000}
 \notag &\snr{E_{\omega,\sigma,\textnormal{i}}(s)-E_{\omega,\sigma,\textnormal{i}}(t)}\\
 & \  \le
\left[(s^2+t^2+1)^{(1-\mu)/2}+a_{\sigma, \textnormal{i}}(B)(s^2+t^2+\omega^2)^{(q-1)/2}\right]\snr{s-t}.
\end{flalign}
\vspace{1mm}
\item For all $x\in B$
\eqn{0.00}
$$
\snr{E_{\omega,\sigma}(x,t)-E_{\omega,\sigma, \textnormal{i}}(t)}\le (1-1/q) \snr{a(x)-a_{\textnormal{i}}(B)}\left([\ell_{\omega}(t)]^{q}-\omega^q\right).
$$
\vspace{1mm}
\item There exists constant $T\equiv T(\mu)\geq 1$ such that 
\eqn{0.001}
$$
t\le 2 [E_{\omega,\sigma}(x,t)]^{\frac{1}{2-\mu}}  \quad \mbox{and}\quad t\le  2[E_{\omega,\sigma,\textnormal{i}}(t)]^{\frac{1}{2-\mu}}
$$
hold for all $x\in \Omega$, $t\ge T$.
\vspace{1mm}
\item There exists a constant $c\equiv c(\mu, q,\nu)$ such that 
\eqn{0.002}
$$
\begin{cases}
|z|+\tilde E_{\omega,\sigma}(x,|z|)\leq c [H_{\omega,\sigma}(x,z)+1]\\
|z|+\tilde E_{\omega,\sigma, \textnormal{i}}(|z|) \leq c [H_{\omega,\sigma,\textnormal{i}}(z)+1]
 \end{cases}
$$
hold for every $x\in \Omega$ and $z\in \er^n$. 
\end{itemize}
\end{lemma}
\vspace{1mm}
\begin{proof} The only inequality deserving some comments is \rif{0.000}, the other being direct consequences of the definitions in \rif{eii2}-\rif{eiit}. Notice that 
$$
|E_{\omega,\sigma,\textnormal{i}}'(t)| \leq c [\ell_{1}(t)]^{1-\mu}+ca_{\sigma}(x)[\ell_{\omega}(t)]^{q-1}
$$
with $c\equiv c (q, \mu)$. Using this last inequality we have 
\begin{flalign*}
 \snr{E_{\omega,\sigma,\textnormal{i}}(s)-E_{\omega,\sigma,\textnormal{i}}(t)} & \le \int_0^1
\snr{E_{\omega,\sigma,\textnormal{i}}'(t+\tau(s-t))}\, d\tau |s-t|\\
&  \leq c \int_0^1 (\snr{t+\tau(s-t)}^2+1)^{\frac{1-\mu}{2}}\, d\tau \\
& \quad +c 
a_{\sigma, \textnormal{i}}(B) \int_0^1 (\snr{t+\tau(s-t)}^2+\omega^2)^{\frac{q-1}{2}}\, d\tau
\end{flalign*}
so that \rif{0.000} follows applying \rif{l60} with $\gamma\equiv \mu-1$ and $\gamma\equiv 1-q$. \end{proof}
\begin{remark}\label{VVremark}The proof of $\eqref{rege.2}_{3}$ is rather standard, but we report it for completeness as the exponents involved here are not the usual ones. Using, in turn, \rif{Vm} (with $p=2-\mu, q$) and \rif{l6} (with $\gamma = \mu, 2-q$), and finally the first inequality in $\eqref{rege.2}_{2}$, we have  
\begin{flalign*}
 \mathcal{V}_{\omega, \sigma}^{2}(x,z_{1},z_{2}) &  \leq c (\snr{z_{1}}^{2}+\snr{z_{2}}^{2}+1)^{-\mu/2} |z_{1}-z_{2}|^2\\
  & \qquad +c a_{\sigma}(x)(\snr{z_{1}}^{2}+\snr{z_{2}}^{2}+\omega^{2})^{(q-2)/2} |z_{1}-z_{2}|^2 \\
  & \leq c \int_{0}^{1}(\snr{z_{2}+\tau(z_{1}-z_{2})}^{2}+1)^{-\mu/2}\, d\tau  |z_{1}-z_{2}|^2\\
& \qquad + c a_{\sigma}(x)\int_{0}^{1}(\snr{z_{2}+\tau(z_{1}-z_{2})}^{2}+\omega^{2})^{(q-2)/2}\, d\tau  |z_{1}-z_{2}|^2\\
& \leq c \int_0^1 \langle \partial_{zz}  H_{\omega,\sigma}(x,z_{2}+\tau (z_{1}-z_{2}))(z_1-z_2),z_{1}-z_{2} \rangle \d\tau\\
& = c\langle \partial_z  H_{\omega,\sigma}(x,z_{1})-\partial_z  H_{\omega,\sigma}(x,z_{2}),z_{1}-z_{2} \rangle\,.
\end{flalign*}

\end{remark}
\section{Lipschitz estimates with small anisotropicity}\label{rere}
Let $B_{\tau}\equiv B_{\tau}(x_{\rm c})\Subset \Omega$ be a ball with $\tau \leq 1$ and let $u_{0}\in W^{1,\infty}(B_{\tau})$. With $\omega \in (0,1]$, we define $v_{\omega,\sigma}\in u_{0}+W^{1,q}_{0}(B_{\tau})$ as the unique solution to the Dirichlet problem
\eqn{pdfz}
$$
v_{\omega,\sigma}\mapsto \min_{w \in u_{0}+W^{1,q}_{0}(B_{\tau})} \int_{B_{\tau}}H_{\omega,\sigma,\textnormal{i}}(Dw) \dx, 
$$
where, according to the notation in \rif{defiH}$_2$, it is $H_{\omega,\sigma,\textnormal{i}}(z)\equiv H_{\omega,\sigma,\textnormal{i}}(z;B_{\tau})$. In particular, we have 
\eqn{e.fz}
$$
\int_{B_{\tau}}H_{\omega,\sigma,\textnormal{i}}(Dv_{\omega,\sigma}) \dx \le \int_{B_{\tau}}H_{\omega,\sigma,\textnormal{i}}(Du_{0}) \dx\,.
$$
Existence, and uniqueness, follow by Direct Methods of the Calculus of Variations and the Euler-Lagrange equation
\eqn{el.fz}
$$
\int_{B_{\tau}}\langle\partial_z  H_{\omega,\sigma,\textnormal{i}}(Dv_{\omega,\sigma}), D\varphi\rangle \dx=0
$$
holds for all $\varphi\in W^{1,q}_{0}(B_{\tau})$ by \rif{regecor.1}. 
Again \eqref{regecor.1} and standard regularity theory \cite[Chapter 8]{giu}, \cite{manth1, manth2} imply
\eqn{0.fz}
$$
\begin{cases}
v_{\omega,\sigma}\in W^{1,\infty}_{\loc}(B_{\tau})\cap W^{2,2}_{\loc}(B_{\tau})
\\
 \partial_z  H_{\omega, \sigma, \textnormal{i}}(Dv_{\omega,\sigma})\in W^{1,2}_{\loc}(B_{\tau};\mathbb{R}^{n})\,.
 \end{cases}
$$
\begin{proposition}\label{fzfz} With $v_{\omega,\sigma}\in u_{0}+W^{1,q}_{0}(B_{\tau})$ as in \trif{pdfz}, for every integer $n\geq 2$ there exists a continuous and non-decreasing function $\mu_n \colon (0,1) \to (1, 3/2)$ such that, for every $\delta \in (0,1)$, if \trif{assif} holds with $1\leq \mu < \mu_n(\delta)$, then
\eqn{8.fz}
$$
\nr{Dv_{\omega,\sigma}}_{L^{\infty}(B_{3\tau/4})}\le c \eh_{\omega,\sigma,\textnormal{i}}^{\delta}\left(\nr{Du_{0}}_{L^{\infty}(B_{\tau})}\right)\nr{Du_{0}}_{L^{\infty}(B_{\tau})}+c
$$
holds too, with $c\equiv c(\data, \delta)\geq 1$, where $\eh_{\omega,\sigma,\textnormal{i}}(\cdot)$ has been defined in \eqref{defiH}$_3$. Moreover, whenever $M$ is a number such that
\eqn{mm}
$$ \nr{Dv_{\omega,\sigma}}_{L^{\infty}(B)}+1\leq M\,, $$
where $B\Subset B_{\tau}$ is a ball (not necessarily concentric to $B_{\tau}$), the Caccioppoli type inequality
\eqn{10.fz2}
$$
\mint_{3B/4}\snr{D(E_{\omega,\sigma,\textnormal{i}}(\snr{Dv_{\omega,\sigma}})-\kk)_{+}}^{2} \dx\leq \frac{cM^{\delta}}{|B|^{2/n}}\mint_{B}(E_{\omega,\sigma,\textnormal{i}}(\snr{Dv_{\omega,\sigma}})-\kk)_{+}^{2} \dx
$$holds whenever $\kappa \geq 0$, again with $c\equiv c(\data, \delta)$. Here, according to the notation in \eqref{eii}$_2$, it is  
$E_{\omega,\sigma,\textnormal{i}}(\snr{Dv_{\omega,\sigma}})\equiv E_{\omega,\sigma,\textnormal{i}}(\snr{Dv_{\omega,\sigma}};B_{\tau})$. \end{proposition}
\begin{proof} The final shape of $\mu_n>1$ will be determined in the course of the proof via successive choices of $\mu$ (to be taken closer and closer to one, depending on $\delta$).  
All the constants in the forthcoming estimates will be independent of $\omega,\sigma$, so we shall omit indicating the dependence on such parameters, simply denoting $v \equiv v_{\omega,\sigma}$, $H_{\textnormal{i}}\equiv H_{\omega,\sigma,\textnormal{i}}$, $\eh_{\textnormal{i}}\equiv \eh_{\omega,\sigma,\textnormal{i}}$, $E_{\textnormal{i}}\equiv E_{\omega,\sigma,\textnormal{i}}$, and so on. Properties \eqref{0.fz} allow to differentiate \eqref{el.fz}, i.e., replacing $\varphi$ by $D_s\varphi$ for $s \in \{1, \ldots, n\}$ and integrate by parts. Summing the resulting equations over $s \in \{1, \ldots, n\}$, we get
\eqn{1.fz}
$$
\sum_{s=1}^{n}\int_{B_{\tau}}\langle\partial_{zz}H_{\textnormal{i}}(Dv)DD_{s}v,D\varphi\rangle \dx=0\,,
$$
that, again thanks to \eqref{0.fz}, holds for all $\varphi\in W^{1,2}(B_{\tau})$ such that $\supp\, \varphi \Subset B_{\tau}$. Let $B\Subset B_{\tau}$ be a ball, $\kk\ge 0$ be any non-negative number and $\eta\in C^{1}_{c}(B_{\tau})$ be a cut-off function such that $\mathds{1}_{3B/4}\le \eta\le \mathds{1}_{5B/6}$ and $\snr{D\eta}\lesssim |B|^{-1/n}$. By  \eqref{0.fz} the functions $\varphi\equiv \varphi_s:=\eta^{2}(E_{\textnormal{i}}(\snr{Dv})-\kk)_{+}D_{s}v$ are  admissible in \eqref{1.fz}, therefore via \eqref{ill}-\eqref{rege.2i} and Young's inequality we obtain
\begin{flalign}
\notag &\int_{B_{\tau}}\lambda_{\textnormal{i}}(\snr{Dv})(E_{\textnormal{i}}(\snr{Dv})-\kk)_{+}\snr{D^{2}v}^{2}\eta^{2} \dx+\int_{B_{\tau}}\snr{D(E_{\textnormal{i}}(\snr{Dv})-\kk)_{+}}^{2} \eta^{2}\dx \nonumber \\
&\qquad       \qquad \qquad\qquad\le c\int_{B_{\tau}}\left[\frac{\Lambda_{\textnormal{i}}(\snr{Dv})}{\lambda_{\textnormal{i}}(\snr{Dv})}\right]^2(E_{\textnormal{i}}(\snr{Dv})-\kk)_{+}^{2}\snr{D\eta}^{2} \dx
\label{vetti}
\end{flalign}
with $c\equiv c (\data)$. See \cite[Lemmas 4.5-4.6]{BM} for more details and a similar inequality. Using
\rif{ellrati} and a few elementary manipulations, we again find 
\begin{flalign*} &\int_{B}\snr{D[\eta^2(E_{\textnormal{i}}(\snr{Dv})-\kk)_{+}]}^{2} \dx \\
& \qquad  \le c
\int_{B}
(|Dv|^2+1)^{ \mu-1 }[g(\snr{Dv})]^2
(E_{\textnormal{i}}(\snr{Dv})-\kk)_{+}^{2}\snr{D\eta}^{2} \dx
\end{flalign*}
and taking $M$ as in 
\rif{mm} we arrive at
$$
\mint_{B}\snr{D[\eta^2(E_{\textnormal{i}}(\snr{Dv})-\kk)_{+}]}^{2} \dx \leq \frac{cM^{2(\mu-1)} [g(M)]^2}{|B|^{2/n}}\mint_{B}(E_{\textnormal{i}}(\snr{Dv})-\kk)_{+}^{2} \dx,
$$
again with $c\equiv c(\data)$. In turn, using \rif{0.1} with $\eps >0$ and recalling that $M\geq 1$, we find
\eqn{10.fz}
$$
\mint_{B}\snr{D[\eta^2(E_{\textnormal{i}}(\snr{Dv})-\kk)_{+}]}^{2} \dx \leq \frac{cM^{2(\mu-1+\eps)} }{|B|^{2/n}}\mint_{B}(E_{\textnormal{i}}(\snr{Dv})-\kk)_{+}^{2} \dx,
$$
again with $c\equiv c (\data, \eps)$. This gives \rif{10.fz2} recalling that $\eta \geq 1$ on $3B/4$ and provided $\mu\equiv \mu(\delta) \geq 1$ and $\eps\equiv \eps(\delta)>0$ are such that 
\eqn{sce1}
$$
\mu-1+\eps \leq \delta/2\,.
$$
We pass to the proof of \rif{8.fz}. With $T$ being the constant in \rif{0.001}, we can assume without loss of generality that 
\eqn{prebb}
$$
\nr{Dv_{\omega,\sigma}}_{L^{\infty}(B_{3\tau/4})} \geq  T\geq 1
$$
holds, otherwise \rif{8.fz} follows trivially. By \rif{10.fz} and 
Sobolev embedding theorem we obtain
\eqn{scegli}
$$
\left(\mint_{B/2}(E_{\textnormal{i}}(\snr{Dv})-\kk)_{+}^{2\chi} \dx\right)^{1/\chi}\le cM^{2(\mu-1+\eps)} \mint_{B}(E_{\textnormal{i}}(\snr{Dv})-\kk)_{+}^{2} \dx,
$$
where $2\chi\equiv 2\chi(n)>2$ is the exponent coming from the Sobolev embedding exponent, and $c\equiv c (\data, \eps)$. We now want to apply Lemma \ref{revlem} to $w\equiv E_{\textnormal{i}}(\snr{Dv})$ using \rif{scegli} to satisfy \rif{revva} (here all the terms involving the functions $f_j$ in \rif{revva} are not present). We fix parameters $3\tau/4\le \tau_{1}<\tau_{2}\le 5\tau/6$ and related balls
$$B_{3\tau/4}(x_{\rm c})\Subset B_{\tau_1}(x_{\rm c})\Subset B_{\tau_2}(x_{\rm c})\subset B_{5\tau/6}(x_{\rm c})\Subset B_{\tau}(x_{\rm c})\equiv B_{\tau}\,,$$ this time all concentric to the initial ball $B_{\tau}\equiv B_{\tau}(x_{\rm c})$. 
Take $x_{0}\in B_{\tau_{1}}$ arbitrary and set $r_{0}:=(\tau_{2}-\tau_{1})/8$, so that $ B_{r_{0}}(x_0) \subset B_{\tau_2}$, choose $M\equiv 2\nr{Dv}_{L^{\infty}(B_{\tau_2})}$ and take $B\equiv B_{\rr}(x_0)\subset B_{r_0}(x_0)$ in \rif{scegli}. This yields
\begin{flalign}
\notag & \left(\mint_{B_{\rr/2}(x_0)}(E_{\textnormal{i}}(\snr{Dv})-\kk)_{+}^{2\chi} \dx\right)^{1/\chi}\\
& \quad \quad  \le cM^{2(\mu-1+\eps)} \mint_{B_{\rr}(x_0)}(E_{\textnormal{i}}(\snr{Dv})-\kk)_{+}^{2} \dx\,,\label{scegli2}
\end{flalign}
where $c\equiv c(\data, \eps)$. 
Note that the above choice of $M$ fits \rif{mm} via \rif{prebb}. By \rif{scegli2} we are able to apply Lemma \ref{revlem} with $\kk_{0}=0$, $M_0:= M^{\mu-1+\eps} $ and $M_1=M_2=M_3\equiv 0$, to get
\eqn{3.fz}
$$
E_{\textnormal{i}}(\snr{Dv(x_{0})})\le cM^{\frac{(\mu-1+\eps)\chi}{\chi-1}}\left(\mint_{B_{r_{0}}}[E_{\textnormal{i}}(\snr{Dv})]^{2} \dx\right)^{1/2}\,,
$$
with $c\equiv c(n,q)$. Being $x_{0}$ arbitrary in $B_{\tau_{1}}$, \eqref{3.fz} implies
$$
E_{\textnormal{i}}\left(\nr{Dv}_{L^{\infty}(B_{\tau_{1}})}\right)\le\frac{c\nr{Dv}_{L^{\infty}(B_{\tau_{2}})}^{\frac{(\mu-1+\eps) \chi}{\chi-1}}}{ (\tau_{2}-\tau_{1})^{n/2}}\left(\int_{B_{\tau_2}}[E_{\textnormal{i}}(\snr{Dv})]^{2} \dx\right)^{1/2}
$$
and we have used the actual definition of $M$. Recalling \rif{prebb}, we continue to estimate as
\begin{flalign}
\notag E_{\textnormal{i}}\left(\nr{Dv}_{L^{\infty}(B_{\tau_{1}})}\right)
&  \stackrel{\eqref{0.001}}{\le}\frac{c}{(\tau_{2}-\tau_{1})^{n/2}}\left[E_{\textnormal{i}}\left(\nr{Dv}_{L^{\infty}(B_{\tau_{2}})}\right)\right]^{\frac{(\mu-1+\eps) \chi}{(\chi-1)(2-\mu)}+\frac{1}{2}}\\
& \hspace{12mm}\cdot\left(\int_{B_{5\tau/6}}[E_{\textnormal{i}}(\snr{Dv})]\dx\right)^{1/2}\nonumber \\
&  \stackrel{\eqref{0.002}}{\le}\frac{c}{{(\tau_{2}-\tau_{1})^{n/2}}}\left[E_{\textnormal{i}}\left(\nr{Dv}_{L^{\infty}(B_{\tau_{2}})}\right)\right]^{\frac{(\mu+2\eps) \chi-2+\mu}{2(\chi-1)(2-\mu)}} \notag \\
& \hspace{12mm}\cdot \left(\int_{B_{5\tau/6}}\left[H_{\textnormal{i}}(Dv)+1\right] \dx\right)^{1/2}
\label{ninetta}
\end{flalign}
for $c\equiv c(\data, \eps)$. By choosing $\eps>0$ and $\mu > 1$ such that 
\eqn{sce2}
$$
 \gamma_*(\mu, \eps):=\frac{(\mu+2\eps) \chi-2+\mu}{2(\chi-1)(2-\mu)}<1
$$
holds (note that this function is increasing in both $\mu$ and $\eps$ and that $ \gamma_*(1,0)=1/2$), we can apply Young's inequality in \rif{ninetta}, thereby obtaining 
\begin{flalign*}
E_{\textnormal{i}}\left(\nr{Dv}_{L^{\infty}(B_{\tau_{1}})}\right) & \leq \frac 12 E_{\textnormal{i}}\left(\nr{Dv}_{L^{\infty}(B_{\tau_{2}})}\right) \\
& \qquad +\frac{c} {(\tau_{2}-\tau_{1})^{n\tilde \gamma}}\left(\int_{B_{5\tau/6}}\left[H_{\textnormal{i}}(Dv)+1\right] \dx\right)^{\tilde \gamma}\,,
\end{flalign*}
where 
$$
\tilde \gamma(\mu, \eps):= \frac{(\chi-1)(2-\mu)}{(4-3\mu-2\eps)\chi-2+\mu}\,.
$$
Notice that $\tilde \gamma(1, 0)=1$ and that this is an increasing function of its arguments, therefore we find $\eps>0$ and $\mu > 1$ such that 
\eqn{sce3}
$$
\tilde \gamma(\mu, \eps) \leq 1 +\delta/4
$$
is satisfied in addition to \rif{sce2}. Lemma \ref{l5}, applied with $$h(\mathfrak t)\equiv E_{\textnormal{i}}\left(\nr{Dv}_{L^{\infty}(B_{\mathfrak{t}})}\right), \quad t\equiv 3\tau/4< \mathfrak t < 5\tau/6\equiv s, $$ which is always finite by \rif{0.fz}$_2$, now gives
$$
E_{\textnormal{i}}\left(\nr{Dv}_{L^{\infty}(B_{3\tau/4})}\right)\le c\left(\mint_{B_{\tau}}\left[H_{\textnormal{i}}(Dv)+1\right] \dx\right)^{1 +\delta/4}\,,
$$
so that, recalling \rif{e.fz}, we arrive at 
\eqn{nowg}
$$
E_{\textnormal{i}}\left(\nr{Dv}_{L^{\infty}(B_{3\tau/4})}\right)\le c\left(\mint_{B_{\tau}}\left[H_{\textnormal{i}}(Du_{0})+1\right] \dx\right)^{1 +\delta/4}\,.
$$
Setting $\ti{E}_{0}(t):=t+a_{\sigma, \textnormal{i}}(B)t^{q}$ for $t\geq 0$, we find, also thanks to \rif{0.001}, that
\eqn{findthat}
$$
\ti{E}_{0}(t) \leq c [E_{\textnormal{i}}(t)]^{\frac{1}{2-\mu}}+c\quad \mbox{holds for every $t\geq 0$}\,,
$$
where $c\equiv c (\nu, q)$. 
Recalling \rif{assif}$_1$, the definitions in \rif{defiH} imply
$$
H_{\textnormal{i}}(z)\equiv H_{\omega,\sigma,\textnormal{i}}(z)\equiv H_{\omega,\sigma,\textnormal{i}}(z;B_{\tau})\leq L\eh_{\omega,\sigma,\textnormal{i}}(|z|; B_{\tau})\equiv L\eh_{\textnormal{i}}(|z|), 
$$
so that inequality \rif{nowg} together with \rif{defiHHH} and \rif{findthat} now gives
\begin{flalign*}
\ti{E}_{0}\left(\nr{Dv}_{L^{\infty}(B_{3\tau/4})}\right) &\le 
c \left[\eh_{\textnormal{i}}\left(\nr{Du_{0}}_{L^{\infty}(B_{\tau})}\right)+1\right]^{\frac{1+\delta/4}{2-\mu}}\\
& 
\leq c \left[\eh_{\textnormal{i}}\left(\nr{Du_{0}}_{L^{\infty}(B_{\tau})}\right)+1\right]^{1+\frac{\delta}{2}}
\end{flalign*}
with $c\equiv c(\data, \delta)$, provided we further take $\mu> 1$ such that
\eqn{sce4}
$$
\frac{1+\delta/4}{2-\mu}\leq 1+\frac{\delta}2\,.
$$
Note that in the above inequality we have incorporated the dependence on $\eps$ in the dependence on $\delta$ as this last quantity influences the choice of $\eps$ via \rif{sce2} and \rif{sce3}. Using the very definition \rif{defiH}$_3$ we then continue to estimate
\begin{flalign*}
&\ti{E}_{0}\left(\nr{Dv}_{L^{\infty}(B_{3\tau/4})}\right) \\ & \quad \le c \eh_{\textnormal{i}}^{\delta/2}\left(\nr{Du_{0}}_{L^{\infty}(B_{\tau})}\right)  \nr{Du_{0}}_{L^{\infty}(B_{\tau})} g\left(\nr{Du_{0}}_{L^{\infty}(B_{\tau})}\right)\nonumber \\
& \qquad   + c\eh_{\textnormal{i}}^{\delta/2}\left(\nr{Du_{0}}_{L^{\infty}(B_{\tau})}\right)a_{\sigma, \textnormal{i}}(B_{\tau})\left[\ell_{\omega}\left(\nr{Du_{0}}_{L^{\infty}(B_{\tau})}\right)\right]^{q}+c    
\end{flalign*}
and, by again using \rif{0.1} with $\eps=\delta/2$, we conclude with 
\eqn{7.fz}
$$
\ti{E}_{0}\left(\nr{Dv}_{L^{\infty}(B_{3\tau/4})}\right)\le c\eh_{\textnormal{i}}^{\delta}\left(\nr{Du_{0}}_{L^{\infty}(B_{\tau})}\right) 
\ti{E}_{0}\left(\nr{Du_{0}}_{L^{\infty}(B_{\tau})}\right)+c
$$
with $c\equiv c (\data, \delta)$. Since $\ti{E}_{0}(\cdot)$ is monotone increasing, convex and such that $\ti{E}_{0}(0)=0$, we deduce that its inverse $\ti{E}_{0}^{-1}(\cdot)$ is increasing, concave and $\ti{E}_{0}^{-1}(0)=0$, therefore it is subadditive and, for any given constant $c_{*}\ge 0$ it is $\ti{E}_{0}^{-1}(c_{*}t)\le (c_{*}+1)\ti{E}_{0}^{-1}(t)$, see for instance \cite[Remark 10 in Section 4]{DM}. Using this last property, we can apply $\ti{E}_{0}^{-1}(\cdot)$ to both sides of \eqref{7.fz} in order to conclude with \eqref{8.fz}. Finally, the choice of the function $\mu_n(\delta)$ mentioned in the statement, that we can always take such that $\mu_n(\delta)<3/2$, comes by the choices made in \rif{sce1}, \rif{sce2}, \rif{sce3} and \rif{sce4} and a standard continuity argument (for this recall that both $\gamma_*(\cdot)$ and $\tilde{\gamma}(\cdot)$ are increasing functions of their arguments and that $\chi$ depends on $n$).  Notice that by the above construction it follows that $\mu_n(0_+)=1$. \end{proof}

\section{Proof of Theorem \ref{t3}}\label{dimos}
The proof will take eleven different steps, distributed along Sections \ref{assenza}-\ref{ultimina} below. We shall concentrate on the singular case $\bb=0$, then giving remarks on how to deal with the (actually simpler) case $\bb >0$ at the very end of Section \ref{holsec}.  We shall start assuming that \rif{assif}-\rif{0.1} for some $\mu \in [1, 3/2)$. We shall make further restrictions on the size of $\mu$ in the course of the proof until we finally come to determine the value $\mum$ mentioned in the statement of Theorem \ref{t3}. 
 \subsection{Absence of Lavrentiev phenomenon}\label{assenza}
We have the following approximation-in-energy result that actually implies the absence of Lavrentiev phenomenon:
\begin{lemma}\label{nolav}
With $H(\cdot)$ defined in \trif{hhh0}, assume \eqref{bound}$_1$ and that $F(\cdot)$ satisfies \trif{assif}$_1$ where $g(\cdot)$ is as in Theorem \ref{t3} (described after \trif{assif}). Let $w\in W^{1,1}_{\loc}(\Omega)$ be any function such that $H(\cdot, Dw)\in L^1_{\loc}(\Omega)$. For every ball $B \Subset \Omega$ there exists a sequence $\{w_{\varepsilon}\}\subset W^{1,\infty}(B)$ such that
$w_{\varepsilon}\to w$ in  $W^{1,1}(B)$ and $H(\cdot,Dw_{\varepsilon})\to H(\cdot,Dw)  $ in $L^{1}(B)$.
\end{lemma}
The proof follows \cite[Section 5]{dm} almost verbatim (see also \cite[Lemma 13]{sharp}). Notice that this is essentially the only point where the assumed convexity of $t \mapsto tg(t)$ is used (this implies that $z \mapsto |z|g(|z|)$ is convex as $g(\cdot)$ is also non-decreasing). For more related results on the absence of Lavrentiev phenomenon we refer to the recent papers \cite{AFM, balci, balci2, buli, koch1, koch2} and related references.   
\subsection{Auxiliary Dirichlet problems and convergence}\label{aux}
In the following $u\in W^{1,1}_{\loc}(\Omega)$ denotes a local minimizer of $\mathcal{N}(\cdot)$, as in Theorem \ref{t3}. By $\omega, \eps\equiv\{\omega\},  \{\eps\} \equiv  \{\omega_k\}_k, \{\eps_k\}_k,$ we denote two decreasing sequences of positive numbers such that $\omega, \eps \to 0$, and $\eps,\omega \leq 1$; we will several times extract subsequences and these will still be denoted by $\omega, \eps$ (for this reason we drop the pendice $k$). We denote by $\texttt{o}(\varepsilon)$ a quantity such that $\texttt{o}(\varepsilon)\to 0$ as $\eps \to 0$. Similarly, we denote by $\texttt{o}_{\eps}(\omega)$ a quantity, depending both on $\eps$ and $\omega$, such that $\texttt{o}_{\eps}(\omega)\to 0$ as $\omega \to 0$ for each fixed $\eps$. The exact value of such quantities might change on different occurences and only the aforementioned asymptotic properties will matter. Let $B_{r}\Subset \Omega$ be a ball with $0< r \leq 1$; by Lemma \ref{nolav}, there exists a sequence $\{\ti{u}_{\varepsilon}\}\in W^{1, \infty}(B_{r})$ so that
\eqn{1}
$$
\ti{u}_{\varepsilon}\to u \ \ \mbox{in} \ \ W^{1,1}(B_{r})\quad \ \ \mbox{and}\quad \ \ \mathcal{N}(\ti{u}_{\varepsilon},B_{r}) = \mathcal{N}(u,B_{r})+\texttt{o}(\varepsilon) \,.
$$
We define the sequence
\eqn{2}
$$
\sigma_{\eps}:=\left(1+\varepsilon^{-1}+\nr{D\ti{u}_{\varepsilon}}_{L^{q}(B_{r})}^{2q}\right)^{-1} \ \Longrightarrow \ \sigma_{\eps}\int_{B_{r}}[\ell_{\omega}(D\ti{u}_{\varepsilon})]^{q} \dx \to 0
$$
(uniformly with respect to $\omega \in (0,1]$). 
Then we consider $u_{\omega,\varepsilon}\in \ti{u}_{\varepsilon}+W^{1,q}_{0}(B_{r})$ as the unique solution to the Dirichlet problem
\eqn{pd}
$$
u_{\omega,\varepsilon} \mapsto \min_{w\in \ti{u}_{\varepsilon}+W^{1,q}_{0}(B_{r})} \mathcal{N}_{\omega,\varepsilon}(w,B_{r})
$$
where 
\eqn{approssimaF}
$$
\mathcal{N}_{\omega,\varepsilon}(w,B_{r}):=\int_{B_{r}} H_{\omega,\sigma_{\eps}}(x,Dw) \dx\,.
$$
Recall that the integrand $H_{\omega,\sigma_{\eps}}(\cdot)$ has been defined in \rif{defiH}, with $\sigma\equiv \sigma_{\eps}$. 
The solvability of \rif{pd} follows by Direct Methods and standard convexity arguments. By \eqref{regecor}, we can apply the by now classical regularity theory contained in \cite[Chapter 8]{giu} and \cite{manth1, manth2}, therefore
\eqn{3}
$$
u_{\omega,\varepsilon}\in C^{1,\beta}_{\loc}(B_{r})\quad \mbox{for some} \ \ \beta\equiv \beta(n,
\nu, L, q,\omega, \varepsilon)\in (0,1)\,.
$$
Mean value theorem and \rif{2} imply
\begin{flalign}
\notag |\mathcal{N}_{\omega,\varepsilon}(\ti{u}_{\varepsilon},B_{r})-\mathcal{N}(\ti{u}_{\varepsilon},B_{r})| & \leq
 c \omega \int_{B_{r}} (|D\ti{u}_{\varepsilon}|^{q-1}+1)\dx \notag
 \\ \notag & \qquad +    \sigma_{\eps}\int_{B_{r}}[\ell_{\omega}(D\ti{u}_{\varepsilon})]^{q} \dx \\
 &  =  \texttt{o}_{\eps}(\omega)+\texttt{o}(\varepsilon) \,.
 \label{22}
\end{flalign}
Similarly, noting that $|z|^{q-1} \leq |z| +1 \leq c F(x,z) +c$ by \rif{assif}$_1$ and $q <3/2$ (follows from \rif{bound}$_1$), we find, using also mean value theorem
\begin{flalign}
\notag  & \left|\mathcal{N}_{\omega,\varepsilon}(u_{\omega, \varepsilon},B_{r})-\mathcal{N}(u_{\omega, \varepsilon},B_{r})- \sigma_{\eps}\int_{B_{r}}[\ell_{\omega}(Du_{\omega, \varepsilon})]^{q} \dx\right| \\ 
& \quad \quad \leq
 c \omega \int_{B_{r}}( |Du_{\omega, \varepsilon}|^{q-1}+1)\dx \leq c \omega  \mathcal{N}_{\omega,\varepsilon}(u_{\omega, \varepsilon},B_{r})+c\omega \,.
 \label{22d}
\end{flalign}
In turn, using in order: the minimality of $u_{\omega,\varepsilon}$, \rif{1} and \rif{22}, we have 
\begin{flalign}
\mathcal{N}_{\omega,\varepsilon}(u_{\omega,\varepsilon},B_{r})&\le \mathcal{N}_{\omega,\varepsilon}(\ti{u}_{\varepsilon},B_{r})\nonumber \\
&\le\mathcal{N}(\ti{u}_{\varepsilon},B_{r})+\snr{\mathcal{N}_{\omega,\varepsilon}(\ti{u}_{\varepsilon},B_{r})-\mathcal{N}(\ti{u}_{\varepsilon},B_{r})}\nonumber \\
&\leq \mathcal{N}(u,B_{r})+\texttt{o}_{\eps}(\omega)+\texttt{o}(\varepsilon) \label{24}
\end{flalign} 
so that, using the content of the last two displays we gain 
\begin{flalign} 
 \notag & \mathcal{N}(u_{\omega, \varepsilon},B_{r}) +\sigma_{\eps}\int_{B_{r}}[\ell_{\omega}(Du_{\omega, \varepsilon})]^{q} \dx \\
 & \qquad \ \  = \mathcal{N}_{\omega,\varepsilon}(u_{\omega, \varepsilon},B_{r}) + \texttt{o}_{\eps}(\omega)+\texttt{o}(\varepsilon) + c\omega \,.
\label{24.5}
\end{flalign} 
Estimate \eqref{24} and \rif{rege.2}$_1$ imply that for every $\varepsilon\in (0,1)$ the sequence $\{u_{\omega,\eps}\}_{\omega}$ is uniformly bounded in $W^{1,q}(B_{r})$, therefore, up to not relabelled subsequences, we have
\eqn{23}
$$
u_{\omega,\eps}\rightharpoonup u_{\varepsilon}\ \ \mbox{weakly in} \ \ W^{1,q}(B_{r})\quad \mbox{and}\quad u_{\varepsilon}-\ti{u}_{\varepsilon} \in W^{1,q}_0(B_{r})
$$
as $\omega \to 0$. Letting $\omega \to 0$ in \rif{24.5} and using standard weak lower semicontinuity theorems, yields
$$
\mathcal{N}(u_{\varepsilon},B_{r}) \leq \liminf_{\omega \to 0} \mathcal{N}(u_{\omega, \varepsilon},B_{r})\leq  \liminf_{\omega \to 0}\mathcal{N}_{\omega,\varepsilon}(u_{\omega,\varepsilon},B_{r}) +\texttt{o}(\varepsilon) 
$$ 
for every fixed $\eps \in (0,1)$. 
Using \rif{24} we conclude with 
\eqn{25}
$$
\mathcal{N}(u_{\varepsilon},B_{r})  \le\mathcal{N}(u,B_{r})+\texttt{o}(\varepsilon)
$$
and again this holds for every $\eps\in (0,1)$. 
By \rif{assif}$_1$ and \eqref{25} the sequence $\{\snr{Du_{\varepsilon}}g(Du_{\varepsilon})\}$ is uniformly bounded in $L^{1}(B_{r})$. Recalling that the assumptions on $g(\cdot)$ imply that $g(t)\to \infty$ as $t\to \infty$, by classical results of Dunford \& Pettis and de la Vall\'ee
Poussin, there exists $\hat{u}\in W^{1,1}(B_{r})$ such that
$
u_{\varepsilon}\rightharpoonup \hat{u}$ weakly in $W^{1,1}(B_{r})$ and $\hat{u}-u \in W^{1,1}_0(B_{r})$. 
Letting $\varepsilon\to 0$ in \eqref{25} weak lower semicontinuity (see \cite[Theorem 4.3]{giu}) and \rif{2} yield $\mathcal{N}(\hat{u},B_{r})\le \mathcal{N}(u,B_{r})$, while the opposite inequality follows by the minimality of $u$. We conclude with $\mathcal{N}(\hat{u},B_{r})= \mathcal{N}(u,B_{r})$, so that, by strict convexity of the functional $w \mapsto \mathcal{N}(w,B_{r})$ we find that $u\equiv \hat{u}$ in $B_{r}$ and we deduce that
\eqn{27}
$$
 u_{\varepsilon}\rightharpoonup u \ \ \mbox{weakly in} \ \ W^{1,1}(B_{r})\,.
$$
\subsection{Blow-up}
We fix $\omega, \varepsilon \in (0,1]$ and $u_{\omega,\varepsilon}\in W^{1,q}(B_{r})$ as in \eqref{pd}. With $B_{\rr}(x_{0})\Subset B_{r}$ being a ball not necessarily concentric to $B_{r}$, we take $\M$ such that
\eqn{mmm}
$$
\mathfrak{M}\ge \nr{\ti{E}_{\omega,\sigma_{\eps}}(\cdot,\snr{Du_{\omega,\varepsilon}})}_{L^{\infty}(B_{\rr}(x_{0}))}+ 1\,,
$$
where $\ti E_{\omega,\sigma_{\eps}}(\cdot)$ is defined in \rif{eiit}. The above quantities are finite by \rif{3}. 
We rescale $u$ and $H(x,z)=F(x,z)+a(x)|z|^q$ on $B_{\rr}(x_{0})$ defining
\eqn{notational0}
$$\begin{cases} 
\ \displaystyle u_{\omega, \varepsilon,\rr}(x):= u_{\omega, \eps}(x_{0}+\rr x)/\rr\\
\ F_{\rr}(x,z):= F(x_0+\rr x,z),\quad   \mathcal{a}_{\rr}(x):= a(x_0+\rr x)\\
\ \mathcal H_{\rr}(x, z) := H(x_0+\rr x, z)= F_{\rr}(x, z)+ \mathcal{a}_{\rr}(x)|z|^q\,,
\end{cases}
$$
with $(x,z) \in \mathcal B_{1}\times  \er^n$. Note that, obviously, $\mathcal H_{\rr}(\cdot)$ is still an integrand of the type in \rif{hhh0} (see also \rif{vedisotto} below) and therefore the content of Section \ref{const} applies to $\mathcal H_{\rr}(\cdot)$ as well. Since $u_{\omega, \eps}$ solves \eqref{pd}, recalling the notation fixed in \rif{defiH}, it follows that $u_{\omega, \varepsilon,\rr}\in W^{1,q}(\mathcal B_{1})$ is a local minimizer on $\mathcal B_{1}$ of the functional
$$
W^{1,q}(\mathcal B_{1})\ni w\mapsto \int_{\mathcal B_{1}}(\mathcal H_{\rr})_{\omega, \sigma_\eps}(x,Dw) \dx
$$
where, recalling the notation in \rif{aii}$_1$, with $(x,z) \in \mathcal B_{1}\times  \er^n$ it is 
\eqn{vedisotto}
$$
\begin{cases}
(\mathcal H_{\rr})_{\omega, \sigma_\eps}(x,z)=F_{\rr}(x, z)+(\mathcal{a}_{\rr})_{\sigma_{\eps}}(x)[\ell_{\omega}(z)]^{q}\\
(\mathcal{a}_{\rr})_{\sigma_{\eps}}(x)= \mathcal{a}_{\rr}(x)+\sigma_{\eps}=
a(x_0+\rr x)+\sigma_{\eps} \,.
 \end{cases}
$$
From now on, keeping fixed the choice of $\omega, \eps$ made at the beginning, in order to simplify the notation we shall omit to specify dependence on such parameters, simply abbreviating
\eqn{abb}
$$
u_{\rr}(x)\equiv u_{\omega, \varepsilon,\rr}(x), \quad  
H_{\rr}(x, z)\equiv (\mathcal H_{\rr})_{\omega, \sigma_\eps}(x,z)
$$
for $(x,z) \in \mathcal B_1\times \er^n$. 
The minimality of $u_{\rr}$ implies the validity of the Euler-Lagrange equation
\eqn{el}
$$
\int_{\mathcal B_{1}}\langle\partial_z  H_{\rr}(x,Du_{\rr}), D\varphi \rangle \dx=0\quad \mbox{for all} \ \ \varphi\in W^{1,q}_{0}(\mathcal B_{1})\,.
$$
By \rif{rege.2} the integrand $H_{\rr}(\cdot)$ satisfies
\eqn{assr}
$$
\begin{cases}
\ \lambda_{\rr}(x,\snr{z})\snr{\xi}^{2}\le c\langle\partial_{zz}H_{\rr}(x,z)\xi,\xi\rangle\\
\  \snr{ \partial_{zz}H_{\rr}(x,z)}\le c\Lambda_{\rr}(x,\snr{z})
\\
\ \snr{\partial_z  H_{\rr}(x,z)-\partial_z  H_{\rr}(y,z)}\\
\qquad \le c\rr^{\ao}\snr{x-y}^{\ao}g(|z|)+ c \rr^{\alpha}\snr{x-y}^{\alpha}[\ell_{\omega}(z)]^{q-1}
\end{cases}
$$
for any $x,y\in \mathcal B_{1}$ and all $z,\xi\in \mathbb{R}^{n}$, where $c\equiv c (\data)$ and, according to the definitions in \rif{autovalori} and the notation in \rif{abb}, we are denoting
\eqn{nuoviaut0}
$$
\lambda_{\rr}(x,\snr{z}):=\lambda_{\omega,\sigma_\varepsilon}(x_{0}+\rr x,\snr{z})\,, \quad \Lambda_{\rr}(x,\snr{z}):=\Lambda_{\omega,\sigma_\varepsilon}(x_{0}+\rr x,\snr{z})\,.
$$
\subsection{Minimal integrands}\label{minisec}
Here we are going to play with auxiliary functionals whose integrands are of the type in \rif{defiH}$_2$, and therefore featuring no explicit dependence on $x$ (these are usually called ``frozen'' integrands). The results of Section \ref{rere} can be therefore applied. Let us fix a number $\beta_{0}\in (0,1)$, to be determined in a few lines, and a vector $h\in \mathbb{R}^{n}\setminus \{0\}$ such that
\eqn{hhh}
$$
0< \snr{h}\leq 
\frac{1}{2^{8/\beta_{0}}}\,.
$$
We take $x_{\rm c}\in \mathcal B_{1/2+ 2|h|^{\beta_0}}$ and fix a ball centered at $x_{\rm c}$ with radius $\snr{h}^{\beta_{0}}$, denoted by $B_h\equiv  B_{\snr{h}^{\beta_{0}}}(x_{\rm c})$. By \eqref{hhh} we have $8B_h\Subset \mathcal B_{1}$. We set
\eqn{mmm0} 
$$
\mm\equiv \mm(8B_h):= \nr{Du_{\rr}}_{L^{\infty}(8B_h)}+1\,.
$$
According to the notation established in \rif{aii}, \rif{defiH} and \rif{vedisotto}-\rif{abb}, we define
\eqn{deffiH}
$$
\begin{cases} 
H_{\rr,\textnormal{i}}(z)\equiv (\mathcal H_{\rr})_{\omega, \sigma_{\eps}, \textnormal{i}}(z;8B_h)\equiv  F_{\rr}(x_{\rm c}, z)+\tilde a_{\rr, \textnormal{i}}(8B_h)[\ell_{\omega}(z)]^{q}\\
\displaystyle \tilde a_{\rr, \textnormal{i}}(8B_h):=(\mathcal{a}_{\rr})_{\sigma_{\eps}, \textnormal{i}}(8B_h)=\inf_{x\in 8B_h}(\mathcal{a}_{\rr})_{ \sigma_{\eps}}(x)=\inf_{x\in 8B_h}a(x_0+\rr x)+\sigma_{\eps}\,.
\end{cases}
$$
Note that $\partial_z  H_{\rr,\textnormal{i}}\in C^{1}(\mathbb{R}^{n};\mathbb{R}^{n})$. As in \rif{ill}, with
      $$
  \begin{cases}
\lambda_{\rr,\textnormal{i}}(\snr{z}) \displaystyle \equiv \lambda_{\rr,\textnormal{i}}(\snr{z}; 8B_h):=(|z|^2+1)^{-\mu/2}+ (q-1)\tilde a_{\rr, \textnormal{i}}(8B_h)[\ell_{\omega}(z)]^{q-2}\\
 \Lambda_{\rr,\textnormal{i}}(\snr{z}) \equiv \Lambda_{\rr,\textnormal{i}}(\snr{z};8B_h):=(|z|^2+1)^{-1/2}g(|z|)+\tilde a_{\rr, \textnormal{i}}(8B_h)[\ell_{\omega}(z)]^{q-2}\,,
    \end{cases}
    $$
from \rif{0.1} and \rif{rege.2i} it follows that
\eqn{5}
$$
\begin{cases}
\ 
\sigma_{\eps} [\ell_{\omega}(z)]^{q}\le H_{\rr,\textnormal{i}}(z)\le c\left([\ell_{\omega}(z)]^{q}+1\right)\\
\  \snr{\partial_z H_{\rr,\textnormal{i}}(z)}\le c\left([\ell_{\omega}(z)]^{q-1}+1\right)\\
\ \lambda_{\rr,\textnormal{i}}(\snr{z})\snr{\xi}^{2}\le c\langle \partial_{zz}H_{\rr,\textnormal{i}}(z)\xi,\xi\rangle, \quad \snr{\partial_{zz}H_{\rr,\textnormal{i}}(z)}\le c \Lambda_{\rr,\textnormal{i}}(\snr{z})
\\
\ \mathcal{V}_{\rr, \textnormal{i}}^{2}(z_{1},z_{2};8B_h)\leq c\langle \partial_z  H_{\rr,\textnormal{i}}(z_{1})-\partial_z  H_{\rr,\textnormal{i}}(z_{2}),z_{1}-z_{2}\rangle
\end{cases}
$$
hold for all $z,z_{1},z_{2},\xi\in \mathbb{R}^{n}$, with $c\equiv c(\data)$. Consistently with \rif{vvv}, here we are denoting 
\begin{flalign}
\notag \mathcal{V}_{\rr, \textnormal{i}}^{2}(z_1,z_2;8B_h) &:=
\snr{V_{1,2-\mu}(z_{1})-V_{1,2-\mu}(z_{2})}^{2}\\
& \qquad +\tilde a_{\rr, \textnormal{i}}(8B_h)\snr{V_{\omega,q}(z_{1})-V_{\omega,q}(z_{2})}^{2}\,.\label{deffiV}
\end{flalign}
Recalling \rif{eii} and \rif{nuoviaut0}, we define, for $t\geq 0$ and $x \in \mathcal B_1$
\eqn{eiin}
$$
\begin{cases}
\displaystyle E_{\rr}(x,t):=\int_{0}^{t}\lambda_{\rr}(x,s)s\ds\\
\displaystyle E_{\rr,\textnormal{i}}(t)\equiv  E_{\rr,\textnormal{i}}(t;8B_h):=\int_{0}^{t}\lambda_{\rr,\textnormal{i}}(s;8B_h)s\ds
\end{cases}
$$
with related explicit expressions as in \rif{eii2} and
$$
\begin{cases}
\ \ti E_{\rr}(x,t):= \frac{1}{2-\mu}[\ell_{1}(t)]^{2-\mu}+(1-1/q) (\mathcal{a}_{\rr})_{\sigma_{\eps}}(x)[\ell_{\omega}(t)]^{q}\\
\ \ti E_{\rr,\textnormal{i}}(t):= \frac{1}{2-\mu}[\ell_{1}(t)]^{2-\mu}+(1-1/q) \tilde a_{\rr, \textnormal{i}}(8B_h)[\ell_{\omega}(t)]^{q}\,.
\end{cases}
$$
By \rif{0.00} we have 
\eqn{0.00bis}
$$
\snr{E_{\rr}(x,t)- E_{\rr,\textnormal{i}}(t)}\le  c|h|^{\alpha\beta_0} \rr^{\alpha}[\ell_{\omega}(t)]^{q}\quad \mbox{for every $x \in 8B_h$ and $t \geq 0$}
$$
and \rif{mmm}  implies
\eqn{3.1M}
$$
 \mathfrak{M}\ge \nr{\ti E_{\rr}(\cdot,\snr{Du_{\rr}})}_{L^{\infty}(\mathcal B_1)}+ 1\,.
$$
Looking at \rif{5}, Direct Methods and strict convexity provide us with a unique minimizer $v\in  u_{\rr}+W^{1,q}_{0}(8B_h)$ defined by
\eqn{atomi} 
$$
 v \equiv v (8B_h)  \mapsto  \min_{w\in u_{\rr}+W^{1,q}_{0}(8B_h)}\int_{8B_h}H_{\rr,\textnormal{i}}(Dw) \dx
$$
so that, \rif{5}$_1$ obviously implies
\eqn{el2}
$$
\int_{8B_h}\langle\partial_z  H_{\rr,\textnormal{i}}(Dv), D\varphi\rangle \dx=0\qquad \mbox{for all} \ \ \varphi\in W^{1,q}_{0}(8B_h)
$$
and minimality gives
\eqn{enes}
$$
\int_{8B_h}H_{\rr,\textnormal{i}}(Dv) \dx \le \int_{8B_h}H_{\rr,\textnormal{i}}(Du_{\rr}) \dx\,.
$$
\begin{lemma}\label{le2} Let $\delta_{1} \in (0,1)$ and let $\mu_n(\cdot)$ be the (non-decreasing) function introduced in Proposition \ref{fzfz}. If $1\leq \mu < \mu_n(\delta_{1}/2)$, then the inequalities 
\eqn{stim3}
$$
 \nr{Dv}_{L^{\infty}(6B_h)}\le c \mm^{1+\delta_{1}}
$$
and 
\eqn{caccada}
$$
 \int_{2B_h}\snr{D(E_{\rr,\textnormal{i}}(\snr{Dv})-\kk)_{+}}^{2}\dx \le 
 \frac{c\, \mm^{\delta_{1}}}{\snr{h}^{2\beta_{0}}}\int_{4B_h}(E_{\rr,\textnormal{i}}(\snr{Dv})-\kk)_{+}^{2} \dx
$$
hold with $c\equiv c(\data,\delta_{1})$ and $\mm$ is defined in \eqref{mmm0}. 
\end{lemma}
\vspace{1mm}
\begin{proof}
The integrand $H_{\rr,\textnormal{i}}(\cdot)$ is of the type considered in Proposition \ref{fzfz} (compare with \rif{deffiH} and \rif{5}), and we apply this last result to $v$, with $B_\tau\equiv 8B_h$. It follows that for every $\delta_{1} >0$ 
\begin{eqnarray}
\notag \nr{Dv}_{L^{\infty}(6B_h)} &\stackleq{8.fz} & c \eh_{\rr,\textnormal{i}}^{\delta_{1}/q}\left(\nr{Du_{\rr}}_{L^{\infty}(8B_h)}\right)\nr{Du_{\rr}}_{L^{\infty}(8B_h)}+c\\
&\stackleq{mmm0} & c \eh_{\rr,\textnormal{i}}^{\delta_{1}/q}(\mathfrak{m})\mathfrak{m}+c \label{9.fz}
\end{eqnarray}
holds provided $1\leq \mu< \mu_n(\delta_{1}/2)\leq  \mu_n(\delta_{1}/q)$, where $c\equiv c(\data,\delta_{1})$. Here, as in \rif{defiH}$_3$, it is 
$
\eh_{\rr,\textnormal{i}}(t) := tg(t)+\tilde a_{\rr, \textnormal{i}}(8B_h)[t^2+\omega^2]^{q/2}+1
$ for  $t\geq 0$ and in fact in \rif{9.fz} we have used \rif{defiHHH}. 
From \rif{9.fz} we can derive \rif{stim3} using \rif{0.1} with $\eps =q-1>0$. Next, we apply \rif{10.fz2}, that gives
\begin{flalign*}
  &\int_{2B_h}\snr{D(E_{\rr,\textnormal{i}}(\snr{Dv})-\kk)_{+}}^{2}\dx \\
  & \qquad \le 
 \frac{c\left(\|Dv\|_{L^\infty(4B_h)}+1\right)^{\delta_{1}/2}}{\snr{h}^{2\beta_{0}}}\int_{4B_h}(E_{\rr,\textnormal{i}}(\snr{Dv})-\kk)_{+}^{2} \dx
\end{flalign*}
so that \rif{caccada} follows using \rif{stim3} in the above inequality and observing that $(1+\delta_1)\delta_1/2 \leq \delta_1$. 
\end{proof}
\subsection{A comparison estimate}\label{hy} This is in the following:
\begin{lemma}\label{comparazione} Let $u_{\rr} \in W^{1,q}(\mathcal B_1)$ be as in \trif{notational0}, $v \in u_{\rr}+W^{1,q}_{0}(8B_h)$ as in \trif{atomi} and $\delta_{2} \in (0,1/2)$. The inequality 
\begin{flalign}\label{10}
\nonumber   \int_{8B_h}\mathcal{V}_{\rr, \textnormal{i}}^{2}(Du_{\rr},Dv;8B_h)  \dx &  \le c\snr{h}^{\beta_{0}\tia}\M^{\frac{1-\delta_{2}/2}{2-\mu}}\rr^{\alpha}\int_{8B_h}(|Du_{\rr}|+1)^{q-1+\delta_{2}} \dx\\
& \quad \ +c\snr{h}^{\beta_{0}\tia}\M^{\frac{1-\delta_{2}/2}{2-\mu}} \rr^{\alpha_0} \int_{8B_h}(|Du_{\rr}|+1)^{3\delta_{2}} \dx
\end{flalign}
holds with $c\equiv c(\data,\delta_{2})$, where
$\tia =\min\{\alpha, \alpha_0\}$, 
$\mathcal{V}_{\rr, \textnormal{i}}(\cdot)$ is as in \trif{deffiV}, and $\M$ is any number satisfying \trif{mmm} and therefore \eqref{3.1M}. 
\end{lemma}
\begin{proof} From \rif{mmm0} and \rif{3.1M}, and yet recalling \rif{0.001}, we deduce
\eqn{4.1}
$$
\begin{cases}
\ \tilde a_{\rr, \textnormal{i}}(8B_h)[\ell_{\omega}(\mathfrak{m})]^{q}\le c\nr{(\mathcal a_{\rr})_{\sigma_{\eps}}(\cdot)[\ell_{\omega}(Du_{\rr})]^{q}}_{L^{\infty}(8B_h)}+c\le c\mathfrak{M}\\
 \ \mm\leq \ell_{\omega}(\mm)\le c(n,q)\M^{\frac{1}{2-\mu}}\,,
\end{cases}
$$
where $c\equiv c(\mu, q, \|a\|_{L^\infty})$ in \rif{4.1}$_1$.  
We have
\begin{flalign*}
&\int_{8B_h}\mathcal{V}_{\rr, \textnormal{i}}^{2}(Du_{\rr},Dv;8B_h) \dx \\
&\quad \stackrel{\eqref{5}_{4}}{\le} c\int_{8B_h}\langle \partial_z  H_{\rr,\textnormal{i}}(Du_{\rr})-\partial_z  H_{\rr,\textnormal{i}}(Dv),Du_{\rr}-Dv\rangle \dx\nonumber \\
&\quad \stackrel{\eqref{el2}}{=} c\int_{8B_h}\langle\partial_z  H_{\rr,\textnormal{i}}(Du_{\rr}),Du_{\rr}-Dv\rangle \dx\nonumber \\
&\quad \stackrel{\eqref{el}}{=} c\int_{8B_h}\langle\partial_z  H_{\rr,\textnormal{i}}(Du_{\rr})-\partial_z  H_{\rr}(x,Du_{\rr}),Du_{\rr}-Dv\rangle \dx\nonumber \\
&\ \ \ \, \stackrel{\eqref{assr}_{3}}{\le}c\snr{h}^{\beta_{0}\alpha}\rr^{\alpha}\int_{8B_h}[\ell_{\omega}(Du_{\rr})]^{q-1}(\snr{Du_{\rr}}+\snr{Dv}) \dx\nonumber \\
&\quad \qquad \quad  +c\snr{h}^{\beta_{0}\alpha_0}\rr^{\alpha_0}\int_{8B_h}g(|Du_{\rr}|)(\snr{Du_{\rr}}+\snr{Dv})\dx\nonumber \\
&\quad \stackrel{\eqref{0.1}}{\le}  c\snr{h}^{\beta_{0}\tia}(\mm^{q-1}\rr^{\alpha}+\mm^{\delta_{2}}\rr^{\alpha_0})\int_{8B_h}(\snr{Du_{\rr}}+\snr{Dv}) \dx\nonumber \\
&\quad \stackrel{\eqref{assif}}{\le}  c\snr{h}^{\beta_{0}\tia}(\mm^{q-1}\rr^{\alpha}+\mm^{\delta_{2}}\rr^{\alpha_0})\int_{8B_h}[H_{\rr,\textnormal{i}}(Du_{\rr})+H_{\rr,\textnormal{i}}(Dv)] \dx\nonumber \\
&\quad \stackrel{\eqref{enes}}{\le}  c\snr{h}^{\beta_{0}\tia}(\mm^{q-1}\rr^{\alpha}+\mm^{\delta_{2}}\rr^{\alpha_0})\int_{8B_h}H_{\rr,\textnormal{i}}(Du_{\rr}) \dx\nonumber \\
&\quad \stackrel{\eqref{assif}}{\le} c\snr{h}^{\beta_{0}\tia}(\mm^{q-1}\rr^{\alpha}+\mm^{\delta_{2}}\rr^{\alpha_0})\\
&\hspace{13mm} \cdot  \int_{8B_h}(\snr{Du_{\rr}}g(\snr{Du_{\rr}})+\tilde a_{\rr, \textnormal{i}}(8B_h)\snr{Du_{\rr}}^{q}+1) \dx \\
&\quad \stackrel{\eqref{0.1}}{\le} c\snr{h}^{\beta_{0}\tia}(\mm^{q-1}\rr^{\alpha}+\mm^{\delta_{2}}\rr^{\alpha_0}) \int_{8B_h}(\snr{Du_{\rr}}^{1+\delta_{2}/2}+\tilde a_{\rr, \textnormal{i}}(8B_h)\snr{Du_{\rr}}^{q}) \dx \\
& \quad \qquad \quad+ c\snr{h}^{\beta_{0}\tia}(\mm^{q-1}\rr^{\alpha}+\mm^{\delta_{2}}\rr^{\alpha_0})|B_h|\,,
\end{flalign*}
where $c\equiv c(\data,\delta_{2})$. We now estimate the four integrals stemming from the second-last line in the above display. As $\delta_{2}<  1/2$, we have
\begin{flalign*}
 \mm^{\delta_{2}}\rr^{\alpha_0} \int_{8B_h}\snr{Du_{\rr}}^{1+\delta_{2}/2}\dx  & = 
 \mm^{\delta_{2}}\rr^{\alpha_0} \int_{8B_h}\snr{Du_{\rr}}^{1-2\delta_{2}}\snr{Du_{\rr}}^{5\delta_{2}/2}\dx\\
& \leq  c\mm^{1-\delta_{2}}\rr^{\alpha_0} \int_{8B_h}\snr{Du_{\rr}}^{5\delta_{2}/2}\dx\\
 &\leq c\M^{\frac{1-\delta_{2}/2}{2-\mu}}\rr^{\alpha_0} \int_{8B_h}(|Du_{\rr}|+1)^{3\delta_{2}} \dx
\end{flalign*}
and in the last line we have used \rif{4.1}$_2$. Similarly, this time using \rif{4.1}$_1$, we find 
\begin{flalign*}
 & \mm^{\delta_{2}}\rr^{\alpha_0} \int_{8B_h}\tilde a_{\rr, \textnormal{i}}(8B_h)\snr{Du_{\rr}}^{q}\dx  \\ & \qquad  \leq  
c[\tilde a_{\rr, \textnormal{i}}(8B_h)]^{1-\delta_{2}/q}\mm^{\delta_{2}} \rr^{\alpha_0}\int_{8B_h}\snr{Du_{\rr}}^{q-2\delta_{2}}\snr{Du_{\rr}}^{2\delta_{2}} \dx\\
& \qquad \leq  c[\tilde a_{\rr, \textnormal{i}}(8B_h)]^{1-\delta_{2}/q} \mm^{q\left(1-\delta_{2}/q\right)} \rr^{\alpha_0}\int_{8B_h}\snr{Du_{\rr}}^{2\delta_{2}}\dx\\
 &\qquad \leq c\M^{1-\delta_{2}/q} \rr^{\alpha_0}\int_{8B_h}|Du_{\rr}|^{2\delta_{2}} \dx\\
 &\qquad \leq c\M^{\frac{1-\delta_{2}/2}{2-\mu}}\rr^{\alpha_0} \int_{8B_h}(|Du_{\rr}|+1)^{3\delta_{2}} \dx\,.
\end{flalign*}
Next, note that $\delta_{2} < 1/2$ and $q<3/2$ (follows from \rif{bound}$_1$) implies $2-q-\delta_{2}/2>0$ and therefore we can estimate
\begin{flalign*}
 \mm^{q-1}\rr^{\alpha} \int_{8B_h}\snr{Du_{\rr}}^{1+\delta_{2}/2}\dx  & = \mm^{q-1}\rr^{\alpha} \int_{8B_h}\snr{Du_{\rr}}^{2-q-\delta_{2}/2}
 \snr{Du_{\rr}}^{q-1+\delta_{2}}\dx\\
& \leq  c\mm^{1-\delta_{2}/2}\rr^{\alpha} \int_{8B_h}\snr{Du_{\rr}}^{q-1+\delta_{2}}\dx\\
 &\leq c\M^{\frac{1-\delta_{2}/2}{2-\mu}}\rr^{\alpha}\int_{8B_h}|Du_{\rr}|^{q-1+\delta_{2}} \dx
\end{flalign*}
where we again used \rif{4.1}$_2$. By means of \rif{4.1}$_1$  we find
\begin{flalign*}
 & \mm^{q-1}\rr^{\alpha} \int_{8B_h}\tilde a_{\rr, \textnormal{i}}(8B_h)\snr{Du_{\rr}}^{q}\dx  \\ & \qquad \leq  
c[\tilde a_{\rr, \textnormal{i}}(8B_h)]^{1-\delta_{2}/q}\mm^{q-1} \rr^{\alpha}\int_{8B_h}\snr{Du_{\rr}}^{1-\delta_{2}}\snr{Du_{\rr}}^{q-1+\delta_{2}} \dx\\
& \qquad \leq  c[\tilde a_{\rr, \textnormal{i}}(8B_h)]^{1-\delta_{2}/q} \mm^{q\left(1-\delta_{2}/q\right)}\rr^{\alpha} \int_{8B_h}\snr{Du_{\rr}}^{q-1+\delta_{2}} \dx\\
 &\qquad \leq c\M^{1-\delta_{2}/q} \rr^{\alpha}\int_{8B_h}|Du_{\rr}|^{q-1+\delta_{2}} \dx\\
 &\qquad \leq c\M^{\frac{1-\delta_{2}/2}{2-\mu}}\rr^{\alpha}\int_{8B_h}|Du_{\rr}|^{q-1+\delta_{2}} \dx\,.
\end{flalign*}
Finally, observe that $\delta_{2}< 2-q$ as $q < 3/2$ so that
$$
\mm^{q-1}\rr^{\alpha}+\mm^{\delta_{2}}\rr^{\alpha_0} \leq c\M^{\frac{1-\delta_{2}}{2-\mu}}(\rr^{\alpha}+\rr^{\alpha_0}) \leq \M^{\frac{1-\delta_{2}/2}{2-\mu}}(\rr^{\alpha}+\rr^{\alpha_0}) \,.
$$
Merging the content of the last six displays yields
\rif{10} with the asserted dependence of the constant $c$. Notice that the dependence of the constants on $\delta_{2}$ comes from \rif{0.1} (that has been used with $\eps = \delta_{2}/2$). Notice also that \rif{10} holds whenever we are assuming \rif{assif} with $\mu \in [1, 3/2)$. 
\end{proof}
\subsection{A fractional Caccioppoli inequality via nonlinear \\ atomic type decompositions}\label{hy2} 
\begin{lemma}[Fractional Caccoppoli inequality]\label{fracacc}
Let $u_{\omega,\varepsilon}\in W^{1,q}(B_{r})$ be as in \eqref{pd}; fix numbers $\delta_{1}\in (0,1)$, $\delta_{2} \in (0,1/2)$
and a ball $B_{\rr}(x_{0})\Subset B_{r}$. The inequality
\begin{flalign}\label{16}
\notag & \left(\mint_{B_{\rr/2}(x_{0})}(E_{\omega,\sigma_{\eps}}(x,\snr{Du_{\omega,\varepsilon}})-\kk)_{+}^{2\chi} \dx\right)^{1/\chi}\\
& \quad \quad   +\rr^{2\beta-n}[(E_{\omega,\sigma_{\eps}}(\cdot,\snr{Du_{\omega,\varepsilon}})-\kk)_{+}]_{\beta,2;B_{\rr/2}(x_{0})}^2\nonumber \\
&\quad \le c\M^{2\ssss_{1}}\mint_{B_{\rr}(x_{0})}(E_{\omega,\sigma_{\eps}}(x,\snr{Du_{\omega,\varepsilon}})-\kk)_{+}^{2} \dx\nonumber\\
& \notag \qquad  \quad +c\M^{2\ssss_{2}}\rr^{2\alpha}\mint_{B_{\rr}(x_{0})}(|Du_{\omega, \varepsilon}|+1)^{2(q-1+\delta_{2})} \dx\nonumber \\
& \notag  \qquad \quad  +c\M^{2\ssss_{3}}\rr^{\alpha}\mint_{B_{\rr}(x_{0})}(|Du_{\omega, \varepsilon}|+1)^{q-1+\delta_{2}} \dx \\
& \qquad \quad 
+c\M^{2\ssss_{3}}\rr^{\alpha_0}\mint_{B_{\rr}(x_{0})}(|Du_{\omega, \varepsilon}|+1)^{3\delta_{2}} \dx
\end{flalign}
holds whenever
\eqn{ilbeta2}
$$ \displaystyle \beta\in (0,\alpha_{*})\ \  \mbox{with $\chi(\beta):=\frac{n}{n-2\beta}, $\ \ \ $\alpha_{*}:=\frac{\tia}{\tia+2}=\frac{\min\{\alpha, \alpha_0\}}{\min\{\alpha, \alpha_0\}+2}$} \,,$$
and provided 
\eqn{sce5} 
$$\begin{cases}
\, 1\leq \mu < \mu_{n}(\delta_{1}/2)\\
 \, \nr{\ti{E}_{\omega,\sigma_{\eps}}(\cdot,\snr{Du_{\omega,\varepsilon}})}_{L^{\infty}(B_{\rr}(x_{0}))}+ 1 \leq  \mathfrak{M}\,,
 \end{cases}
 $$
where $c\equiv c (\data, \delta_{1}, \delta_{2}, \beta)$ and \eqn{sonot}
$$
\begin{cases}
\displaystyle \
\ssss_{1}\equiv \ssss_{1}(\mu, \delta_{1}):=\frac{\delta_{1}}{2(2-\mu)} \\ 
\displaystyle \ \ssss_{2}\equiv \ssss_{2}(\mu, \delta_{2} ):=\frac{1-\delta_{2}}{2-\mu}\\
 \displaystyle\ \ssss_{3}\equiv \ssss_{3}(\mu, \delta_{1},\delta_{2}):=\frac{3-\mu+(q+1)\delta_{1}-\delta_{2}/2}{2(2-\mu)}\,.
 \end{cases}
 $$
The function $\mu_{n}(\cdot)$ is the one defined in Lemma \ref{le2}. 
\end{lemma}
\begin{proof} This will be obtained by a technique that works as an analogue of dyadic decomposition in Besov spaces, but using the functions $v\equiv v(8B_h)$ in \rif{atomi} as ``atoms''; see Remark \ref{atomicore} below. We therefore divide the proof in two steps. 

{\em Step 1: Estimates on a single ball $B_h$}. Here we again use the notation and the results in Sections \ref{minisec}-\ref{hy}. In particular, here we again argue on a fixed ball $B _h$. Our goal here is to prove estimate \rif{patch} below. We recall that basic properties of difference quotients yield
$$
\int_{B_h}\snr{\tau_{h}(E_{\rr,\textnormal{i}}(\snr{Dv})-\kk)_{+}}^{2}\dx \le 
|h|^{2}\int_{2B_h}\snr{D(E_{\rr,\textnormal{i}}(\snr{Dv})-\kk)_{+}}^{2}\dx
$$
where $\kk\ge 0$ is any number (recall that $|h|\leq |h|^{\beta_0}$). Using this in connection with \rif{caccada} we have  
\eqn{caccia2}
$$
\int_{B_h}\snr{\tau_{h}(E_{\rr,\textnormal{i}}(\snr{Dv})-\kk)_{+}}^{2}\dx 
  \leq c\snr{h}^{2(1-\beta_{0})}\mm^{\delta_{1}}\int_{4B_h}(E_{\rr,\textnormal{i}}(\snr{Dv})-\kk)_{+}^{2} \dx 
$$
with $c\equiv c(\data,\delta_{1})$ and moreover \rif{stim3} holds by \rif{sce5}$_1$; recall  that $\mm \equiv \mm(8B_h)$ is defined in \rif{mmm0}. Use of \rif{caccada} is legitimate here as 
\rif{sce5}$_1$ is assumed. Let us recall that here it is $E_{\rr,\textnormal{i}}(\snr{z})\equiv E_{\rr,\textnormal{i}}(\snr{z};8B_h)$. 
Let us now record a couple of auxiliary estimates. The first is obtained as follows:
\begin{flalign}\label{11}
 &\notag  \int_{4B_h}\snr{E_{\rr}(x,\snr{Du_{\rr}})-E_{\rr,\textnormal{i}}(\snr{Du_{\rr}})}^{2} \dx\\ &\notag \qquad \stackrel{\eqref{0.00bis}}{\le}c\snr{h}^{2\beta_{0}\alpha}\rr^{2\alpha}\int_{4B_h}[\ell_{\omega}(Du_{\rr})]^{2q} \dx\nonumber\\
&\qquad \stackleq{mmm0}c\snr{h}^{2\beta_{0}\alpha}\mm^{2(1-\delta_{2})}\rr^{2\alpha}\int_{4B_h}(|Du_{\rr}|+1)^{2(q-1+\delta_{2})} \dx\nonumber \\
&\qquad \stackrel{\eqref{4.1}_2}{\le} c\snr{h}^{\beta_{0}\tilde \alpha}\M^{\frac{2(1-\delta_{2})}{2-\mu}}\rr^{2\alpha}\int_{4B_h}(|Du_{\rr}|+1)^{2(q-1+\delta_{2})} \dx
\end{flalign}
for $c\equiv c(n,q)$. For the second auxiliary inequality, we estimate
\begin{flalign*}
&\int_{4B_h}\snr{E_{\rr,\textnormal{i}}(\snr{Du_{\rr}})-E_{\rr,\textnormal{i}}(\snr{Dv})}^{2} \dx\\
& \quad \stackleq{0.000} c\int_{4B_h}(\snr{Du_{\rr}}^2+\snr{Dv}^2+1)^{1-\mu}\snr{Du_{\rr}-Dv}^{2} \dx\nonumber \\
&\quad \qquad +c[\tilde a_{\rr, \textnormal{i}}(8B_h)]^{2}\int_{4B_h}(\snr{Du_{\rr}}^2+\snr{Dv}^2+\omega^2)^{q-1}\snr{Du_{\rr}-Dv}^{2} \dx\nonumber \\
&\quad  \stackleq{stim3} c\mm^{(2-\mu)(1+\delta_{1})}\int_{4B_h}
(\snr{Du_{\rr}}^2+\snr{Dv}^2+1)^{-\mu/2}\snr{Du_{\rr}-Dv}^2 \dx\nonumber \\
&\quad \qquad+c[\tilde a_{\rr, \textnormal{i}}(8B_h)]^2\mm^{q(1+\delta_{1})}\\
& \hspace{17mm}\cdot \int_{4B_h}(\snr{Du_{\rr}}^2+\snr{Dv}^2+\omega^2)^{(q-2)/2}\snr{Du_{\rr}-Dv}^{2} \dx\nonumber \\
&\quad  \stackleq{Vm} c\mm^{(2-\mu)(1+\delta_{1})}\int_{4B_h}\snr{V_{1,2-\mu}(Du_{\rr})-V_{1,2-\mu}(Dv)}^{2} \dx\nonumber \\
&\qquad \quad   +c\left( \tilde a_{\rr, \textnormal{i}}(8B_h)[\ell_{\omega}(\mathfrak{m})]^{q}\right)\mm^{q\delta_{1}}\\
& \hspace{17mm}\cdot\int_{4B_h}\tilde a_{\rr, \textnormal{i}}(8B_h)\snr{V_{\omega,q}(Du_{\rr})-V_{\omega,q}(Dv)}^{2} \dx\nonumber \\
&\quad  \stackrel{\eqref{deffiV}}{\leq} c\left(\mm^{(2-\mu)(1+\delta_{1})}+\M\mm^{q\delta_{1}}\right)\int_{4B_h}\mathcal{V}_{\rr, \textnormal{i}}^{2}(Du_{\rr},Dv;8B_h) \dx\,,
\end{flalign*}
where $c\equiv c (\data, \delta_{1})$ and in the last line we have also used \rif{4.1}$_1$. 
Using \eqref{4.1}$_2$, we gain
$$
\int_{4B_h}\snr{E_{\rr,\textnormal{i}}(\snr{Du_{\rr}})-E_{\rr,\textnormal{i}}(\snr{Dv})}^{2} \dx \leq c \M^{\frac{2-\mu+q\delta_{1}}{2-\mu}}\int_{4B_h}\mathcal{V}_{\rr, \textnormal{i}}^{2}(Du_{\rr},Dv;8B_h) \dx\,.
$$
Using this last estimate with \rif{10} we conclude with 
\begin{flalign}\label{12}
\notag & \int_{4B_h}\snr{E_{\rr,\textnormal{i}}(\snr{Du_{\rr}})-E_{\rr,\textnormal{i}}(\snr{Dv})}^{2} \dx\\ & \qquad 
\le c\snr{h}^{\beta_{0}\tia}\M^{\frac{3-\mu+q\delta_{1}-\delta_{2}/2}{2-\mu}}\rr^{\alpha}\int_{8B_h}(|Du_{\rr}|+1)^{q-1+\delta_{2}} \dx\notag \\
& \qquad \quad +c\snr{h}^{\beta_{0}\tia}\M^{\frac{3-\mu+q\delta_{1}-\delta_{2}/2}{2-\mu}}\rr^{\alpha_0} \int_{8B_h}(|Du_{\rr}|+1)^{3\delta_{2}} \dx
\end{flalign}
where $c\equiv c(\data, \delta_{1}, \delta_{2})$, that is the second auxiliary estimate we were aiming at. Triangle inequality now yields
\begin{flalign*}
&\int_{B_h}\snr{\tau_{h}(E_{\rr}(x,\snr{Du_{\rr}})-\kk)_{+}}^{2} \dx \le c\int_{B_h}\snr{\tau_{h}(E_{\rr,\textnormal{i}}(\snr{Dv})-\kk)_{+}}^{2} \dx\nonumber \\
&\ \   +c\int_{B_h}\snr{(E_{\rr}(x+h,\snr{Du_{\rr}(x+h)})-\kappa)_+-(E_{\rr,\textnormal{i}}(\snr{Du_{\rr}(x+h)})-\kappa)_+}^{2} \dx\\
&\ \   +c\int_{B_h}\snr{(E_{\rr,\textnormal{i}}(\snr{Du_{\rr}(x+h)})-\kappa)_+-(E_{\rr,\textnormal{i}}(\snr{Dv(x+h)})-\kappa)_+}^{2} \dx\\
&\ \    +c\int_{B_h}\snr{(E_{\rr,\textnormal{i}}(\snr{Dv})-\kappa)_+-(E_{\rr,\textnormal{i}}(\snr{Du_{\rr}})-\kappa)_+}^{2} \dx
\\
&\ \    +c\int_{B_h}\snr{(E_{\rr,\textnormal{i}}(\snr{Du_{\rr}})-\kappa)_+-(E_{\rr}(x,\snr{Du_{\rr}})-\kk)_{+}}^{2} \dx\,.
\end{flalign*} 
Using the standard property of translations
$$
\int_{B_h}\snr{g(x+h)}^{2} \dx\leq  \int_{4B_h}\snr{g}^{2} \dx 
$$ valid for every $g \in L^2(4B_h)$, (for this note that $h + B_h \subset 4B_h$ as $|h|\leq 1$), and also the Lipschitz continuity of truncations, that is
$
\snr{(s-\kk)_{+}-(t-\kk)_{+}}\leq \snr{s-t}$ for every $ s, t\in \er$, 
we continue to estimate as follows 
\begin{flalign*}
& \int_{B_h}\snr{\tau_{h}(E_{\rr}(x,\snr{Du_{\rr}})-\kk)_{+}}^{2} \dx \\ &\quad  \  \ \le c\int_{B_h}\snr{\tau_{h}(E_{\rr,\textnormal{i}}(\snr{Dv})-\kk)_{+}}^{2} \dx\\  &\quad \qquad +c\int_{4B_h}\snr{E_{\rr}(x,\snr{Du_{\rr}})-E_{\rr,\textnormal{i}}(\snr{Du_{\rr}})}^{2} \dx\\ & \qquad \quad +c\int_{4B_h}\snr{E_{\rr,\textnormal{i}}(\snr{Dv})-E_{\rr,\textnormal{i}}(\snr{Du_{\rr}})}^{2} \dx\nonumber \\
&\quad   \stackrel{\rif{caccia2}}{\leq} c\snr{h}^{2(1-\beta_{0})}\mm^{\delta_{1}}\int_{4B_h}(E_{\rr}(x,\snr{Du_{\rr}})-\kk)_{+}^{2} \dx\nonumber \\
&\quad      \qquad  +c \int_{4B_h}\snr{E_{\rr}(x,\snr{Du_{\rr}})-E_{\rr,\textnormal{i}}(\snr{Du_{\rr}})}^{2} \dx \\
&\quad    \qquad  +c\mm^{\delta_{1}}\int_{4B_h}\snr{E_{\rr,\textnormal{i}}(\snr{Du_{\rr}})-E_{\rr,\textnormal{i}}(\snr{Dv})}^{2} \dx
\end{flalign*}
where $c\equiv c(\data,\delta_{1})$. Note that we have used the elementary estimate $$(E_{\rr,\textnormal{i}}(\snr{Du_{\rr}})-\kk)_{+}\leq (E_{\rr}(x,\snr{Du_{\rr}})-\kk)_{+} \quad \mbox{on $8B_h$}\,,$$ that follows from the very definition of $E_{\rr,\textnormal{i}}(\cdot)$ in \rif{eiin}$_2$. By then using \rif{11} and \rif{12} to estimate the last two integrals in the above display, respectively, and again \rif{4.1}$_2$, we come to
\begin{flalign}
& \int_{B_h}\snr{\tau_{h}(E_{\rr}(\snr{Du_{\rr}})-\kk)_{+}}^{2} \dx \notag\\
&\quad \le \tilde c\snr{h}^{2\alpha_{*}}\M^{2\ssss_{1}}\int_{4B_h}(E_{\rr}(x,\snr{Du_{\rr}})-\kk)_{+}^{2} \dx\nonumber \\
&\quad \qquad + \tilde c\snr{h}^{2\alpha_{*}}\M^{2\ssss_{2}}\rr^{2\alpha}\int_{8B_h}(|Du_{\rr}|+1)^{2(q-1+\delta_{2})} \dx\nonumber \\
&\quad \qquad + \tilde c\snr{h}^{2\alpha_{*}}\M^{2\ssss_{3}}\rr^{\alpha}\int_{8B_h}(|Du_{\rr}|+1)^{q-1+\delta_{2}}\notag  \\
& \quad \qquad +\tilde c\snr{h}^{2\alpha_{*}}\M^{2\ssss_{3}}\rr^{\alpha_0} \int_{8B_h}(|Du_{\rr}|+1)^{3\delta_{2}} \dx
\label{patch}
\end{flalign}
with $\tilde c\equiv \tilde c(\data, \delta_{1}, \delta_{2})$, where $\mathfrak s_1$, $\mathfrak s_2$ and $\mathfrak s_3$ are as in \rif{sonot} 
 and we have taken $\beta_0$ such that
 $$\beta_{0}:=\frac 2{\tia+2} \Longleftrightarrow \beta_{0}\tia=2(1-\beta_{0})=2 \alpha_{*}\,.$$

{\em Step 2: Patching estimates \trif{patch} on different balls}. 
In this second and final step we are now going to recover estimates for $\tau_{h}(E_{\rr}(\cdot, \snr{Du_{\rr}})-\kk)_{+}$ on $\mathcal B_{1/2}$ by patching up estimates \rif{patch} via a dyadic covering argument. This goes as follows: we take a lattice of cubes $\{Q_{\gamma}\}_{\gamma\le \mathfrak{n}}$ with sidelength equal to $2\snr{h}^{\beta_{0}}/\sqrt{n}$, centered at points $\{x_{\gamma}\}_{\gamma\le \mathfrak{n}}\subset \mathcal \mathcal B_{1/2+ 2|h|^{\beta_0}}$, with sides parallel to the coordinate axes, and such that
\begin{flalign}\label{11.1}
\left| \ \mathcal B_{1/2}\setminus \bigcup_{\gamma\le \mathfrak{n}}Q_{\gamma} \ \right|=0,\qquad Q_{\gamma_{1}}\cap Q_{\gamma_{2}}=\emptyset \ \Leftrightarrow \ \gamma_{1}\not =\gamma_{2}.
\end{flalign}
This family of cubes corresponds to a family of balls in the sense that $Q_{\gamma}\equiv Q_{\textnormal{inn}}(B_{\gamma})$ and $B_{\gamma}:= B_{\snr{h}^{\beta_{0}}}(x_{\gamma})$, as defined above. By construction, and in particular by \eqref{hhh}, it is 
$
8B_{\gamma}\Subset \mathcal B_{1}$ for all $\gamma\le \mathfrak{n}$ and $ \mathfrak{n}\approx \snr{h}^{-n\beta_{0}}
$, 
where the implied constant depends on $n$. Moreover, by \eqref{11.1}, each of the dilated balls $8B_{\gamma_{t}}$ intersects the similar ones $8B_{\gamma_{s}}$ fewer than $\mathfrak{c}_n$ times, that is a number depending only on $n$ (uniform finite intersection property). This implies that 
\eqn{sommamis}
$$
\sum_{\gamma=1}^{\mathfrak{n}}\lambda(8B_{\gamma}) \leq \mathfrak{c}_n
\lambda(\mathcal B_1)
$$
holds for every Borel measure $\lambda(\cdot)$ defined on $\mathcal B_1$. 
We then write estimates \rif{patch} on balls $B_h\equiv B_{\gamma}$ and sum up in order to obtain
\begin{flalign}
\notag &  \int_{\mathcal B_{1/2}}\snr{\tau_{h}(E_{\rr}(x,\snr{Du_{\rr}})-\kk)_{+}}^{2} \dx  \\ &\quad\stackrel{\eqref{11.1}}{\leq} 
\sum_{\gamma=1}^{\mathfrak{n}}\int_{B_{\gamma}}\snr{\tau_{h}(E_{\rr}(x,\snr{Du_{\rr}})-\kk)_{+}}^{2} \dx\notag \\
\notag& \quad \stackleq{patch} \tilde c\snr{h}^{2\alpha_{*}}\M^{2\ssss_{1}}\sum_{\gamma=1}^{\mathfrak{n}}\int_{8B_{\gamma}}(E_{\rr}(x,\snr{Du_{\rr}})-\kk)_{+}^{2} \dx\notag \\ & \notag \qquad   \qquad +\tilde c\snr{h}^{2\alpha_{*}}\M^{2\ssss_{2}}\rr^{2\alpha}\sum_{\gamma=1}^{\mathfrak{n}}\int_{8B_{\gamma}}(|Du_{\rr}|+1)^{2(q-1+\delta_{2})} \dx\nonumber\\
\notag&\qquad \qquad  +\tilde c\snr{h}^{2\alpha_{*}}\M^{2\ssss_{3}}\rr^{\alpha}\sum_{\gamma=1}^{\mathfrak{n}}\int_{8B_{\gamma}}(|Du_{\rr}|+1)^{q-1+\delta_{2}} \dx\\ \notag &\qquad \qquad+\tilde c\snr{h}^{2\alpha_{*}}\M^{2\ssss_{3}}\rr^{\alpha_0}\sum_{\gamma=1}^{\mathfrak{n}}\int_{8B_{\gamma}}(|Du_{\rr}|+1)^{3\delta_{2}}  \dx\nonumber\\
\notag& \quad\stackleq{sommamis}\mathfrak{c}_n\tilde c\snr{h}^{2\alpha_{*}}\M^{2\ssss_{1}}\int_{\mathcal B_{1}}(E_{\rr}(x,\snr{Du_{\rr}})-\kk)_{+}^{2} \dx\\ \notag & \qquad \qquad  +\mathfrak{c}_n\tilde c\snr{h}^{2\alpha_{*}}\M^{2\ssss_{2}}\rr^{2\alpha}\int_{\mathcal B_{1}}(|Du_{\rr}|+1)^{2(q-1+\delta_{2})} \dx\nonumber \\
&\notag  \qquad  \qquad +\mathfrak{c}_n\tilde c\snr{h}^{2\alpha_{*}}\M^{2\ssss_{3}}\rr^{\alpha}\int_{\mathcal B_{1}}(|Du_{\rr}|+1)^{q-1+\delta_{2}} \dx\\ & \qquad   \qquad +\mathfrak{c}_n\tilde c\snr{h}^{2\alpha_{*}}\M^{2\ssss_{3}}\rr^{\alpha_0}\int_{\mathcal B_{1}}(|Du_{\rr}|+1)^{3\delta_{2}} \dx\,. \label{frommi}
\end{flalign}
Thanks to the inequality in the last display we apply  Lemma \ref{l4}. This yields $(E_{\rr}(\cdot,\snr{Du_{\rr}})-\kk)_{+}\in W^{\beta,2}(\mathcal B_{1/2})$  for all $\beta\in (0,\alpha_{*})$ with the related a priori bound
\begin{flalign}
\notag & \nr{(E_{\rr}(\cdot,\snr{Du_{\rr}})-\kk)_{+}}_{L^{2\chi}(\mathcal B_{1/2})}+ [(E_{\rr}(\cdot,\snr{Du_{\rr}})-\kk)_{+}]_{\beta,2;\mathcal B_{1/2}} \\
&\qquad \le c\M^{\ssss_{1}}\nr{E_{\rr}(\cdot,\snr{Du_{\rr}})-\kappa)_{+}}_{L^{2}(\mathcal B_{1})} \notag   \\
 & \qquad \quad+c\M^{\ssss_{2}}\rr^{\alpha}\nr{|Du_{\rr}|+1}_{L^{2(q-1+\delta_{2})}(\mathcal B_1)}^{q-1+\delta_{2}}\nonumber\\
& \qquad \quad +c\M^{\ssss_{3}}\rr^{\alpha/2}\nr{|Du_{\rr}|+1}_{L^{q-1+\delta_{2}}(\mathcal B_1)}^{(q-1+\delta_{2})/2}\notag \\ & \qquad \quad +c\M^{\ssss_{3}}\rr^{\alpha_0/2}\nr{|Du_{\rr}|+1}_{L^{3\delta_{2}}(\mathcal B_1)}^{3\delta_{2}/2}
\label{riscalon}
\end{flalign}
that holds for every $\chi\equiv \chi(\beta)$ as in \rif{ilbeta2}, where $c\equiv c(\data, \delta_{1}, \delta_{2}, \beta)$. Notice that here we have used, in order, first \rif{cru2}, as a consequence of \rif{frommi}, and then 
\eqref{immersione}. Scaling back from $\mathcal B_1$ to $B_{\rr}$ in \rif{riscalon} via \rif{notational0}, \rif{nuoviaut0} and \rif{eiin}, squaring the resulting inequality, and restoring the original notation, we arrive at \rif{16} and the proof is complete.
\end{proof}
\vspace{1mm}
\begin{remark}[Nonlinear atoms]\label{atomicore} Here we briefly expand on the possible analogy between the classic atomic decompositions in fractional spaces (see for instance \cite[Section 4.6]{AH} and \cite[Chapter 2]{triebel}) and the construction made in the proof of Lemma \ref{fracacc}. Atomic decompositions of a Besov function $w$ usually go via decompositions of the space $\er^n$ in dyadic grids with mesh $2^{-k}$, $k\in \en$, corresponding to the annuli (in the frequency space) considered in Littllewood-Paley theory. On each cube $Q_{k,\gamma}$ of the grid one considers an atom $a_{k,\gamma}(\cdot)$, i.e., a smooth function with certain control on its derivatives, up to the maximal degree of regularity one is interested in describing for $w$. Specifically, one requires that
\eqn{atom1}
$$
\textnormal{supp}\, a_{k,\gamma} \subset Q_{k,\gamma}\,, \qquad |D_{S} a_{k,\gamma}|\lesssim 2^{-|S|k}
$$
hold for sufficiently large multi-indices $S$. Summing up (over $\gamma$) such atoms multiplied  by suitable modulating coefficients, and then yet over all possible grids $k\in \en$, allows to give a precise description of the smoothness of the function $w$. Such ``linear'' decompositions, although very efficient, are of little use when dealing with nonlinear problems as those considered in this paper. The idea in Lemma \ref{fracacc} is then, given a grid of size $|h|^{\beta_0} \approx 1/2^{k\beta_0}$, and therefore a certain ``height'' in the frequency space, to consider atoms $v$ that are in a sense close to the original solution $u$ in that they are themselves solutions to nonlinear problems (with frozen coefficients). In other words we attempt a decomposition of the type 
\eqn{atom2}
$$
u_{\rr}(x) \approx \sum_{\gamma\le \mathfrak{n}} v_\gamma(x) \mathds{1}_{B_\gamma}(x) + \texttt{o}(|h|^{\tilde \alpha})
 $$
 where $v_{\gamma}$ is defined as in \rif{atomi}, with $B_{\gamma}\equiv B_{h}$ as in Step 2 from Lemma \ref{fracacc} ($ \mathds{1}_{B_\gamma}$ is the indicator function of the ball $B_\gamma$). Notice in fact the analogy with the second information in 
\rif{atom1}, describing the maximal smoothness of a classical atom, with the Caccioppoli inequality \rif{caccada}, from which one infers its fractional version for $u_{\rr}$, that is \rif{16}.

\end{remark}

 \subsection{$C^{0,1}$-bounds via nonlinear potentials} Here we deliver
\begin{proposition}\label{priora}
Let $u_{\omega,\varepsilon}\in W^{1,q}(B_{r})$ be as in \eqref{pd}. The inequality
\eqn{18}
$$
\nr{\ti{E}_{\omega,\sigma_{\eps}}(\cdot,\snr{Du_{\omega,\varepsilon}})}_{L^{\infty}(B_{t})}\le\frac{c}{(s-t)^{n\vartheta}}\nr{H_{\omega,\sigma_{\eps}}(\cdot,Du_{\omega,\varepsilon})+1}_{L^{1}(B_{s})}^{\vartheta}+c
$$
holds whenever $B_{t}\Subset B_s\subseteq B_{r}$ are concentric balls such that $r\leq 1$, where $c\equiv c(\data )$ and $\vartheta\equiv \vartheta(n,q,\alpha, \alpha_0)$. The function $\ti{E}_{\omega,\sigma_{\eps}}(\cdot)$ has been introduced in \trif{eiit}. 
\end{proposition}
\begin{proof} In the following we are going to use Lemma \ref{fracacc} with $\delta_1\in (0,1)$ whose size will determined in a few lines as a function of $n,q,\alpha, \alpha_0$, and with $\delta_2$ by now fixed by
\eqn{ildelta}
$$
 0< \delta_{2}:= \frac12\min\left\{1+\frac{\alpha}{n}-q, \frac{\alpha_0}{3n}\right\}< \frac 12\,.
$$
Note that $\delta_2>0$ follows as a consequence of the assumed bound \rif{doppiob2} and that the choice in \rif{ildelta} makes $\delta_2$ depending only on $n,q,\alpha, \alpha_0$. Without loss of generality we can assume that
\eqn{ammissibile}
$$
\nr{\ti{E}_{\omega,\sigma_{\eps}}(\cdot,\snr{Du_{\omega,\varepsilon}})}_{L^{\infty}(B_{t})} \geq 1
$$
otherwise \rif{18} is obvious.  
We consider concentric balls $B_{t}\Subset B_{\tau_{1}}\Subset B_{\tau_{2}}\Subset B_{s}$ and set $r_{0}:=(\tau_{2}-\tau_{1})/8$. It follows that $B_{2r_{0}}(x_{0})\Subset B_{\tau_{2}}$ holds whenever $x_0\in B_{\tau_1}$. Notice that by \rif{3} every point is a Lebesgue point for both $\snr{Du_{\omega,\varepsilon}}$ and $E_{\omega,\sigma_{\eps}}(\cdot,\snr{Du_{\omega,\varepsilon}})$. By Lemma \ref{fracacc}, and \eqref{16} used on $B_{\rr}(x_0)\subset B_{r_0}(x_0)$, we can apply Lemma \ref{revlem} on $B_{r_{0}}(x_{0})$ verifying \rif{revva} with
\vspace{3mm}
\eqn{sceltona}
$$
\begin{cases}
\ w\equiv E_{\omega,\sigma_{\eps}}(\cdot,\snr{Du_{\omega,\varepsilon}})\\
 \ \M:=2\nr{\ti{E}_{\omega,\sigma_{\eps}}(\cdot,\snr{Du_{\omega,\varepsilon}})}_{L^{\infty}(B_{\tau_{2}})}\\
 \quad \mbox{(such a choice of $\M$ is admissible in \rif{16} by \rif{ammissibile})}\\
\ M_{0}:=\M^{\ssss_{1}}, \ M_{1}\equiv\M^{\ssss_{2}}, \ M_{2}\equiv M_{3}:=\M^{\ssss_{3}}\\
\ \chi:=n/(n-2\beta)>1, \quad \beta:=\alpha_*/2 \quad \mbox{(recall \rif{ilbeta2})}\\
\  \sigma_{1}:=\alpha, \ \sigma_{2}:=\alpha/2, \ \sigma_{3}:=\alpha_0/2, \  \ f_{1}\equiv f_{2}\equiv f_{3}:=\snr{Du_{\omega,\varepsilon}}+1,\\
\ \theta_{1}\equiv \theta_{2}\equiv \theta_{3}:=1, \ c_*\equiv c_*(\data,\delta_1, \beta)\equiv c_*(\data,\delta_1),   \\
\ m_{1}:=2(q-1+\delta_{2}), \ m_{2}:=q-1+\delta_{2}, \ m_{3}:=3\delta_{2}, \ \kappa_0:=0\,.
\end{cases}
$$
\vspace{3mm}
As $x_0 \in B_{\tau_1}$ has been chosen arbitrarily, \rif{siapplica} with the choices in \rif{sceltona} implies
\begin{flalign}\label{20}
 &\notag \nr{E_{\omega,\sigma_{\eps}}(\cdot,\snr{Du_{\omega,\varepsilon}})}_{L^{\infty}(B_{\tau_{1}})}\\
 &\qquad \le\frac{c\M^{\frac{\ssss_1\chi}{\chi-1}}}{(\tau_{2}-\tau_{1})^{n/2}}\left(\int_{B_{\tau_2}}[E_{\omega,\sigma_{\eps}}(x,\snr{Du_{\omega,\varepsilon}})]^2 \dx\right)^{1/2}\nonumber \\
&\quad \qquad +c\M^{\frac{\ssss_1}{\chi-1}+ \ssss_2}\nr{\mathbf{P}^{2(q-1+\delta_{2}),1}_{2,\alpha}( \snr{Du_{\omega,\varepsilon}}+1;\cdot,(\tau_{2}-\tau_{1})/4)}_{L^{\infty}(B_{\tau_{1}})}\nonumber \\
&\quad \qquad +c\M^{\frac{\ssss_1}{\chi-1}+ \ssss_3}\nr{\mathbf{P}^{q-1+\delta_{2},1}_{2,\alpha/2}(\snr{Du_{\omega,\varepsilon}}+1;\cdot,(\tau_{2}-\tau_{1})/4)}_{L^{\infty}(B_{\tau_{1}})}\nonumber \\
&\quad \qquad +c\M^{\frac{\ssss_1}{\chi-1}+ \ssss_3}\nr{\mathbf{P}^{3\delta_{2},1}_{2,\alpha_0/2}(\snr{Du_{\omega,\varepsilon}}+1;\cdot,(\tau_{2}-\tau_{1})/4)}_{L^{\infty}(B_{\tau_{1}})},
\end{flalign}
for $c\equiv c(\data, \delta_{1})$. Recalling the definitions in \rif{eii2}-\rif{eiit}, and the choice of $\M$ in \rif{sceltona}, after a few elementary manipulations in \rif{20} we arrive at
\vspace{.5mm}
\begin{flalign}
& \notag \nr{\ti{E}_{\omega,\sigma_{\eps}}(\cdot,\snr{Du_{\omega,\varepsilon}})}_{L^{\infty}(B_{\tau_{1}})}  
\\ \quad \ \  & \leq \frac{c}{(\tau_{2}-\tau_{1})^{n/2}}\nr{\ti{E}_{\omega,\sigma_{\eps}}(\cdot,\snr{Du_{\omega,\varepsilon}})}_{L^{\infty}(B_{\tau_{2}})}^{\frac{\ssss_1\chi}{\chi-1}+\frac 12}\left(\int_{B_{s}}\ti{E}_{\omega,\sigma_{\eps}}(x,\snr{Du_{\omega,\varepsilon}}) \dx\right)^{\frac 12}\nonumber \\
&\notag\quad +c\nr{\ti{E}_{\omega,\sigma_{\eps}}(\cdot,\snr{Du_{\omega,\varepsilon}})}_{L^{\infty}(B_{\tau_{2}})}^{\frac{\ssss_1}{\chi-1}+ \ssss_2}\\
& \qquad \quad \cdot \nr{\mathbf{P}^{2(q-1+\delta_{2}),1}_{2,\alpha}(\snr{Du_{\omega,\varepsilon}}+1;\cdot,(\tau_{2}-\tau_{1})/4)}_{L^{\infty}(B_{\tau_{1}})}\nonumber\\
&\notag\quad +c\nr{\ti{E}_{\omega,\sigma_{\eps}}(\cdot,\snr{Du_{\omega,\varepsilon}})}_{L^{\infty}(B_{\tau_{2}})}^{\frac{\ssss_1}{\chi-1}+ \ssss_3}\\
& \qquad \quad \cdot \nr{\mathbf{P}^{q-1+\delta_{2},1}_{2,\alpha/2}(\snr{Du_{\omega,\varepsilon}}+1;\cdot,(\tau_{2}-\tau_{1})/4)}_{L^{\infty}(B_{\tau_{1}})}\nonumber\\
&\quad +c\nr{\ti{E}_{\omega,\sigma_{\eps}}(\cdot,\snr{Du_{\omega,\varepsilon}})}_{L^{\infty}(B_{\tau_{2}})}^{\frac{\ssss_1}{\chi-1}+ \ssss_3}\notag \\
& \qquad \quad \cdot \nr{\mathbf{P}^{3\delta_{2},1}_{2,\alpha_0/2}(\snr{Du_{\omega,\varepsilon}}+1;\cdot,(\tau_{2}-\tau_{1})/4)}_{L^{\infty}(B_{\tau_{1}})}  +c
\label{quasi}
\end{flalign}
\vspace{.5mm}
with $c\equiv c(\data, \delta_{1} )$. We now take  $\delta_{1}$ to be such that
\eqn{thanks}
$$
\begin{cases}
\displaystyle
\frac{\ssss_{1}(1,\delta_{1})\chi}{\chi-1}+\frac 12= \frac{n\delta_{1}}{4\beta}+\frac 12<1  \\
\displaystyle
\frac{\ssss_{1}(1,\delta_{1})}{\chi-1}+\ssss_{2}(1,  \delta_{2})= \frac{(n-2\beta)\delta_{1}}{4\beta}+1-\delta_{2} <1 \\
\displaystyle
\frac{\ssss_{1}(1,\delta_{1})}{\chi-1}+\ssss_{3}(1 ,\delta_{1}, \delta_{2})= \frac{(n-2\beta)\delta_{1}}{4\beta}+1 +\frac{(q+1)\delta_{1}-\delta_{2}/2}{2} <1\,.
\end{cases}
$$
Recalling \rif{ildelta}, the quantity $\delta_{1}$ is now determined as a function of the parameters $n,q,\alpha, \alpha_0$. This determines a first restriction on the size of $\mu$ via  \rif{sce5} and the choice of $\delta_{1}$. Notice that at this stage the number $\mu_{*}:=\mu_n(\delta_{1}/2)$, determining an upper bound on $\mu$, depends on $n,q,\alpha, \alpha_0$; notice also that all the above computation remain valid provided $1\leq \mu < \mu_{*}\equiv \mu_{*}(n,q,\alpha, \alpha_0)$. Finally, by \rif{thanks} and a standard continuity argument we can further restrict the value of $\mu$ finding $\mum \leq \mu_{*}$ such that 
\eqn{minori}
$$
1\leq \mu < \mum \Longrightarrow 
\begin{cases}
\displaystyle
\frac{\ssss_{1}(\mu,\delta_{1})\chi}{\chi-1}+\frac 12<1  \\
\displaystyle
\frac{\ssss_{1}(\mu,\delta_{1})}{\chi-1}+\ssss_{2}(\mu, \delta_{2})<1 \\
\displaystyle
\frac{\ssss_{1}(\mu,\delta_{1})}{\chi-1}+\ssss_{3}(\mu ,\delta_{1}, \delta_{2},)<1\,.
\end{cases}
$$
Notice also that the functions $\mu \mapsto \ssss_{1}(\mu,\cdot), \ssss_{2}(\mu,\cdot), \ssss_{3}(\mu,\cdot)$ are increasing. This finally fixes the value of  $\mum$ from the statement of Theorem \ref{t3}, with the asserted dependence on the constants.  We now want to the estimate the potential terms appearing in the right-hand side of \rif{quasi}
 by means of Lemma \ref{crit}. In this respect, notice that $
n>2\alpha
$ allows to verify \rif{required}, while we can take $\gamma=1$ in \rif{stimazza} because
 $$
\rif{ildelta} \Longrightarrow 1 > \max\left\{\frac{n(q-1+\delta_{2})}{\alpha}, \frac{3n\delta_{2}}{\alpha_0}\right\}\,.
 $$
Applying \rif{stimazza} gives 
\begin{flalign}\label{21.11}
\begin{cases}
\ \nr{\mathbf{P}^{q-1+\delta_{2},1}_{2,\alpha/2}(\snr{Du_{\omega,\varepsilon}}+1;\cdot,(\tau_{2}-\tau_{1})/4)}_{L^{\infty}(B_{\tau_{1}})}\\
\qquad \le c\nr{\snr{Du_{\omega,\varepsilon}}+1}_{L^{1}(B_{s})}^{(q-1+\delta_{2})/2}\\
\ \nr{\mathbf{P}^{2(q-1+\delta_{2}),1}_{2,\alpha}(\snr{Du_{\omega,\varepsilon}}+1;\cdot,(\tau_{2}-\tau_{1})/4)}_{L^{\infty}(B_{\tau_{1}})} \\
\qquad \le c\nr{\snr{Du_{\omega,\varepsilon}}+1}_{L^{1}(B_{s})}^{q-1+\delta_{2}}
\\
\ \nr{\mathbf{P}^{3\delta_{2},1}_{2,\alpha_0/2}(\snr{Du_{\omega,\varepsilon}}+1;\cdot,(\tau_{2}-\tau_{1})/4)}_{L^{\infty}(B_{\tau_{1}})} \\
\qquad \le c\nr{\snr{Du_{\omega,\varepsilon}}+1}_{L^{1}(B_{s})}^{3\delta_{2}/2},
\end{cases}
\end{flalign}
with $c\equiv c(n,q,\alpha, \alpha_0)$. Using \rif{21.11} in \rif{quasi}, and recalling \rif{minori}, we can use Young's  inequality to finally get
\begin{flalign*}
&\nr{\ti{E}_{\omega,\sigma_{\eps}}(\cdot,\snr{Du_{\omega,\varepsilon}})}_{L^{\infty}(B_{\tau_{1}})}\le \frac{1}{2}\nr{\ti{E}_{\omega,\sigma_{\eps}}(\cdot,\snr{Du_{\omega,\varepsilon}})}_{L^{\infty}(B_{\tau_{2}})}\nonumber \\
&  \qquad  \qquad  +\frac{c}{(\tau_{2}-\tau_{1})^{n\vartheta}}\nr{\ti{E}_{\omega,\sigma_{\eps}}(\cdot,\snr{Du_{\omega,\varepsilon}})}^{\vartheta}_{L^{1}(B_{s})}+\nr{Du_{\omega,\varepsilon}}_{L^{1}(B_{s})}^{\vartheta}+c
\end{flalign*}
with $c\equiv c(\data )$ and $\vartheta \equiv \vartheta(n,q,\alpha, \alpha_0)\geq 1$. This allows to apply Lemma \ref{l5} that provides
\begin{flalign*}
 & \nr{\ti{E}_{\omega,\sigma_{\eps}}(\cdot,\snr{Du_{\omega,\varepsilon}})}_{L^{\infty}(B_{t})} \\& \qquad \le \frac{c}{(s-t)^{n\vartheta}}\nr{\ti{E}_{\omega,\sigma_{\eps}}(\cdot,\snr{Du_{\omega,\varepsilon}})}^{\vartheta}_{L^{1}(B_{s})} +c\nr{Du_{\omega,\varepsilon}}_{L^{1}(B_{s})}^{\vartheta}+c
\end{flalign*}
from which \eqref{18} follows using \rif{0.002} and a few elementary manipulations. \end{proof}
\subsection{Proof of \trif{30bis}}\label{conv} Keeping in mind the notation of Proposition \ref{priora}, \eqref{24} and \eqref{18} give
\begin{flalign*}
& \nr{\ti{E}_{\omega,\sigma_{\eps}}(\cdot,\snr{Du_{\omega,\varepsilon}})}_{L^{\infty}(B_{t})}\\& \qquad \leq  \frac{c}{(r-t)^{n\vartheta}}\left[\mathcal{N}(u,B_{r}) +\texttt{o}_{\eps}(\omega)+\texttt{o}(\varepsilon) +|B_{r}|\right]^{\vartheta}+c\,,
\end{flalign*}
so that, recalling \rif{eiit} we find
\begin{flalign}
\notag \nr{Du_{\omega,\varepsilon}}_{L^{\infty}(B_{t})}^{2-\mu}& \leq \frac{c}{(r-t)^{n\vartheta}}\left[\mathcal{N}(u,B_{r}) +\texttt{o}_{\eps}(\omega)+\texttt{o}(\varepsilon) +|B_{r}|\right]^{\vartheta}+c  \\
& =: [\mm_{\omega,\varepsilon}(t;B_{r})]^{2-\mu}\,.
\label{28b}
\end{flalign}
 It follows that, up to not relabelled subsequences, the convergence in \rif{23} can be upgraded to $u_{\omega,\varepsilon}\rightharpoonup^{*} u_{\varepsilon}$ in $ W^{1,\infty}(B_{t})$ for every $\eps >0$. Letting $\omega\to 0$ in \rif{28b} yields 
\begin{flalign}
\notag \nr{Du_{\varepsilon}}_{L^{\infty}(B_{t})}^{2-\mu} &\leq \frac{c}{(r-t)^{n\vartheta}}\left[\mathcal{N}(u,B_{r}) +\texttt{o}(\varepsilon)+|B_{r}|\right]^{\vartheta}+c\\
& =:[\mm_{\varepsilon}(t;B_{r})]^{2-\mu},  \label{28bb}
\end{flalign}
that again holds for every $\eps>0$, with $c\equiv c (\data)$ and $\vartheta\equiv \vartheta(n,q,\alpha, \alpha_0)$. Similarly, as after \rif{28b}, by \rif{28bb} the convergence in \rif{27} can be upgraded to $u_{\varepsilon}\rightharpoonup^{*}  u$ in $ W^{1,\infty}(B_{t})$, so that \rif{30bis} (for $\bb=0$) follows letting $\eps \to 0$ in \rif{28bb}, taking $t=r/2$, recalling that $\mu <3/2$ and renaming $2\vartheta$ into $\vartheta$. 
\subsection{Local gradient H\"older continuity}\label{holsec} Since the result is local, up passing to smaller open subsets we can assume with no loss of generality that $\mathcal N(u, \Omega)$ is finite (see the comment after \rif{defilocal}). 
We fix an open subset  $\Omega_0 \Subset \Omega$ and the radius $r :=\min\{\dist(\Omega_0, \partial \Omega)/8, 1\}$. We take $B_{2r}$ centred in $ \Omega_0$ and recall that the quantities $\mm_{\omega,\varepsilon}\equiv \mm_{\omega,\varepsilon}(r;B_{2r})$ and  $\mm_{\varepsilon}\equiv \mm_{\varepsilon}(r;B_{2r})$ are defined in \rif{28b} and \rif{28bb}, respectively. With $\tau \leq r$ and $B_{\tau}\subset B_{r}$ being a ball concentric to $B_{2r}$, we define $v\in u_{\omega,\varepsilon}+W^{1,q}_{0}(B_{\tau})$ as
$$
v\mapsto \min_{w \in u_{\omega,\varepsilon}+W^{1,q}_{0}(B_{\tau})} \int_{B_{\tau}}H_{\omega,\sigma_{\eps},\textnormal{i}}(Dw; B_{\tau}) \dx
$$
so that
\eqn{28bbb}
$$
\nr{Du_{\omega, \eps}}_{L^{\infty}(B_{\tau})}+ \nr{Dv}_{L^{\infty}(B_{\tau/2})}\le \bar c \mm_{\omega,\varepsilon}^{2}
$$
holds with $\bar c\equiv \bar c (\data)$. The derivation of \rif{28bbb} follows using first \rif{28b} (with $B_{r}$ replaced by $B_{2r}$ and $t=r$) and combining it with \rif{8.fz} and easy manipulations as in Proposition \ref{fzfz} and Lemma \ref{le2}. By Proposition \ref{tecnica} in the subsequent step and \rif{28bbb}, we find that
\eqn{holly1}
$$
\mint_{B_\varrho} \snr{Dv-(Dv)_{B_\varrho}}^2 \dx \leq c_{\omega,\varepsilon} \left(\varrho/\tau\right)^{2\beta_{\omega,\varepsilon}}$$
holds for every $\rr \leq \tau/2$. Here, in the notation of Proposition \ref{tecnica} below, it is $c_{\omega,\varepsilon}:= c_*(2\bar c \mm_{\omega,\varepsilon}^2)$ and $\beta_{\omega,\varepsilon}:= \beta_*(2\bar c \mm_{\omega,\varepsilon}^2)$; these constants are non-decreasing and non-increasing functions of $\mm_{\omega,\varepsilon}$, respectively. Letting first $\omega\to 0$ and then $\eps \to 0$, we have that $$\mm_{\omega,\varepsilon}^{2-\mu}\to \mm_{\varepsilon}^{2-\mu}\to \mm^{2-\mu} := c
r^{-n\vartheta}[\mathcal N(u,B_{2r})+ |B_{2r}|]^{\vartheta}+ c$$ (see \rif{28bb}). In particular
\eqn{finalbound}
$$
 \mm^{2-\mu} \leq c r ^{-n\vartheta}[\mathcal N(u,\Omega)+1] + c =:\mm_0^{2-\mu}\,,
$$
holds with $c\equiv c (\data)$ and $\vartheta\equiv \vartheta(n,q,\alpha, \alpha_0)$ so that $\mm_0$ only depends on $\data, \dist(\Omega_0, \partial \Omega)$ and $\mathcal N(u,\Omega)$. 
Again letting $\omega\to 0$ and then $\eps\to 0$, and passing to not relabelled subsequences, we can assume that $c_{\omega,\eps}\to: = c_{\eps}\to =: \underbar{c} \leq  c_*(4\bar c \mm^2) < \infty$. Notice that by \rif{finalbound} we have $\underbar{c} \leq  c_*(4\bar c \mm_0^2)$ and this last quantity depends only on $\data, \dist(\Omega_0, \partial \Omega)$ and $\mathcal N(u, \Omega)$ but it is otherwise independent of the chosen ball $B_{2r}$.  In the following, with some abuse of notation, we keep on denoting by $c_{\omega,\eps}\geq 1$ a double-sequence of constants with the above property, typically being itself released via a non-decreasing function of $c_*(2\bar c \mm_{\omega,\varepsilon}^2)$; the exact value of the numbers $c_{\omega,\eps}$ may vary from line to line. A similar reasoning can be done for the exponents $\beta_{\omega,\eps}$, that is, we have $\beta_{\omega,\eps}\to: = \beta_{\eps}\to =: \beta \geq  \beta_*(4\bar c \mm^2)  \geq \beta_*(4\bar c \mm_0^2)=: \tilde \beta \in (0,1)$. Similarly, $\tilde \beta$ depends only  
$\data, \dist(\Omega_0, \partial \Omega)$ and $\mathcal N(u, \Omega)$ and is  independent of the ball $B_{2r}$ considered. Proceeding as for \rif{10}, and also using \rif{28bbb} repeatedly, it follows that 
\eqn{segno}
$$
 \mint_{B_{\tau}}\snr{V_{1,2-\mu}(Du_{\omega, \eps})-V_{1,2-\mu}(Dv)}^{2}  \dx   \le c_{\omega,\eps}\tau^{\tia}\,,
$$
where $\tia=\min\{\alpha, \alpha_0\}$. Using \rif{Vm} with $p=2-\mu$ and then \rif{28bbb} yields
\begin{flalign*}
& \mint_{B_{\tau/2}}\snr{Du_{\omega, \eps}-Dv}^2 \dx \\
&\quad \leq c\mint_{B_{\tau/2}}(\snr{Du_{\omega, \eps}}^2+\snr{Dv}^2+1)^{\mu/2}\snr{V_{1,2-\mu}(Du_{\omega, \eps})-V_{1,2-\mu}(Dv)}^2
\dx
\\
&\quad \leq c \mm_{\omega,\varepsilon}^{2\mu}\mint_{B_{\tau}}\snr{V_{1,2-\mu}(Du_{\omega, \eps})-V_{1,2-\mu}(Dv)}^2
\dx\,.
\end{flalign*}
By again using  \rif{segno} in the above display we conclude with 
\eqn{holly2}
$$
\mint_{B_{\tau/2}}\snr{Du_{\omega, \eps}-Dv}^2  \dx\leq   c_{\omega,\eps}\tau^{\tia}\,.
$$
We can now complete the proof with more standard arguments (see for instance \cite{gg2, manth1, manth2}), that we recall for completeness. Combining \rif{holly1} and \rif{holly2} yields
$$
\mint_{B_\varrho} \snr{Du_{\omega, \eps}-(Du_{\omega, \eps})_{B_\varrho}}^2\dx
\leq  c_{\omega, \eps}(\rr/\tau)^{2\beta_{\omega,\eps}} +c_{\omega, \eps}(\tau/\rr)^n \tau^{\tia}
$$
for every $\rr \leq \tau/2$, 
and taking $\varrho=  (\tau/2)^{1+\tia/(n+2\beta_{\omega,\varepsilon})}$ finally yields
$$
 \mint_{B_\varrho}\snr{Du_{\omega,\varepsilon}-(Du_{\omega,\varepsilon})_{B_\varrho}}^2   \dx
\leq  c_{\omega,\varepsilon} \varrho^{\frac{2\tia\beta_{\omega,\varepsilon}}{n+2\beta_{\omega,\varepsilon}+\tilde \alpha}}$$ 
for every $\rr \leq  (r/2)^{1+\tia/(n+2\beta_{\omega,\varepsilon})}$.  
In the above display we first let $\omega\to 0$ and then $\eps \to 0 $, and finally conclude with
\eqn{holdsfor}
$$
 \mint_{B_\varrho}\snr{Du-(Du)_{B_\varrho}}^2  \dx
\leq c\varrho^{\frac{2\tia\tilde \beta}{n+2\tilde \beta+\tilde \alpha}}\,,
$$ 
where, by the discussion made after \rif{finalbound}, both $c\geq 1$ and $\tilde \beta \in (0,1)$ depend only on $\data$, $\dist(\Omega_0, \partial \Omega)$ and $\mathcal N (u, \Omega)$, but are otherwise independent of the ball considered $B_{2r}$. Since $\Omega_0$ is arbitrary, \rif{holdsfor} and a standard covering argument and Campanato-Meyers integral characterization of H\"older continuity yield that for every open subset $\Omega_0 \Subset \Omega$ it holds that 
$
[Du]_{0,\beta; \Omega_0}<\infty$, where
$\beta := \tilde\alpha\tilde \beta/(n+2\tilde \beta+\tilde \alpha)
$. The proof of Theorem \ref{t3} in the case $\bb =0$ is therefore complete up to \rif{holly1}, whose proof will be given in the subsequent section. We now briefly comment on how to obtain the local gradient H\"older continuity in the nonsingular case $\bb >0$, which is simpler since it requires no approximation via the additional parameter $\omega$. For the proof of \rif{30bis} it is sufficient to use functionals $\mathcal{N}_{\bb, \eps}$ as in \rif{approssimaF}, and all the subsequent estimates remain independent of $\bb$ and $\eps$; the approximation and the convergence then only take place with respect to the parameter $\eps$ and leads to \rif{30bis}. The same reasoning applies to the proof in this section by taking $\omega \equiv \bb$. It remains to deal with the last issue, that is the local H\"older continuity of $Du$ with explicit exponent when $\bb>0$. This will be done in Section \ref{ultimina}. 

\subsection{A technical decay estimate.}\label{latecnica} Here we prove \trif{holly1}; this relies on certain hidden facts from regularity theory of singular parabolic equations that we take as a starting point to treat the nonuniformly elliptic case considered here. Relevant related methods are also in \cite{dibe}. Consistently with the notation established in Sections \ref{const} and \ref{rere}, with $\bar a > 0$ being a fixed number, in the following we denote 
\eqn{lareferenza}
$$\Hi(z):=
F(x_{\rm c}, z)+ \bar a (\snr{z}^2+\omega^2)^{q/2}$$ for every $z\in \er^n$, where $\omega \in (0,1]$ and  $x_{\rm c}\in \Omega$. Here we permanently assume \trif{assif}-\trif{assi3} and therefore $\Hi(\cdot)$ is an integrand of the type in \rif{defiH}$_2$ with $a_{\sigma, \textnormal{i}}(B)$ replaced by $\bar a$. Indeed, conditions \rif{rege.2i} apply here when accordingly recasted (see also \rif{ragnar2} below). Moreover, let us define
\eqn{ragnar}
$$
\tilde g(t) := 
\begin{cases} 
t & \mbox{if $0 \leq t \leq 1$}\\
g(t)& \mbox{if $t> 1$}\,.
\end{cases}
$$
Then \rif{assi3} and Lemma \ref{marclemma} imply that 
\eqn{assi33}
$$
|\partial_z F(x,z)| \leq c\tilde g(|z|) \Longrightarrow |\partial_z \Hi(z)| \leq c\tilde g(|z|) + c \bar a  (\snr{z}^2+\omega^2)^{(q-2)/2}|z|
$$
holds for every $z \in \er^n$, where $c\equiv c (q,L, c_g(1))$. 

\begin{proposition}\label{tecnica} Under assumptions \trif{assif}-\trif{assi3}, with $B\subset \er^n$ being a ball, let $v\in W^{1,\infty}(B)$ such that  
\eqn{exp1}
$$
\begin{cases}
\, -\diver\, \partial_z\Hi(Dv)=0\\
\, \|Dv\|_{L^{\infty}(B)} +1 \leq \mm\,.
\end{cases}
$$
There exist two functions of $\mm$, $c_*\colon [1, \infty)\to [1, \infty)$ and $\beta_*\colon [1, \infty) \to (0,1)$, such that
\eqn{diminutio}
$$
\textnormal{\texttt{osc}}_{\tau B}\,  Dv  \leq c_* \tau^{\beta_*}
$$
holds for every $\tau \in (0,1)$. 
The functions $c_*(\cdot)$ and $\beta_*(\cdot)$ are non-decreasing and non-increasing, respectively, and also depend on $n,q, \mu,\nu, L, g(\cdot)$. They are independent of $\bar a$ and $\omega \in (0,1]$. 
\end{proposition}
The proof of 
\rif{diminutio} is achieved via two preliminary lemmas. In the rest of section, we denote $\|Dv\|:= \max_{1\leq s\leq n}|D_s v|$.
\vspace{.5mm} 
\begin{lemma}\label{limlem1} With $B\subset \er^n$ being a ball, assume that $v\in W^{1,\infty}(B)$ satisfies \eqref{exp1}$_1$ and 
\eqn{limitatio}
$$
\|Dv\|+\omega \leq \lambda \ \  \mbox{in $B$, where $\lambda \leq \mm$, $\mm\geq 1$}\,.
$$
There exists $\sigma \equiv \sigma(n,\mu, q, \nu, L, c_{g}(1), \mm, g(\mm))\in (0,1)$, which is independent of $\bar a$ and is a non-increasing function of $\mm$, such that if either
\eqn{limitatio2}
$$
\begin{cases}
\ \displaystyle \frac{|B\cap \{D_s v< \lambda /2\}|}{|B|}\leq \sigma\ \quad   \mbox{or} \ \quad 
\frac{|B\cap \{D_s v> -\lambda/2\}|}{|B|}\leq \sigma\\
 \  \quad =:\mbox{condition $\textnormal{\texttt{ex}}(B, \lambda, s)$}
 \end{cases}
$$
holds for some $s \in \{1, \ldots, n\}$, then 
\eqn{dasotto}
$$\snr{Dv} \geq \snr{D_s v}\geq \lambda /4 \quad \mbox{holds a.e.\,in $B/2$}\,.$$ 
In this case 
\eqn{decadi}
$$
\textnormal{\texttt{osc}}_{\tau B}\, Dv \leq c_{1} \tau^{\beta_1} \lambda \qquad \mbox{holds for every $\tau \in (0,1)$,}
$$
where $c_{1}\geq 1$ and $\beta_{1} \in (0,1)$ both depend on $n,\mu, q, \nu, L,\mm, g(\mm)$, but are otherwise independent of $\bar a$. The constants $c_{1}$ and $\beta_1$ are non-decreasing and non-increasing functions of $\mm$, respectively. On the other hand there exists $\tilde \eta\equiv \tilde \eta (n,\mu, q,\nu, L, \mm, g(\mm)) \in (1/2,1)$ such that, if $\textnormal{\texttt{ex}}(B, \lambda, s)$ fails for every $s \in \{1, \ldots, n\}$, then 
\eqn{decadi2}
$$\|Dv\|\leq \tilde \eta \lambda \quad \mbox{holds a.e.\,in $B/2$}\,.$$ The constant $\tilde \eta$ is a non-decreasing function of $\mm$. 
\end{lemma}
\begin{proof} For the proof of the first assertion we use and extend some of the arguments in \cite[Proposition 3.7]{KuM}; see also \cite{dibe} for the case of the classical $p$-Laplacian operator. We can of course confine ourselves to the case the first inequality in \rif{limitatio2} occurs (for a $\sigma$ to be determined in the course of the proof), the other being similar. By scaling, as in \rif{notational0}, we can assume that $B\equiv \mathcal B_1$. Since conditions \rif{regecor.1} are satisfied by $H_{\textnormal{i}}(\cdot)$, it is standard to prove that $v\in W^{2,2}_{\loc}(\mathcal B_1)$. We can then differentiate the equation $\diver\, \partial_z H_{\textnormal{i}} (Dv)=0$ in the $x_s$ direction, thereby obtaining $\diver\,  (\partial_{zz} H_{\textnormal{i}} (Dv)DD_sv)=0$, that is an analogue of \rif{1.fz}. Recalling that $v \in W^{1,\infty}(\mathcal B_1)\cap W^{2,2}_{\loc}(\mathcal B_1)$ (as in \rif{0.fz} and by \rif{limitatio}), we then use as test function $(D_s v-\kappa)_{-}\phi^2$, where $\phi \in C^{\infty}_{0}(\mathcal B_1)$ is non-negative and $0 \leq \kk \leq \lambda$; this is still possible by \rif{regecor.1}). Integrating by parts the resulting equation we find 
\begin{flalign}
\notag  & \int_{\mathcal B_1} \langle D_s(\partial_z H_{\textnormal{i}}(Dv)), D (D_s v-\kappa)_{-}\rangle \phi^2\dx  \\ &\notag  \qquad = \int_{\mathcal B_1} 
\langle \partial_z H_{\textnormal{i}}(Dv), D_s D \phi^2\rangle (D_s v-\kappa)_{-} \dx  \\
&  \qquad \quad - 2 \int_{\mathcal B_1\cap \{D_s v<\kk\}}  
\langle   \partial_z H_{\textnormal{i}}(Dv), D\phi\rangle D_{ss}v  \phi \dx\,.
\label{fromthis}
\end{flalign} 
We define the functions
\vspace{.5mm}
\eqn{ragnar2}
$$
\begin{cases}
G_0(t):= \tg (t) +\bar a (t^2 +\omega^2)^{(q-1)/2} \\
G(t):= [\tg (t)]^2(t^2+1)^{\mu/2} +\bar a (t^2 +\omega^2)^{q/2} \\
\lambda_{\bar a}(t)= 
(t^2+1)^{-\mu/2}+ (q-1)\bar a(t^2+\omega^2)^{(q-2)/2}\\
\end{cases}
$$ for $t\geq 0$, so that 
\eqn{holdy}
$$G_0(\lambda)\lambda +\lambda_{\bar a}(\lambda)\lambda^2 + [G_0(\lambda) ]^2/\lambda_{\bar a}(\lambda)\leq cG(\lambda)$$
holds for $c\equiv c (q)$. By also using \rif{rege.2i}, \rif{assi33} and \rif{limitatio}, and the monotonicity features of $ \lambda_{\bar a}(\cdot)$ and $G_0(\cdot)$ in \rif{fromthis}, we easily find 
\begin{flalign*}
\notag & \lambda_{\bar a}(\lambda)\int_{\mathcal B_1} |D(D_s v-\kappa)_{-}|^2\phi^2\dx\\
&\qquad  \leq c G_0(\lambda)  \int_{\mathcal B_1}(\phi |D^2 \phi|+|D\phi|^2)(D_s v-\kappa)_{-} \dx\\
& \qquad \quad  + 
c G_0(\lambda)  \int_{\mathcal B_1} \phi |D\phi||D(D_s v-\kappa)_{-} |\dx\,.
\end{flalign*}
By then using Young's inequality and \rif{holdy}, and yet recalling that $(D_s v-\kappa)_{-}\leq 2\lambda$, we find 
\begin{flalign}
\notag & \lambda_{\bar a}(\lambda)\|D(\phi(D_s v-\kappa)_{-})\|_{L^2(\mathcal B_1)}^2\\
& \qquad \leq cG(\lambda) |\mathcal B_1\cap \{D_s v<\kk\}\cap \{\phi>0\}| (\|D\phi\|_{L^{\infty}}^2+\|D^2\phi\|_{L^{\infty}})\,.\label{limitatio3}
\end{flalign}
For integers $m\geq 0$ we choose levels $\{\kk_m=\lambda(1+1/2^m)/4\}$, radii $\{\rr_m=1/2+1/2^{m+1}\}$ and cut-off functions $\phi_m\in C^{\infty}_0(\mathcal B_{\rr_{m}})$, with $\phi_m\equiv 1$ on $\mathcal B_{\rr_{m+1}}$, $\|D\phi_m\|_{L^\infty}^2+\|D^2\phi_m\|_{L^\infty} \lesssim 4^m$. With 
\eqn{disA}
$$A_m:=\{D_sv< \kk_m\}\cap B_{\rr_m}$$ estimate \rif{limitatio3} becomes $$ \lambda_{\bar a}(\lambda)\|D(\phi_m(D_s v-\kappa_{m})_{-})\|_{L^2(\mathcal B_1)}^2 \leq c 4^mG(\lambda)|A_m|\,.$$ We then find, via Sobolev embedding 
\begin{flalign*}
  \lambda_{\bar a}(\lambda) (\kk_{m}-\kk_{m+1})^2|A_{m+1}| & \leq  \lambda_{\bar a}(\lambda)\|\phi_m(D_s v-\kappa_m)_{-}\|_{L^2(\mathcal B_1)}^2\\ & \leq c  \lambda_{\bar a}(\lambda)\|D(\phi_m(D_s v-\kappa_m)_{-})\|_{L^2(\mathcal B_1)}^2|A_m|^{1/n}\\ & \leq c 4^mG(\lambda)|A_m|^{1+1/n}\,.
\end{flalign*}
Observing that $\kk_{m}-\kk_{m+1}\approx 2^{-m}\lambda$ and that \rif{limitatio} implies $\lambda^2\approx \lambda^2+\omega^2$, the above display gives 
\begin{eqnarray}
\notag |A_{m+1}|  & \leq & \frac{c 2^{4m} G(\lambda)}{ \lambda_{\bar a}(\lambda)\lambda^2} |A_m|^{1+1/n}\\
&\notag  \stackrel{\rif{ragnar},\rif{ragnar2}}{\leq}  &c \left[ (\lambda^2+1)^{\mu-1}[g(\lambda)]^2+1\right] 2^{4m} |A_m|^{1+1/n}\\
&\stackrel{\rif{limitatio}}{\leq} &c \left[ \mm^{2(\mu-1)}[g(\mm)]^2+1\right]2^{4m} |A_m|^{1+1/n} \label{limitatio5}
\end{eqnarray}
with $c\equiv c(n,\mu, q, \nu, L,c_{g}(1))$ which is independent of $\bar a$. 
Inequality \rif{limitatio5} allows to perform a standard geometric iteration (see for instance \cite[Lemma 7.1]{giu}) leading to $ D_s v\geq \lambda /4$ once the first inequality in \rif{limitatio2} is verified with $\sigma\equiv \sigma (n,\mu, q, \nu, L, c_{g}(1),\mm, g(\mm))$ small enough (that corresponds to require that $|A_0|$ is small, see \rif{disA}). This allows to prove \rif{decadi} by considering
\eqn{differenza}
$$-\diver \, (\texttt{a}(x)DD_t v)=0 \,,\quad \texttt{a}(x):= [ \lambda_{\bar a}(\lambda) ]^{-1}\partial_{zz} H_{\textnormal{i}}(Dv(x))$$
this time for every $t \in \{1, \ldots, n\}$. Recalling \rif{rege.2i}-\rif{ellrati} it follows that $\texttt{a}(\cdot)$ satisfies 
\eqn{elli-lim}
$$
|\xi|^2 \stackleq{limitatio} c \langle \texttt{a}(x)\xi, \xi \rangle\,, \quad \snr{\texttt{a}(x)}\stackleq{dasotto} c  [\mm^{\mu-1}g(\mm)+1]  
$$
a.e. in $x\in \mathcal B_{1/2}$ and for every $\xi \in \er^n$, where $c \equiv c (n,\mu, q, \nu, L)$. Therefore a standard application of De Giorgi-Nash-Moser theory to $D_tv$  and \rif{limitatio} imply the validity of \rif{decadi} (note that for this we obtain \rif{decadi} for $\tau \leq 1/8$, and then the case $1/8< \tau \leq 1$ follows trivially by \rif{limitatio}). We now turn to \rif{decadi2}; the proof is a variant of the one in \cite[Proposition 3.11]{KuM} apart from the parabolic case considered there. The only remark we need to do is that the analogue of equation in \cite[(3.53)]{KuM} is here given by $\diver \, (\texttt{b}(x)Dz)=0$, where  $\texttt{b}(x):= [\lambda_{\bar a}(\lambda) ]^{-1} \partial_{zz}H_{\textnormal{i}}(Dv(x))=\texttt{a}(x)$ if $x$ belongs to the support of $\tilde v := (D_sv-\lambda/2)_{+}$, and $\texttt{b}(x)\equiv \mathds{I}_d$ (the identity matrix) otherwise. Then $\tilde v$ turns out to be a weak subsolution. Moreover, as $|Dv| \geq \lambda/2$ on the support of $\tilde v$, it follows that $\texttt{b}(x)$ also satisfies \rif{elli-lim}. We can now use the arguments in \cite[pp. 784-786]{KuM} to conclude with \rif{decadi2}. Finally, as for the monotone dependence on $\mm$ of the constants $c_{1}, \beta_{1}, \tilde \eta, \sigma$, this is classical, and it is a consequence of the fact that all the estimates above feature constants, usually denoted by $c$, that are non-decreasing functions of $\mm$ and $g(\mm)$; in turn, $g(\cdot)$ is a non-decreasing function too.  \end{proof}
\begin{lemma} \label{limlem2} Let $v$ be as in Lemma \ref{limlem1}, assume \trif{limitatio} and that $m \omega \geq \lambda$ holds for some integer $m\geq 1$. Then
\eqn{decadi22}
$$
\textnormal{\texttt{osc}}_{\tau B}\, Dv \leq c_{2} \tau^{\beta_2} \lambda \qquad \mbox{holds for every $\tau \in (0,1)$,}
$$
where $c_{2}\geq 1$ and $\beta_{2} \in (0,1)$ both depend on $n,\mu, q, \nu, L,\mm,g(\mm), m$, but are otherwise independent of $\bar a$. The constants $c_{2}$ and $\beta_2$ are non-decreasing and non-increasing functions of $\mm$ and $m$, respectively. \end{lemma}
\begin{proof} As for Lemma \ref{limlem1} we reduce to the case $B\equiv \mathcal B_1$ and  consider \rif{differenza}. The first inequality in \rif{elli-lim} follows as $\|Dv\|\leq \lambda$ a.e.; using $m \omega \geq \lambda$ and \rif{rege.2i}$_3$ we instead have  
$$\snr{\texttt{a}(x)}\leq  c  [\mm^{\mu}g(\mm)+m^{2-q}] \quad 
\mbox{a.e.\,in $x\in \mathcal B_1$}\,,$$ 
where $c \equiv c (n,\mu, q, \nu, L)$.
Now \rif{decadi22} follows from standard De Giorgi-Nash-Moser theory. 
\end{proof}
\vspace{1mm}
\begin{proof}[of Proposition \ref{tecnica}] We take $\sigma, \tilde{\eta}\equiv \sigma, \tilde{\eta} (n,\mu, q, \nu, L, \mm, g(\mm))\in (0,1)$ from Lemma \ref{limlem1} and determine an integer $m\equiv m (n,\mu, q, \nu, L, \mm, g(\mm)) \geq 1$ such that $\eta:= \tilde \eta + 1/m<1$; notice that $m$ can be determined as a non-decreasing function of $\mm$. All in all, once $\mathfrak{m}$ is fixed so is $m$ and therefore $\eta$. We set $\lambda_0:= \sup_{B}\, \|Dv\|+\omega$, $\lambda_k:= \eta^{k} \lambda_0$, $B_k:=B/2^k$ for every $k \in \en$, and define $\bar k \in \en$ as the smallest non-negative integer for which $m\omega> \lambda_{\bar k}$. This implies that 
\eqn{decadi2h}
$$
\omega \leq \frac{\lambda_{k}}{m}, \quad \mbox{for \  $0\leq k < \bar k$}\,.
$$
Note that \rif{decadi2h} is empty when $\bar k =0$. We further define 
$$A:= \{k \in \en \, \colon 0\leq k < \bar k \ \mbox{and}  \ \mbox{$\textnormal{\texttt{ex}}(B_k, \lambda_k, s)$ holds for some $s\in \{1, \ldots, n\}$}\}$$ 
and $\tilde k$ as $\tilde k := \min A$ if $A\not=\emptyset$, and $\tilde k=\bar k$ if $A=\emptyset$. 
Using \rif{decadi2} and \rif{decadi2h} iteratively, we find
\eqn{decadi3}
$$
\sup_{B_k}\|Dv\|+\omega \leq \lambda_k= \eta^k \lambda_0 \quad    \mbox{for every $0 \leq k \leq  \tilde k$} \,.
 $$
On the other hand, by the very definition of $\tilde k$ and Lemmas \ref{limlem1}-\ref{limlem2} (the latter is only needed when $\tilde k=\bar k$ and therefore when $A=\emptyset$), we find that
\eqn{decadi4}
$$ 
\textnormal{\texttt{osc}}_{\tau B_{\tilde k}}\, Dv \leq \max\{c_1, c_2\} \tau^{\min\{\beta_1, \beta_2\}} \lambda_{\tilde k} \qquad \mbox{holds for \ $0< \tau < 1$}\,.
$$
At this stage \rif{diminutio} follows via a standard interpolation argument combining \rif{decadi3}-\rif{decadi4}. 
\end{proof}
\begin{remark}\label{mare}
As a consequence of the local gradient boundedness obtained in Section \ref{conv}, the arguments developed for Proposition \ref{tecnica} do not require any upper bound on $\mu$. Moreover, the specific structure in \rif{lareferenza} is not indispensable and the methods used here can be used as a starting point to treat general nonautonomous functionals with $(p,q)$-growth of the type considered in \cite{M2}.
\end{remark}
\subsection{Improved H\"older exponent when $\bb>0$}\label{ultimina}
The $ C^{1, \tilde \alpha/2}_{\loc}(\Omega)$-regularity proof follows almost verbatim the one in \cite[Section 10]{piovra} since now we already know that $Du$ is locally H\"older continuous in $\Omega$ with exponent $\beta$. We confine ourselves to give a few remarks on the main modifications and to facilitate the adaptation we use the same notation introduced in \cite{piovra}. We set $M:=\|Du\|_{L^\infty(\Omega)}+[Du]_{0, \beta;\Omega}+1$ (we can assume this is finite since our result is local). This time we take $A_r(z):= \partial_z H_{r^{\tia},\textnormal{i}}(z)
= \partial_z F(x_{\rm c}, z) +
qa_{r^{\tia}, \textnormal{i}}(B_{r})[\ell_{\bb}(z)]^{q-2}z
$ and recall that $|\partial_z H_{r^{\tia},\textnormal{i}}(z)-\partial_z H(x, z)|\leq cr^{\tia}[\ell_{1}(z)]^{q-1}$ holds whenever $x\in B_{r}$. With $v \in u+W^{1,q}_0(B_{r})$ to solve 
$\diver\, A_r(Dv)=0$ in $B_{r}$, as in 
\rif{28bbb} we gain
$
\nr{Dv}_{L^{\infty}(B_{r/2})}\le  c M^{2}
$ with $c\equiv c (\data)$. With these informations estimate $
\nr{Du-Dv}_{L^2(B_{r/2})}^2 \leq   cr^{n+\tia}$ follows as in \rif{holly2} (in turn as in \rif{10}), with $c\equiv c (\data,M)$. This is the analogue of \cite[(10.3)]{piovra}, but for the fact that here we find $\tia$ rather than $2\tia$ and the integral is supported in $B_{r/2}$ instead of $B_{r}$. Next, by defining the matrix $[\mathbb{A}(x)]_{ij}:= \partial_{z_j}A_{r}^i(Dv(x))$, to use \cite{piovra} we need to prove that there exists $\lambda \equiv \lambda (\data, \bb, M)>0$, independent of $r$, such that $\lambda \mathds{I}_{\rm d} \leq \mathbb{A}(x)\leq (1/\lambda)\mathds{I}_{\rm d}$ holds for a.e. $x \in B_{r/2}$. 
As a consequence of \rif{assif}$_2$ and of $|Dv|\lesssim M^2$ in $B_{r/2}$, we have
$$
\frac{|\xi|^2}{(M^4+1)^{\mu/2}} \leq c\langle \mathbb{A}(x)\xi, \xi\rangle\,,  \qquad
|\mathbb{A}(x)| \leq c [g(M^2)+ \bb^{q-2}]
$$
for a.e. $x\in B_{r/2}$ and $z, \xi\in \er^n$, where $c\equiv (\data, M)$. This is the crucial point where we use  After this, the proof follows exactly as \cite[(10.6)]{piovra} and the proof of the entire Theorem \ref{t3} is finally complete. 
\section{Theorem \ref{t1}, Corollary \ref{t2} and model examples \rif{verylinear}}\label{corsec}
For every integer $k\geq 0$, we consider the integrand 
$$
F(x,z) := \ccc(x) |z|L_{k+1}(|z|) +1,
$$
with $\ccc(\cdot)$ as in \rif{bound}$_2$ and $L_{k+1}(\cdot)$ as in \rif{verylinear}$_2$. 
Direct computations show that $F(\cdot)$ satisfies \rif{assif} with $g(t)\equiv L_{k+1}(t) +1$, with $\mu=1$ when $k=0$, and with any choice of $\mu>1$ when $k>0$; moreover, it also satisfies \rif{assi3}. The constants $\nu$ and $L$ depend on $\mu, k$ and $\Lambda$. Theorem \ref{t3} therefore applies to local minimizers of
$$
w \mapsto \int_{\Omega} [ \ccc(x) |Dw|L_{k+1}(|Dw|) +a(x)(Dw|^{2}+\bb^2)^{q/2}+1]\dx\,.
$$
On the other hand
the functional in the above display and
$$
w \mapsto \int_{\Omega} [ \ccc(x) |Dw|L_{k+1}(|Dw|) +a(x)(Dw|^{2}+\bb^2)^{q/2}]\dx
$$
share the same local minimizers, and therefore the regularity results stated in Theorem \ref{t3} hold for local minimizers of the last functional too. In this way the model case in \rif{verylinear} is covered. 
Theorem \ref{t1} and Corollary \ref{t2} then follow as special cases taking $k=0$ and $\bb=0$ and $\bb=1$, respectively. Note that, in the spirit of Corollary \ref{t2}, local minimizers of the functional 
$$
w \mapsto \int_{\Omega} \ccc(x) |Dw|L_{k+1}(|Dw|)\, dx
$$
are locally $C^{1,\alpha_0/2}$-regular in $\Omega$, for every $k\geq 0$, provided \rif{bound}$_2$ is assumed. 

\section{Vectorial cases and further generalizations}
When functionals as in \rif{modello} and \rif{generale} are considered in the vectorial case, i.e. minima are vector valued $u \colon \Omega \to \er^N$ and $N>1$, we can still obtain Lipschitz continuity results. In this situation it is unavoidable to impose a so-called Uhlenbeck structure \cite{uh}, that is \eqn{assi4}
$$
F(x, z)= \tilde F(x, |z|)\,,
$$
where $\tilde F\colon \Omega \times [0, \infty) \to [0, \infty)$ is a continuous function such that $t \mapsto \tilde F(x, t)$ is $C^2$-regular for every choice of $x\in \Omega$. This is obviously satisfied by the models in \rif{modello} and \rif{verylinear}. 
\begin{theorem}\label{t4}
Let $u\in W^{1,1}_{\loc}(\Omega;\er^N)$ be a local minimizer of the functional in \trif{generale} under assumptions \trif{bound}$_1$, \trif{assif}-\trif{0.1} and \trif{assi4}. There exists $\mum \in (1,2)$, depending only $n,q, \alpha, \alpha_0$, such that, if $1\leq \mu < \mum$, then 
\trif{30bis} 
holds as in Theorem \ref{t3}. \end{theorem}
The proof of Theorem \ref{t4} is essentially the same of the one given for Theorem \ref{t3}, once the content of Proposition \ref{fzfz} is available. This is in fact the only point where \rif{assi4} enters the game. Inspecting the proof of Proposition \ref{fzfz} shows that this works in the vectorial case provided \rif{vetti} holds. In turn, this inequality follows along the lines of the estimates in \cite[Lemma 5.6]{BM}. Notice that, without an additional structure assumption as in \rif{assi4}, Theorem \ref{t4} cannot hold and counterexamples to Lipschitz regularity emerge already when considering uniformly elliptic systems \cite{SY}. We also remark that we expect gradient H\"older continuity in the vectorial case as well; the proof must be different from the one given in Section \ref{latecnica}, and based on linearization methods as those originally introduced in \cite{uh} and developed in large parts of the subsequent literature. Another direct generalization, this time in the scalar case, occurs for functionals of the type
\eqn{tipino}
$$
 w\mapsto \mathcal{N}(w,\Omega):=\int_{\Omega}\left[F(x,Dw)+a(x)G( Dw)\right] \dx\,,
$$
where $F(\cdot)$ is as Theorem \ref{t3} and $G\colon \er^n\to [0, \infty)$ belongs to $C^1(\er^n)\cap C^2(\er^n\setminus\{0\})$ and satisfies 
$$
\begin{cases}
\ \nu(|z|^2+\bb^2)^{q/2} \leq G(z) \leq L(|z|^2+\bb^2)^{q/2}\\
\ \displaystyle \nu (|z|^2+\bb^2)^{(q-2)/2}|\xi|^2  \leq \langle \partial_{zz}G(z)\xi,\xi\rangle\\
\ |\partial_{zz}G(z)| \leq L(|z|^2+\bb^2)^{(q-2)/2}
\end{cases}
$$
for every $z, \xi \in \er^n$, $|z|\not= 0 $. 
Theorem \ref{t3} continuous to hold in this last case, with essentially the same proof. The only difference worth pointing out is the new shape of the integrands $H(x, z):=F(x, z)+a(x)G(z)$ in \rif{hhh0} and $H_{\omega,\sigma}(\cdot)$ (and the related minimal one $H_{\omega,\sigma,\textnormal{i}}(\cdot)$) in \rif{defiH}, accordingly defined as 
$$
\begin{cases}
\, H_{\omega,\sigma}(x,z):= F(x,z)+a_{\sigma}(x)G_{\omega}(z)\\
\, \displaystyle G_{\omega}(z):=\int_{\mathcal B_1}G(z+\omega \lambda)\phi(\lambda)\, {\rm d}\lambda\,.
\end{cases}
$$
Here $\{\phi_{\varepsilon}\}\subset C^{\infty}(\mathbb{R}^n)$, denotes a family of radially symmetric mollifiers, defined as $\phi_{\varepsilon}(x):=\phi(x/\varepsilon)/\varepsilon^n$, where $\phi\in C^{\infty}_{c}(\BB)$, $ \nr{\phi}_{L^{1}(\mathbb{R}^n)}=1$, $\mathcal B_{3/4} \subset \supp\,  \phi$.  The definition of $\eh_{\omega,\sigma,\textnormal{i}}(\cdot)$, instead, remains the same. Following for instance \cite[Section 4.5]{dm} and \cite[Section 6.7]{piovra}
it is possible to show that the newly defined integrands $H_{\omega,\sigma}(\cdot)$ still have the properties described in Section \ref{const}. In particular \rif{rege.2} and \rif{regecor} still hold for a suitable constant $c$ with the same dependence upon the various parameters described there. Finally, when again considering the vectorial case for functionals as in \rif{tipino} we can still obtain local Lipschitz regularity of minima provided \rif{assi4} holds together with $G(z)\equiv \tilde G(|z|)$ and $\tilde G(\cdot)$ is $C^2$-regular outside the origin (see for instance \cite{BM, ciccio} for precise assumptions). 
\vspace{6mm}

{\em Acknowledgments}. The first author is supported by INdAM-GNAMPA via the project ``Problemi non locali: teoria cinetica e non uniforme ellitticità", and by the University of Parma via the project ``Local vs nonlocal: mixed type operators and nonuniform ellipticity". Both the authors are grateful to the referees for the careful reading of the original version of the manuscript and for the many suggestions and comments that eventually led to a better presentation. 
\vspace{5mm}

{\bf Funding} Open access funding provided by Universit\`a degli Studi di Parma within the CRUI-CARE Agreement.

\vspace{5mm}

{\bf Data Availibility} No data is attached to this paper.

\vspace{5mm}

{\bf Declarations}

\vspace{5mm}

{\bf Conflict of interests and data.} The authors declare to have no conflict of interests. No data are attached to this paper. 

\vspace{5mm}

{\bf Open Access} This article is licensed under a Creative Commons Attribution 4.0 International License, which permits use, sharing, adaptation, distribution and reproduction in any medium or format, as long as you give appropriate credit to the original author(s) and the source, provide a link to the Creative Commons licence, and indicate if changes were made. The images or other third party material in this article are included in the article’s Creative Commons licence, unless indicated otherwise in a credit line to the material. If material is not included in the article’s Creative Commons licence and your intended use is not permitted by statutory regulation or exceeds the permitted use, you will need to obtain permission directly from the copyright holder. To view a copy of this licence, visit \href{http://creativecommons.org/ licenses/by/4.0/}{creativecommons.org/ licenses/by/4.0/}.

\address{
Dipartimento SMFI\\
Universit\`a di Parma\\
Parco Area delle Scienze 53/a, I-43124, Parma, Italy \\
\email{cristiana.defilippis@unipr.it}
\and
Dipartimento SMFI\\
Universit\`a di Parma\\
Parco Area delle Scienze 53/a, I-43124, Parma, Italy \\
email:{giuseppe.mingione@unipr.it}}

\end{document}